\documentclass[11pt,leqno]{amsart}

\usepackage{amsmath, amssymb, amscd, mathrsfs}

\newtheorem{theorem}{Theorem}[section]
\newtheorem{proposition}[theorem]{Proposition}
\newtheorem{lemma}[theorem]{Lemma}
\newtheorem{sublemma}[theorem]{Sublemma}
\newtheorem{corollary}[theorem]{Corollary}

\theoremstyle{remark}
\newtheorem{remark}[theorem]{Remark}

\theoremstyle{definition}
\newtheorem{definition}[theorem]{Definition}

\newcommand{\integer}{{\mathbb{Z}}}
\newcommand{\real}{{\mathbb{R}}}

\newcommand{\zero}{0}

\newcommand{\supp}{{\mathrm{supp}\,}}
\newcommand{\wt}{w^{(\lambda)}}
\newcommand{\Fourier}{\mathbb{F}} 

\newcommand{\cB}{\mathscr{B}}  
\newcommand{\bB}{\mathbf{B}}
\newcommand{\B}{B}
\newcommand{\cL}{\mathcal{L}}  
\newcommand{\bL}{\mathbf{L}}
\newcommand{\cH}{\mathcal{H}}
\newcommand{\cC}{\mathscr{C}}
\newcommand{\cM}{\mathcal{M}}  
\newcommand{\cN}{\mathcal{N}}
\newcommand{\cR}{\mathcal{R}}
\newcommand{\cone}{\mathbf{C}}
\newcommand{\bu}{\mathbf{u}}
\newcommand{\bk}{\mathbf{k}}
\newcommand{\dt}{d}
\newcommand{\disk}{\mathbb{D}}
\newcommand{\hG}{{\hat{G}}}
\newcommand{\tB}{\widetilde{\mathbf{B}}}
\newcommand{\cP}{\mathcal{P}}
\newcommand{\cQ}{\mathcal{Q}}

\newcommand{\unv}{{w}}
\newcommand{\Ze}{Z}
\newcommand{\talpha}{\widetilde{\alpha}}

\begin{document}

\title[Quasi-compactness of transfer operators]
{Quasi-compactness of transfer operators\\ for contact Anosov flows
}
\author[M. Tsujii]{Masato TSUJII}
\address{Department of Mathematics\\
Kyushu University\\
Hakozaki\\
  Fukuoka\\
  812-8581\\JAPAN}
\email{tsujii@math.kyushu-u.ac.jp}
\date{\today}
\thanks{This work is partly supported by KAKENHI (B) 18340044}

\begin{abstract}
For any $C^r$ contact Anosov flow with $r\ge 3$, we construct a scale of Hilbert spaces, which are embedded in the space of distributions on the phase space and contain all the $C^r$ functions, such that the one-parameter family of the transfer operators for the flow extend to them boundedly and that the extensions are quasi-compact. Further we give explicit bounds on the essential spectral radii of those extensions in terms of the differentiability~$r$ and the hyperbolicity exponents of the flow. 
\end{abstract}

\maketitle


\section{Introduction}
\subsection{Main result}
Geodesic flows on closed Riemannian manifolds with negative sectional curvature are a typical class of flows that exhibit chaotic behavior of orbits and  have been studied extensively since the works of Hopf\cite{Hopf39} and Anosov\cite{Anosov69} for this reason. 
Ergodicity and mixing, which characterize chaotic dynamical systems qualitatively, are  established for those flows already in early stage of study\cite{Hopf39, Anosov69}. However, quantitative estimates on the rate of mixing were obtained only recently in late 90's, while there had been 
some precise results  in the case of constant curvature by means of representation theory\cite{ColletET84, Moore84, Pollicott92,Ratner87}.
This is quite in contrast to the case of  Anosov diffeomorphisms for which  exponential decay of correlations had been established  already  in 70's\cite{Bowen75}. The difficulty in the case of geodesic flows (or hyperbolic flows, more generally) is in brief that there is no exponential expansion nor contraction in the flow direction. The mechanism behind mixing in hyperbolic flows is  different from and in fact subtler than that in hyperbolic discrete  dynamical systems.

In 1998, Chernov\cite{Chernov98} made a breakthrough by showing that 
the rate of mixing is stretched exponential at slowest for 3-dimensional Anosov flows satisfying the uniform non-integrability condition and, in particular, for all geodesic flows on closed surfaces with negative variable curvature. 
\hbox{Chernov} also conjectured in \cite{Chernov98} that the rate should be exponential. 
Shortly, this conjecture is proved affirmatively by Dolgopyat\cite{Dolgopyat98}. Dolgopyat analyzed the perturbed transfer operators closely and gave a necessary estimate on the Laplace transforms of the correlations. Dolgopyat's method has been extended and applied to many situations to get exponential or rapid decay of correlations. (\cite{Anantharaman00, BaladiVallee05,  Dolgopyat98b, Dolgopyat02, FMNT05, MelbourneTorok02,  Pollicott99, PollicottSharp01,    Stoyanov01, Stoyanov07})

More recently, Liverani\cite{Liverani04} established exponential decay of correlations for $C^4$ contact Anosov flows and, in particular,  for $C^4$ geodesic flows on closed Riemannian manifolds with negative  curvature in arbitrary dimension. He combined Dolgopyat's method with his method of using Banach spaces of distributions developped in his previous paper\cite{Blank02} coauthored with M.~Blank and G.~Keller. A remarkable feature of the argument in \cite{Liverani04}  is that it is free from Markov partitions,  which was a convenient artifact used in many works including \cite{Chernov98} and \cite{Dolgopyat98} and was an obstacle in making use of the smoothness of the flow. 

In this paper, we proceed  the argument further along the line of study described above, providing  a clearer picture in terms of spectral properties of the associated transfer operators: For any $C^r$ contact Anosov flow with \hbox{$r\ge 3$}, we construct a scale of Hilbert spaces,
which are embedded in the space of distributions on the phase space and contain all  $C^r$ functions,
so that the one-parameter family of 
the transfer operators for the flow extend naturally to bounded operators on them and  that the extensions are {\em quasi-compact}. Moreover we give  an explicit upper bound on the essential spectral radii of the extensions in terms of differentiability $r$ and the hyperbolicity exponents of the flow. 
This implies  not only exponential decay of correlations but also a precise asymptotic estimate on the decay rate. (See 
Corollary~\ref{cor:corr}.) Our argument is also free from Markov partitions.

To state the main  result more precisely, we introduce some definitions. 
Let $d\ge 1$ and $r\ge 3$ be integers. 
Let $M$ be an orientable  $(2d+1)$-dimensional closed $C^r$  manifold and
$\alpha$ a $C^r$ contact form on $M$. By definition, $\alpha$ is a $1$-form such that $\omega:=\alpha\wedge  (d\alpha)^{d}$ is a volume form on $M$. 
Let  $F^t:M\to M$ be a $C^r$ Anosov flow preserving the contact  form $\alpha$. Such a flow is called a $C^r$ contact Anosov flow. 
 Geodesic flows on closed Riemannian manifolds with negative sectional  curvature are types of contact Anosov flows, 
when we regard them as flows on the unit cotangent bundles  equipped with the canonical contact forms. 
  
Let $v$ be the vector field that generates the flow $F^t$. By the definition of Anosov flow, there exists an invariant splitting of the tangent bundle, $
TM=E^c\oplus E^s \oplus E^u$, 
such that $E^c$ is the one-dimensional subbundle spanned by the vector field $v$ and that there exist $\lambda_0>0$ and $C>0$ such that\footnote{For convenience in the later argument, we consider the exponential function with base $2$ (instead of $e$), though this is of course not essential.}
\[
\|DF_z^t|_{E^s}\|\le C\cdot  2^{-\lambda_0 t}\quad\text{and}\quad
\|DF_z^{-t}|_{E^u}\|\le C\cdot 2^{-\lambda_0 t}\quad \text{$\forall t\ge 0$, $\forall z\in M$. }
\]
Since the flow $F^t$ preserves the contact form $\alpha$, the subspaces $E^s$ and $E^u$ should be contained in the null space of $\alpha$.
This implies that the subspace  $E^s\oplus E^u$ coincides with the null space of $\alpha$ and hence that $\alpha(v)\neq 0$ at any point. 
In what follows, we suppose $\alpha(v)\equiv 1$ by replacing  $\alpha$ by $\alpha/\alpha(v)$. 
Since the 2-form $d\alpha$ is preserved by the flow $F^t$ and gives a symplectic form on the null space of $\alpha$, we see that 
$\dim E^s=\dim E^u=d$ and also that $E^0$ coincide with the null space of $d\alpha$ at each point. 
Notice that the vector field $v$ is characterized by the conditions $\alpha(v)=1$ and $v\in \mathrm{null}\, d\alpha$. Such a vector field is called the Reeb vector field of $\alpha$.

Let $\Lambda_0>0$ be another constant for the flow $F^t$  such that, for some $C>1$,
\[
|\det (DF_z^{-t}|_{E^u})|\le C\cdot  2^{-\Lambda_0 t}\qquad \forall t\ge 0, \forall z\in M.
\]
Obviously we may take $\Lambda_0$ so that $\Lambda_0\ge d\lambda_0$. 
For the flow $F^t$, we associate the one-parameter family of transfer operators $
\cL^t: C^r(M)\to C^r(M)$ defined by 
$\cL^t(u)(z)= u\circ F^t(z)$.
For a real number $s$ with $|s|\le r$, let $W^s(M)$ be  the Sobolev space\footnote{See Remark \ref{rm:Sobolev} for the definition. For $s\ge 0$, $W^s(M)$ contains $C^s(M)$, and  $W^{-s}(M)$  is contained in the space of distributions of order $s$.} of order $s$ on $M$. 
Our main result is the following spectral property of $\cL^t$.
\begin{theorem}\label{mainth}
For each $0<\beta<(r-1)/2$, there exists a Hilbert space $\B^\beta$, which is contained in $W^s(M)$ for $s< -\beta$ and contains $W^s(M)$ for $s> \beta$, such that the transfer operator $\cL^t$ for sufficiently large $t$ extends to a bounded operator on $\B^\beta$ and the essential spectral radius of the extension $\cL^t:\B^\beta\to \B^\beta$ is bounded by $
\max\{ 2^{-\Lambda_0 t/2}, 2^{-\beta \lambda_0 t}\}<1$. 
\end{theorem}

In the case where the flow $F^t$ is $C^\infty$, we may choose a large $\beta$ in the theorem above so that $2^{-\beta \lambda_0 t}<2^{-\Lambda_0 t/2}$ and hence that the essential spectral radius of $\cL^t:\B^\beta\to \B^\beta$ is bounded by 
$2^{-\Lambda_0 t/2}$.
It should be worth noting that this bound on the essential spectral radius is quite reasonable at least 
in the case of geodesic flows on closed surfaces with  constant negative curvature. 
In fact, if we suppose that the curvature is constantly $-1$, 
we may set $\Lambda_0=\lambda_0=-1/\log 2$, so that  the bound  equals $2^{-\Lambda_0 t/2}=e^{-t/2}$. 
From the famous result\cite{McKean72} of Selberg on his zeta function, we find that the dynamical Fredholm determinant of $\cL^t$ for such a flow has infinitely many zeros on the line $\Im(s)=-1/2$. 
Admitting the conjectural statement that  the zeros of  the dynamical Fredholm determinant of $\cL^t$  should coincide with the spectrum of the generator of $\cL^t$,  we expect that the essential spectral radius of  $\cL^t$ should not be smaller than $e^{-t/2}$ (for any reasonable choice of Banach spaces on which it acts). 

Since contact Anosov flows are mixing (or even Bernoulli\cite{Katok94}) with respect to the contact volume $\omega$, Theorem \ref{mainth}  implies not only exponential decay of correlations but also the following asymptotic estimate on correlations. (See \cite{Tsujii2008} for the detail of the deduction.)
\begin{corollary}\label{cor:corr}
For any $0<\alpha<\min\{\Lambda_0, (r-1)\lambda_0\}/2$, there exists finitely many complex numbers $\eta_i$ with $-\alpha\le \Re(\eta_i)<0$ and integers $k_i\ge 0$ for $1\le i\le \ell$ such that, 
for any $\psi$ and $\varphi$ in $C^r(M)$,  we have the asymptotic estimate for the correlation 
\begin{align*}
\frac{1}{\omega(M)} \int &\psi\cdot \varphi\circ F^t \; d\omega\; -\;\frac{1}{\omega(M)} \int \psi \; d\omega\cdot \frac{1}{\omega(M)}  \int \varphi \; d\omega\\
&= \sum_{i=1}^\ell \sum_{j=0}^{k_i} C_{ij}(\varphi, \psi)\cdot  t^j 
2^{t\eta_i}+\mathcal{O}(2^{-\alpha t})
\end{align*}
as $t\to \infty$, where $C_{ij}(\varphi, \psi)$ are constants depending on $\psi$ and $\varphi$ bilinearly. 
\end{corollary}
Also  we can deduce from Theorem \ref{mainth} the central limit theorem and  the (generalized) local limit theorem for observables in $C^r(M)$ by a general abstract argument.  (See \cite{Iwata07}.)

The main point in Theorem \ref{mainth} is definitely the   construction of the Hilbert spaces 
$\B^\beta$. 
The original idea for the construction was as follow\footnote{This idea originated in the work \cite{Baladi05} of V.~Baladi, which treated  spectral properties of transfer operators for Anosov diffeomorphisms. See \cite{BaladiTsujii07, BaladiTsujii08} for later developments.}: Take an appropriate  positive-valued smooth function $p_\beta:T^*M \to \real$ on the cotangent bundle and define a norm on $C^\infty(M)$ by  $\|u\|_{\beta}:=\| p_\beta(D) u\|_{L^2}$, using the pseudodifferential operator $p_\beta(D)$ associated to $p_\beta$. Then define the Hilbert space $\B^\beta$ as the completion of $C^\infty(M)$ with respect to this norm.  Unfortunately we could not put this idea directly into a rigorous argument because of  some technical difficulties. 
(See Remark \ref{re:ori} for more detail.) Instead, we  use a modified version of the Littlewood-Paley decomposition for the construction of the Hilbert spaces $B^\beta$. 
This makes the argument in this paper somewhat long and involved. Still the argument in each step is fairly elementary and does not require any knowledge on pseudodifferential operators.

\subsection{Plan of the proof}
In the following sections, we proceed  to  the proof of the main theorem as follows.  
Section \ref{sec:darboux}, \ref{sec:localtrans} and \ref{sec:localG} are devoted to  preliminary arguments. 
In Section \ref{sec:darboux}, we set up a finite system of local charts on $M$ adapted to the contact structure $\alpha$ and the hyperbolic structure of the flow. In Section \ref{sec:localtrans}, we then reduce the main theorem to the corresponding claim (Theorem \ref{th:mainlocal}) about transfer operators on the local charts. 
In Section \ref{sec:localG}, we give a local geometric property of the diffeomorphisms between the  local charts induced by the time-$t$-maps of the flow. This property  is simple but  crucial for our argument. 

In Section \ref{sec:part} and \ref{sec:Hil}, we define  Hilbert spaces $\cB^\beta_\nu$ for real numbers $\beta$ and~$\nu$, which consist of distributions on the unit disk $\disk$ in the Euclidean space $E$ of dimension $2d+1$.
The Hilbert spaces $B^\beta$ in the main theorem is made up from copies of such Hilbert spaces on the local charts by using a partition of unity on $M$. In Section \ref{sec:part}, we construct a $C^\infty$ countable partition of unity $\{p_{\gamma}\}_{\gamma\in \Gamma}$ on the cotangent bundle $T^*_{\disk}E=\disk\times E^*$. 
Then, in Section~\ref{sec:Hil}, we give  a method of decomposing  a function $u$ on  $\disk$ into countably many smooth components~$u_\gamma$, $\gamma\in \Gamma$, by using the pseudodifferential operators with  symbol~$p_\gamma$. 
By definition, each component $u_\gamma$ is a "{\em wave packet}\/" which are localized  both in the real and  frequency  spaces. The Hilbert space $\cB^\beta_\nu$ will be defined as the completions of the space $C^\infty(\disk)$ of  $C^\infty$ functions on the unit disk $\disk$ with respect to a norm $\|\cdot \|_{\beta,\nu}$ that counts the $L^2$ norms of  the components $u_\gamma$ with some appropriate weight. 

Our basic strategy  is that we regard  a transfer operator $\cL$ acting on $\cB^\beta_\nu$  as an infinite matrix of  operators $\cL_{\gamma\gamma'}$, each of which concerns the transition from one component to another induced by $\cL$ and  deduce the required properties of $\cL$ from relatively simple estimates on each $\cL_{\gamma\gamma'}$.  In Section~\ref{sec:aux}, we  introduce some definitions in order to describe the argument along this strategy. And we
 find that each operator  $\cL_{\gamma\gamma'}$ is a tame integral operator with smooth rapidly decaying kernel. Further we  give simple estimates on the kernel of $\cL_{\gamma\gamma'}$, regarding it as an oscillatory integral.

Section  \ref{sec:pre}--\ref{sec:center} are the main body of the proof. 
In the proof, we divide the transfer operator $\cL$ on the local charts into three parts: the {\em compact}, {\em central} and  {\em hyperbolic part}. The   compact part  is the part that  concerns the components of functions with low frequencies. In Section  \ref{sec:pre}, we show  that the compact part is in fact a compact operator and therefore negligible  in our argument because the essential spectral radius of an operator does not change by perturbation by compact operators. 
The definitions of the  central  and hyperbolic part are more involved. 
Roughly, the central part is the part that concerns the components of functions which are localized along the central (or flow) direction in the frequency space, and the hyperbolic part is the remainder. 

In Section \ref{sec:body}--\ref{sec:tail2}, we deal with  the  hyperbolic part and estimate its operator norm. 
The argument in these sections makes use of hyperbolicity of the flow in the  directions transversal to the flow, and 
is partially similar to that  
in the author's previous papers \cite{BaladiTsujii07, BaladiTsujii08} coauthored with V.~Baladi, which treat hyperbolic diffeomorphisms. 
The estimate on the hyperbolic part leads to the term $2^{-\beta\lambda_0 t}$ in the main theorem. 

In Section \ref{sec:center}, we deal with  the central part, which is responsible for the difficulty in the case of hyperbolic flow that we noted in the beginning. 
The argument on the central part is in fact the main point of this paper and makes use of the  non-integrability of the contact form $\alpha$ essentially. 
The estimate on the central part leads to the term $2^{-\Lambda_0 t/2}$ in the main theorem. 

  \begin{remark} 
A prototype of the argument on the central part can be found in the author's previous paper \cite{Tsujii2008}, where a class of expanding semi-flows are considered as a simplified model of Anosov flows. 
\end{remark}

\noindent{\bf Acknowledgement.} The author would like to thank the referees of this paper for many valuable comments, which were very important 
in improving descriptions in the text.


\section{Darboux theorem for contact structure}
\label{sec:darboux}
In this section, we set up a finite system of coordinate charts on $M$ which is adapted to the contact structure $\alpha$ on $M$ and also to the hyperbolic structure of the flow $F^t$. 
Let $E$ be an Euclidean space of dimension $2d+1$, equipped with an orthonormal coordinate
\[
x=(x_0,x^+_1, \dots, x^+_d, x^-_1, \dots, x^-_d). 
\]
Let $E^*$ be the dual space of $E$, equipped with the dual coordinate
\[
\xi=(\xi_0, \xi^+_1, \dots, \xi^+_d, \xi^-_1, \dots, \xi^-_d),
\]
so that evaluation of $\xi\in E^*$ at $x\in E$ is given by
\[
\langle \xi, x \rangle = \xi_0 \cdot x_0+\xi^+_1\cdot x^+_1+\dots+\xi^+_d\cdot x^+_d+
\xi^-_1\cdot x^-_1+\dots+\xi^-_d\cdot x^-_d.
\]
For brevity, we write $x=(x_0, x^{+},x^{-})$ and $\xi=(\xi_0, \xi^+,\xi^-)$ for $x$ and $\xi$ as above, setting $
x^{\pm}=(x^\pm_1,\dots, x^\pm_d)$ and 
$\xi^{\pm}=(\xi^\pm_1, \dots, \xi^\pm_d)$ respectively. 
Let $
E=E_0\oplus E_+\oplus  E_- $ and $
E^*=E_0^*\oplus E_+^*\oplus  E_-^*$ be the corresponding orthogonal decomposition. For $\sigma\in \{0,+,-\}$, let $
\pi_\sigma:E\to E_\sigma$ and $\pi^*_\sigma:E^*\to E^*_\sigma$
be the orthogonal projections. Also we  set $
\pi_{+,-}=\pi_+\oplus \pi_-:E\to  E_+\oplus  E_-$ and define  
$\pi_{0,+}$, $\pi_{0,-}$, $\pi^*_{+,-}$, $\pi^*_{0,+}$ and $\pi^*_{0,-}$ analogously. 

The standard contact form on the Euclidean space $E$ is the 1-form
\[
\alpha_0=dx_0+x^{-} \cdot dx^{+}-x^{+} \cdot dx^{-}
\]
where $
x^{-} \cdot dx^{+}=\sum_{i=1}^d\;  x^-_i\cdot  d x^+_i$ and 
$x^{+} \cdot dx^{-}=\sum_{i=1}^d\;  x^+_i\cdot  d x^-_i$.
We will refer to $
v_0=\partial/\partial x_0$ as 
the standard vector field on $E$, which is nothing but the Reeb vector field of $\alpha_0$.  
A local chart $\kappa:U\to V\Subset E$ on an open subset $U\subset M$ is called a Darboux chart if $\kappa^*(\alpha_0)=\alpha$ on $U$. 
Darboux theorem for contact structure\cite[pp.168]{Aebischer} tells that there exists a system of Darboux charts on $M$. Below we choose a finite system of Darboux charts adapted to the hyperbolicity  of the flow.

Let $\cone_+$ and $\cone_-$ be the closed cones on $E$ defined by
\begin{align*}
\cone_+&=\{(x_0,x^{+},x^{-})\in E\mid \|x^{-}\|\le \|x^{+}\|/10\}
\intertext{and}
\cone_-&=\{(x_0,x^{+},x^{-})\in E\mid \|x^{+}\|\le \|x^{-}\|/10\}.
\end{align*}
\begin{definition}
For $\lambda>1$ and $\Lambda>1$, let $\cH(\lambda,\Lambda)$ be the set of  $C^r$ diffeomorphisms $
G:V'\to V:=G(V')$ on $E$ satisfying the  conditions
\begin{itemize}
\setlength{\itemsep}{4pt}
\item[(H0)] $V'$ and $V$ are open subsets of the unit disk $\disk\subset E$,
\item[(H1)]
$G^*(\alpha_0)=\alpha_0$ on $V'$, and $G_*(v_0)=v_0$ on $V$,
\item[(H2)]
$DG_z(E\setminus \cone_{+})\subset \cone_-$ and $(DG_z)^{-1}(E\setminus \cone_-)\subset \cone_+$ for any $z\in V'$,
\item[(H3)]
$\|\pi_{+,-}(DG_z(v))\|\ge 2^\lambda \|\pi_{+,-}(v)\|$ for any $ z\in V'$ and $ v\in E\setminus \cone_{+}$, 
\\
$\|\pi_{+,-}((DG_z)^{-1}(v))\|\ge 2^\lambda \|\pi_{+,-}(v)\|$ for any $ z\in V'$ and $ v\in E\setminus \cone_{-}$,
\item[(H4)]
$\det (DG_z|_{Y})\ge 2^\Lambda$ for any $(d+1)$-dim subspaces $Y\subset \cone_-$, and \\
$\det ((DG_z)^{-1}|_{Y'})\ge  2^\Lambda$ for any $(d+1)$-dim subspaces $Y'\subset \cone_+$,
\end{itemize}
where $\det(A|_Y)$ is the expansion factor of the linear map $A:Y\to A(Y)$ with respect to the standard volumes on $Y$ and $A(Y)$. Let $\mathcal{H}$ be the union of $\mathcal{H}(\lambda, \Lambda)$ for all $\lambda>0$ and $\Lambda>0$.
\end{definition}

The following  is a slight modification of the Darboux theorem. 
\begin{proposition}\label{th:Darboux}
There exists a finite system of Darboux charts on $M$,
\[
\kappa_a:U_a\to V_a:=\kappa_a(U_a)\subset \disk\subset E \quad \text{for }a\in A, 
\]
and a constant $c_0>0$ 
such that, if $t$ is sufficiently large and if 
\[
V(a,b;t):=\kappa_a(U_a\cap F^{-t}(U_b))\neq \emptyset\quad \text{
for some $a,b\in A$,}
\]
the induced diffeomorphism on the charts,
\[
F^t_{ab}:=\kappa_b\circ F^t\circ \kappa_a^{-1}
:V(a,b;t) \to F^t_{ab}(V(a,b;t))\subset V_b,
\]
 belongs to the class $
\cH(\lambda_0 t-c_0, \Lambda_0 t-c_0)$ defined above.
\end{proposition}
\begin{proof} 
By compactness of $M$, it is enough to show, for each  $z\in M$, that there exists a Darboux chart
$\kappa:U\to V$ on a neighborhood $U$ of~$z$ so that $\kappa(z)=\zero$, 
$D\kappa_{z}(E^s(z))=E_+$ and $D\kappa_{z}(E^u(z))=E_-$. 
By Darboux theorem, there exists a Darboux chart $\kappa':U'\to V'$ on a neighborhood $U'$ of $z$ so that  $\kappa'(z)=\zero$. 
For $E'_+:=D\kappa'_{z}(E^s(z))$ and $E'_-:=D\kappa'_{z}(E^u(z))$, we have $
E'_+\oplus E'_-=D\kappa'_z(\mathrm{null}(\alpha_0(0)))=E_+\oplus E_-$.
Since $d\alpha$ is preserved by the flow $F^t$, we see $d\alpha|_{E^s}=d\alpha|_{E^u}=0$ and therefore   $d\alpha_0|_{E'_+}=d\alpha_0|_{E'_-}=0$.  So we can find 
a linear map $L:E_+\oplus E_-\to E_+\oplus E_-$ which preserves the symplectic form $d\alpha_0(0)|_{E_+\oplus E_-}$ and satisfies
$L(E'_+)=E_+$ and $L(E'_-)=E_-$. Define  $L':E\to E$  by $L'(x_0,x^+,x^-)=(x_0, L(x^+,x^-))$. 
Then it is easy to check that  $L'$ preserves the contact form~$\alpha_0$ and that the composition 
$\kappa:=L'\circ \kappa'$
is a chart with the required properties.
\end{proof}
Henceforth we fix a finite system of Darboux charts $\kappa_a:U_a\to V_a$, $a\in A$, with the property in Proposition~\ref{th:Darboux}. 

\section{Transfer operators on local charts}\label{sec:localtrans}

In this section, we reduce Theorem \ref{mainth} to the corresponding claim about transfer operators on the local charts. To state the claim, we prepare some definitions. 
For an  open subset $V\subset E$, 
let $C^r(V)$ be the set of $C^r$~functions whose supports are contained in $V$, and 
let $\cC^{r}(V)$ be the subset of $g\in C^r(V)$ such that
the differential  $(v_0)^k g=\partial^k g/\partial x_0^k$   for arbitrarily large $k$ exists and belongs to  the class $C^r(V)$.  
We henceforth  fix a large positive integer $r_*\ge 20(r+1)$
and set
\[
\|g\|_*=\max_{0\le k\le r_*} \|\partial^k g /\partial x_0^k \|_{L^\infty}\quad \text{for $g\in \cC^{r}(V)$.}
\]
For a $C^r$ diffeomorphism $G:V'\to V$ in $\cH$ and a function $g\in \cC^r(V')$,  we consider the transfer operator
$\cL(G,g):C^r(V)\to C^r(V')$ 
defined by
\[
\cL(G,g) u(x)=
\begin{cases} g(x)\cdot u(G(x)),&\text{for $x\in V'$};\\
0, &\text{otherwise.}
\end{cases}
\]

The Sobolev space $W^s(\disk)$ on the unit disk $\disk\subset E$ is the completion of the space $C^\infty(\disk)$ with respect to the norm $
\|u\|_{W^s}=\| (1+|\xi|^2)^{s/2} \cdot \Fourier u(\xi)\|_{L^2}$,
where $\Fourier:L^2(E)\to L^2(E^*)$ is the Fourier transform.

\begin{remark} \label{rm:Sobolev}The Sobolev space $W^s(M)$ for $|s|\le r$ on $M$  is defined from copies of $W^s(\disk)$ on the local charts in an obvious manner using a  partition of unity. Clearly we have $C^r(M)\subset C^s(M)\subset W^s(M)$ for $0\le s\le r$. 
\end{remark}

We will construct Hilbert spaces $\cB^\beta_\nu$ 
for $\beta>0$ and $\nu\ge 2d+2$, which satisfy $W^s(\disk)\subset \cB^\beta_\nu\subset W^{-s}(\disk)$ for $s>\beta$ and prove the following claims:
\begin{theorem}\label{th:mainlocal}
There exist positive constants $\lambda_*$ and $\Lambda_*$ so that the   operator $\cL(G,g)$ for any $G:V'\to V$ in $\cH(\lambda_*,\Lambda_*)$ and  $g\in \cC^r(V')$ extends to a bounded operator
$\cL(G,g):\cB^\beta_\nu\to \cB^\beta_{\nu'}$ for  
any $0<\beta<(r-1)/2$ and  $\nu,\nu'\ge 2\beta+2d+2$. 
Further, for any $\epsilon>0$ and $0<\beta<(r-1)/2$, 
there exist constants  $\nu_*\ge 2\beta+2d+2$, $C_*>0$ and 
a family of norms $\|\cdot\|^{(\lambda)}$ on $\cB^\beta_{\nu_*}$ for $\lambda>0$, which are all equivalent to the standard norm on $\cB^\beta_{\nu_*}$, such that, 
if $G:V'\to V$ belongs to $\cH(\lambda,\Lambda)$ for $\lambda\ge \lambda_*$ and  $\Lambda\ge \Lambda_*$ with $\Lambda\ge d\lambda$ and if $g\in \cC^r(V')$,
there exists a compact operator $\mathcal{K}(G,g):\cB^\beta_{\nu_*}\to \cB^\beta_{\nu_*}$ such that the operator norm of $\cL(G,g)-\mathcal{K}(G,g):\cB^\beta_{\nu_*}\to \cB^\beta_{\nu_*}$ with respect to the norm $\|\cdot\|^{(\lambda)}$ is bounded by $
C_* \|g\|_{*} 2^{-(1-\epsilon)\min\{\Lambda/2,\beta \lambda\}}$. 
\end{theorem}

We show that  Theorem \ref{mainth} follows from Theorem \ref{th:mainlocal}. 
Take  $C^r$ functions
$\rho_a:V_a\to [0,1]$ and $\widetilde{\rho}_a:V_a\to [0,1]$ for $a\in A$
so that  the family 
$\{\rho_a\circ \kappa_a\}_{a\in A}$ is a $C^r$ partition of unity on $M$ and that 
 $\widetilde{\rho}_a\equiv 1$ on $\supp\rho_a$ and $\supp \widetilde{\rho}_a\Subset V_a$. 
We may (and do) suppose that  $\rho_a$ and $\widetilde{\rho}_a$ belong to the  class $\cC^r(V_a)$, applying an appropriate $C^\infty$ mollifier along the coordinate $x_0$ simultaneously. 
More precisely, if  either of $\rho_a$  does not belong to $\cC^r(V_a)$, we replace $\rho_a$ by
\[
\rho'_a(x_0,x^+,x^-)=\int (1/\epsilon)\cdot p(s/\epsilon ) \cdot \rho_a(x_0+s,x^+,x^-) ds
\]
where $p:\real\to \real$ is a positive-valued $C^\infty$ function with compact support such that $\int p(s) ds=1$ and $\epsilon>0$  a small real number. 
Taking sufficiently small $\epsilon>0$, we may suppose that the support of $\rho'_a$ is contained in $V_a$ and that  $\rho'_a$ belongs to $\cC^r(V_a)$. 
Since we have  
\[
\rho'_a\circ \kappa_a(x)=\int (1/\epsilon)\cdot p(s/\epsilon ) \cdot \rho_a\circ \kappa_a(F^s(x)) ds
\]
from the relation $v_0=(\kappa_a)_* v$, we see  that  the family $\{\rho'_a\circ \kappa_a\}_{a\in A}$ is also a $C^r$ partition of unity on $M$.
We may apply the same modification to $\widetilde{\rho}_a$ if either of 
$\widetilde{\rho}_a$ does not belong to $\cC^r(V_a)$.

For $a,b\in A$, we define the transfer operator $\cL_{ab}^t:C^r(V_b)\to C^r(V_a)$  by
\[
\cL_{ab}^t u(x)=
\begin{cases}
g_{ab}^t(x) \cdot u(F^t_{ab}(x)),&\text{if $x\in V(a,b;t)$;}\\
 0,&\text{otherwise}
\end{cases}
\]
where $g^t_{ab}(x)=\rho_a(x)\cdot \widetilde{\rho}_b(F^t_{ab}(x))$ belongs to $\cC^r(V_a)$. 
Then we consider  the matrix of operators
\[
\bL^t: \oplus_{a\in A} C^r(V_a) \to  \oplus_{a\in A} C^r(V_a), \quad 
\bL^t((u_a)_{a\in A})=\left(
\sum_{b\in A} \cL^t_{ab}(u_b)
\right)_{a\in A}.
\]
Let  $\iota:C^r(M)\to \oplus_{a\in A} C^r(V_a)$ be the injection defined by  
\[
\iota(u)=(\rho_a\cdot (u\circ \kappa_a^{-1}))_{a\in A}.
\]
By the definition, we have the commutative diagram
\[
\begin{CD}
 \oplus_{a\in A} C^r(V_a) @>>{\bL^t}>  \oplus_{a\in A} C^r(V_a)\\
@AA{\iota}A @AA{\iota}A\\
C^r(M) @>>{\cL^t}> C^r(M)
\end{CD}
\]

Let $B^\beta_\nu$ be the completion of $C^r(M)$ with respect to the pull-back of the product norm on $\oplus_{a\in A} \cB^\beta_\nu\supset \oplus_{a\in A} C^r(V_a)$ by the injection $\iota$, so that the injection $\iota$ extends to the isometric embedding $\iota:B^\beta_\nu\to \oplus_{a\in A} \cB^\beta_\nu$ and that 
$
W^s(M)\subset B^\beta_\nu \subset W^{-s}(M)$ for $s>\beta$. 

Let $c_0$ be the constant in Proposition~\ref{th:Darboux},  and $\lambda_*$ and  $\Lambda_*$  those in the former statement of Theorem \ref{th:mainlocal}. 
Take $t_0>0$ so large that $\lambda_0 t_0 -c_0\ge \lambda_*$ and $\Lambda_0 t_0-c_0\ge \Lambda_*$. 
Applying the former statement of Theorem \ref{th:mainlocal} to each $\cL_{ab}^t$, we see that the commutative diagram above extends to  
\[
\begin{CD}
\oplus_{a\in A} \cB^\beta_\nu @>>{\bL^t}> \oplus_{a\in A} \cB^\beta_{\nu'}\\
@AA{\iota}A @AA{\iota}A\\
\B^\beta_\nu @>>{\cL^t}> \B^\beta_{\nu'}
\end{CD}
\]
for any  $t\ge t_0$, provided that $0<\beta<(r-1)/2$ and $\nu,\nu'\ge 2\beta+2d+2$. 

Suppose that  $\epsilon>0$ and $0<\beta<(r-1)/2$ are given arbitrarily and let $\nu_*$, $C_*$ and $\|\cdot\|^{(\lambda)}$ be those in the latter statement of Theorem \ref{th:mainlocal}.
Recall that the essential spectral radius of an operator on a Banach space coincides with the infimum of the spectral radii of its purturbations by compact operators. (See \cite{Nussbaum70}.)
Hence, applying the latter statement of Theorem \ref{th:mainlocal} to each~$\cL_{ab}^t$, we see that  the essential spectral radius of $\bL^t:\oplus_{a\in A} \cB^\beta_{\nu_*}\to \oplus_{a\in A} \cB^\beta_{\nu_*}$ is bounded by
\[
C_* \cdot  \# A\cdot \left( \max_{a,b\in A} \|g^t_{ab}\|_*\right)\cdot 
2^{-(1-\epsilon)\min\{ (\Lambda_0 t-c_0)/2, \beta (\lambda_0 t-c_0)\}}
\]
and so is that of $\cL^t:\B^\beta_{\nu_*}\to \B^\beta_{\nu_*}$ from the commutative diagram above. 
Note that the term $\max_{a,b\in A} \|g^t_{ab}\|_*$ is bounded by a constant independent of $t$,  because $F^t_{ab}$ preserves the standard vector field~$v_0$. 
From  the multiplicative property of  essential spectral radius, the essential spectral radius of  $\cL^t:\B^\beta_{\nu_*}\to \B^\beta_{\nu_*}$ is bounded by $2^{-(1-\epsilon)\min\{\Lambda_0 t/2, \beta \lambda_0 t\}}$. 
Fix some  $\nu\ge 2\beta+2d+2$ arbitrarily and decompose $\cL^t:\B^\beta_{\nu}\to \B^\beta_{\nu}$ for $t>3t_0$  as
\[
\begin{CD}
\B^\beta_{\nu} @>{\cL^{t_0}}>> \B^\beta_{\nu_*}
@>{\cL^{t-2t_0}}>>
\B^\beta_{\nu_*} @>{\cL^{t_0}}>> \B^\beta_{\nu}
\end{CD}
\]
Letting $t\to \infty$ and using the basic properties of essential spectral radius mentioned above, we see that the essential spectral radius of $\cL^t:\B^\beta_{\nu}\to \B^\beta_{\nu}$ is bounded by that of $\cL^t:\B^\beta_{\nu_*}\to \B^\beta_{\nu_*}$ and hence by $2^{-(1-\epsilon)\min\{\Lambda_0 t/2, \beta \lambda_0 t\}}$. Since $\epsilon>0$ is arbitrary, we obtain the main theorem, setting $\B^\beta=\B^\beta_{\nu}$.

\section{A local geometric property of the diffeomorphisms in $\cH$}\label{sec:localG}
In this section, we give a local geometric property of the diffeomorphisms in $\cH$. 
Let $G:V'\to V=G(V')$ be a $C^r$ diffeomorphism satisfying  the conditions
 (H0) and (H1) in the definition of $\cH(\lambda,\Lambda)$. 
Take a small disk $D\subset V'$ and set $D'=\pi_{+,-}(D)$. 
Since $G$ preserves the standard vector field~$v_0$, there exist a $C^r$ function $G_0:D' \to \real$
and a $C^r$ diffeomorphism
\[
G_{+,-}:D' \to G_{+,-}(D')\subset \real^{2d}, \;\; G_{+,-}(x^{+},x^{-})=(G_+(x^{+},x^{-}), G_-(x^{+},x^{-})),
\]
such that
\[
G(x_0,x^{+}, x^{-})=(x_0+G_0(x^{+},x^{-}),  G_{+}(x^{+},x^{-}), G_{-}(x^{+},x^{-}))\quad \text{on $D$.}
\]
\begin{lemma}\label{lm:local}
If $G(\zero)=\zero\in D$ in addition, we have $
DG_0(\zero)=D^2G_0(\zero)=0$.
\end{lemma}
\begin{proof} 
Comparing the coefficients of $dx^+$ and $dx^-$ in $G^*(\alpha_0)=\alpha_0$, we get
 \begin{align*}
\frac{\partial G_0}{\partial x^{+}}  &=- G_-\cdot \frac{\partial G_+}{\partial x^{+}} + G_+\cdot \frac{\partial G_-}{\partial x^{+}}  +x^-
\intertext{and}
\frac{\partial G_0}{\partial x^{-}}  &=- 
G_-\cdot \frac{\partial  G_+}{\partial x^{-}} 
+
G_+\cdot \frac{\partial G_-}{\partial x^{-}}  -x^+.
\end{align*}
This implies $
{\partial G_0}/{\partial x^{+}}(\zero)={\partial G_0}/{\partial x^{-}}(\zero)=0$.
Differentiating 
both sides with respect to $x^+$ and $x^-$ and using the assumption $G(\zero)=\zero$, we also obtain $
{\partial^2 G_0}/{\partial x^{+} \partial x^{+}}(\zero)={\partial^2 G_0}/{\partial x^{+} \partial x^{-}}(\zero)={\partial^2 G_0}/{\partial x^{-} \partial x^{-}}(\zero)=0$.
\end{proof}

For $y=(y_0,y^+,y^-)\in E$, the affine bijection $\Phi_{y}:E\to E$  defined by
\begin{equation}\label{eqn:Phi}
\Phi_{y}(x_0, x^+, x^-)=(y_0+x_0-(y^-\cdot x^+)+ (y^+\cdot x^-), \;y^++x^+, \;y^-+x^-)
\end{equation}
moves the origin $\zero$ to $y$,
preserving the contact form $\alpha_0$ and the vector field~$v_0$. 
So the assumption $G(\zero)=\zero$ in Lemma \ref{lm:local} is not essential. 
\begin{corollary}\label{cor:local}
For  any diffeomorphism $G:V'\to V$  in $\mathcal{H}$ and any compact subset
$K$  of\/ $V'$, there exists a constant $C=C(G)>0$ such that,
if $y,y'\in K$ and if $\xi\in E^*$ is written in the form $
\xi=\xi_0 \cdot \alpha_0(G(y)) +\xi_{+,-}$ with 
$\xi_0=\pi^*_0(\xi)$ and $\xi_{+,-}\in E^*_+\oplus E^*_-$, we have
\[
\|DG_{y'}^*(\xi)-DG_{y}^*(\xi)\|\le C\cdot( |\xi_0| \cdot \|y'-y\|^2+\|\xi_{+,-}\|\cdot \|y'-y\|)
\]
and
\[
|\langle \xi, D^2G_{y'}(v,v')\rangle| \le C\cdot (|\xi_0|\cdot \|y'-y\|+\|\xi_{+,-}\|)
\cdot \|v\|\cdot \|v'\|\quad \mbox{for $v,v'\in E$.}
\]
\end{corollary}
\begin{proof} Changing coordinates by the affine bijections $\Phi_y$ and $\Phi_{G(y)}$, we may suppose $y=G(y)=\zero$. Then the claim follows from Lemma~\ref{lm:local}. 
\end{proof}



\section{Partitions of Unity}\label{sec:part}
In this section, we construct a partition of unity $\{p_\gamma\}_{\gamma\in \Gamma}$ on the cotangent bundle $T^*_{\disk}E=\disk\times E^*$ over the unit disk $\disk\subset E$. This will be used in the definition of the Hilbert spaces $\cB^\beta_\nu$ in the next section.

\subsection{Partitions of unity on $E$}\label{ss:poe}
Take a $C^\infty$ function  $\chi:\real\to [0,1]$   so that
\[
\chi(s)=\begin{cases}
1, &\text{ if $s\le 4/3$;}\\
0, &\text{ if $s\ge 5/3$,}
\end{cases}
\]
and define a $C^\infty$ function  $\rho:\real\to [0,1]$ by
\[
\rho(s)=\chi(s+1)-\chi(s+2),
\]
which is supported on the interval $[-2/3,2/3]$. 
For integers $n$ and~$k$, we define the $C^\infty$ function $\rho_{n,k}:\real\to [0,1]$ by
\[
\rho_{n,k}(s)=\rho(2^{n/2}s-k).
\]
Then, for each $n$, the family of functions $
\{ \rho_{n,k}(s)\mid k\in \integer\}$ is a $C^\infty$ partition of unity on the real line $\real$, such that $\supp\rho_{n,k}(\cdot)$ is contained in the interval 
\[
[2^{-n/2}(k-(2/3)), 2^{-n/2}(k+(2/3))]\subset [2^{-n/2}(k-1), 2^{-n/2}(k+1)].
\]

Similarly, for an integer  $n$ and $\bk=(k_0,k^+_1,\cdots, k^+_d, k^-_1,\cdots, k^-_d)\in \integer^{2d+1}$, we define the $C^\infty$ function $\rho_{n,\bk}:E\to [0,1]$ by
\[
\rho_{n,\bk}(x)=
\rho(2^{n/2} x_0-k_0) \prod_{\sigma=\pm}\prod_{i=1}^d\rho(2^{n/2}x_i^\sigma-k_i^\sigma).
\]
Again, for each $n$, the family of functions $ \{ \rho_{n,\bk}(s)\mid \bk\in \integer^{2d+1}\}$
are $C^\infty$ partition of unity on $E$, such that  
$\supp \rho_{n,\bk}$ is contained in the cube
\[
\Ze(n,\bk)=
[2^{-n/2}(k_0-1), 2^{-n/2}(k_0+1)]\times \prod_{\sigma=\pm}\prod_{i=1}^d [2^{-n/2}(k_i^\sigma-1), 2^{-n/2}(k_i^\sigma+1)],
\]
whose center is at the point 
\[
z(n, \bk):=2^{-{n/2}}\bk=2^{-{n/2}}(k_0,k^+_1,\dots, k^+_d, k_1^-,\dots, k^-_d).
\] 
Note that the functions  $\rho_{n,k}$ and $\rho_{n',k'}$ (resp. $\rho_{n,\bk}$ and $\rho_{n',\bk'}$) introduced above 
are related each other by translation if $n=n'$ and by translation and similitude otherwise. For any integer $\ell\ge 0$ (resp. for any multi-index $\alpha\in (\integer_+)^{2d+1}$), there exists  a constant $C_\ell>0$ (resp.~$C_\alpha>0$), which does not depend on $n$ and $k$ (resp.~$n$ and $\bk$), such that 
\[
\| D^\ell \rho_{n,k}\|_{L^\infty}<C_\ell \cdot 2^{\ell \cdot  n/2}\quad \mbox{(resp. } \| D^\alpha \rho_{n,\bk}\|_{L^\infty}<C_\alpha \cdot 2^{|\alpha|\cdot  n/2}\mbox{ )}.
\]

\subsection{Partitions of unity on $E^*$}\label{ss:PU}
We next introduce a few partitions of unity on the dual space $E^*$. 
For $n\ge 0$, we consider the $C^\infty$ function 
\[
\chi_n:\real\to [0,1],\qquad \chi_n(s)=\begin{cases}
\chi(2^{-n}|s|)-\chi(2^{-n+1}|s|), &\text{ if $n\ge 1$;}\\
\chi(|s|), &\text{ if $n=0$.}
\end{cases}
\]
The functions $\chi_n$ for $n\ge 0$ is a $C^\infty$ partition of unity on $\real$, which is sometimes called the Littlewood-Paley partition of unity. 
The function $\chi_n$ for $n\ge 1$ is related to $\chi_1$ by similitude. More precisely, we have 
\[
\chi_n(s)=\chi_1(2^{-n+1}s).
\] 
So, for an integer $\ell\ge 0$, there exists a constant $C_\ell>0$, which does not depend on $n$, such that
\[
\|D^\ell \chi_n\|_{L^\infty} <C_\ell\cdot 2^{-\ell n}.
\]

We also introduce  the $C^\infty$ functions
\[
\widetilde{\chi}_n:\real\to [0,1],\qquad 
\widetilde{\chi}_n(s)=\begin{cases}
\chi_{n-1}(s)+\chi_{n}(s)+\chi_{n+1}(s),
&\text{ if $n\ge 1$;}\\ 
\chi_{0}(s)+\chi_{1}(s), &\text{ if $n=0$.}
\end{cases}
\]
Note that  $\widetilde{\chi}_n\equiv 1$ on the support of $\chi_n$ and that the family $\widetilde{\chi}_n$ enjoys the same scaling property and derivative estimates as we stated   for $\chi_n$. 

Next, for $n\ge 0$ and $k \in \integer$, we consider the $C^\infty$ function
\begin{align*}
\chi_{n,k}:E^*\to [0,1],\quad  \chi_{n,k}(\xi)&=\rho_{(-n),k}(\xi_0)\cdot  \chi_n(\xi_0)\\
&=
\rho\big(2^{-n/2}\cdot \xi_0-k\big)\cdot  \chi_n(\xi_0)
\end{align*}
where $\xi_0=\pi^*_0(\xi)$. Notice that the size of the support of the function $\rho_{(-n),k}$ is proportional to $2^{n/2}$ while that of $\chi_n$ is proportional to $2^n$. 
Obviously the family of functions $
\{\chi_{n,k}\mid n\ge 0, k\in \integer\}$
is a $C^\infty$ partition of unity on $E$. 
In the same spirit as in  the definition of  $\widetilde{\chi}_n$ above, we also introduce the functions
\[
\widetilde{\chi}_{n,k}:E^*\to [0,1],\quad \widetilde{\chi}_{n,k}(\xi)=
 \rho_{(-n), k-1}(\xi_0)+\rho_{(-n), k}(\xi_0)+\rho_{(-n), k+1}(\xi_0),
\]
which satisfy   $\widetilde{\chi}_{n,k}\equiv 1$ on the support of $\chi_{n,k}$. From the estimates on the derivatives of 
$\rho_{n,k}$ and $\chi_n$, there exists a constant $C_\ell>0$ for each $\ell\ge 0$, which does not depend on $n$ nor $k$, such that
\[
\|D^\ell \chi_{n,k}\|_{L^\infty}<C_\ell \cdot 2^{-\ell  n/2}\quad \mbox{and}\quad 
\|D^\ell \widetilde{\chi}_{n,k}\|_{L^\infty}<C_\ell \cdot 2^{-\ell  n/2}.
\] 
\begin{remark}\label{rem:conv}
We will ignore the functions $\chi_{n,k}$ that vanish everywhere. Thus, for each $n\ge 0$, we consider the functions $\chi_{n,k}$ only for finitely many $k$'s. 
\end{remark}

For $\theta>0$, we consider the cones 
\begin{align*}
\cone_+^*(\theta)&=\{(0,\xi^+,\xi^-)\in E^*_+\oplus E^*_-\mid \|\xi^-\|\le \theta\|\xi^+\|\;\}
\intertext{and}
\cone_-^*(\theta)&=\{(0,\xi^+,\xi^-)\in E^*_+\oplus E^*_-\mid \|\xi^+\|\le \theta\|\xi^-\|\;\}
\end{align*}
in $E^*_+\oplus E^*_-\subset E^*$.
These cones for $\theta=1/10$ may be regarded as the duals of the cones $\cone_-$ and $\cone_+$ in the definition of $\cH$ respectively.  
Let $S^*$ be the unit sphere 
in $E^*_+\oplus E^*_-$. 
We henceforth fix  $C^\infty$ functions
$
\varphi_{\sigma}:S^*\to [0,1]$ and $ 
\widetilde{\varphi}_{\sigma}:S^*\to [0,1]$ for $\sigma\in \{+,-\}$
such that 
\begin{itemize}
\setlength{\itemsep}{4pt}
\item[(i)]
$\varphi_\sigma\equiv 1$ on a neighborhood  $S^*\cap C_\sigma^*(4/10)$  for $\sigma=\pm$, 
\item[(ii)] $\varphi_+(\xi)+\varphi_-(\xi)= 1$  for all  $\xi\in S^*$,
\item[(iii)]
$\widetilde{\varphi}_\sigma\equiv 1$ on $ C_\sigma^*(6/10)$ and  $\supp\widetilde{\varphi}_\sigma\subset C_\sigma^*(7/10)\cap S^*$  for $\sigma=\pm$.
\end{itemize}
Note that the conditions (i) and (ii) above imply that the support of $\varphi_\sigma$ is contained in $ C_\sigma^*(6/10)$ and, hence, the condition (iii) implies that $\widetilde{\varphi}_\sigma\equiv 1$ on the support of $\varphi_\sigma$. 

For an integer $m$, let $\psi_{m}:E^*_+\oplus E^*_-\to [0,1]$ and $\widetilde{\psi}_{m}:E^*_+\oplus E^*_-\to [0,1]$ be $C^\infty$ functions defined respectively by
\[
\psi_{m}(\xi)=
\begin{cases}
\chi_m(\|\xi\|)\cdot \varphi_+(\xi/\|\xi\|), &\text{ if $m\ge 1$;}\\
\chi_0(\|\xi\|), &\text{ if $m=0$;}\\
\chi_{|m|}(\|\xi\|)\cdot  \varphi_-(\xi/\|\xi\|), &\text{ if $m\le -1$}
\end{cases}
\] 
and
\[
\widetilde{\psi}_{m}(\xi)=
\begin{cases}
\widetilde{\chi}_m(\|\xi\|)\cdot \widetilde{\varphi}_+(\xi/\|\xi\|),& \text{if $m\ge 1$;}\\
\;\widetilde{\chi}_0(\|\xi\|),& \text{if $m=0$;}\\
\widetilde{\chi}_{|m|}(\|\xi\|)\cdot \widetilde{\varphi}_-(\xi/\|\xi\|),& \text{if $m\le -1$.}
\end{cases}
\]
The family $\{\psi_{m}\}_{m\in \integer}$ is a $C^\infty$ partition of unity on the subspace $E^*_+\oplus E^*_-$ and we have  $\widetilde{\psi}_m\equiv 1$ on $\supp \psi_m$. 
Note that the functions $\psi_m$ for $m>0$ (resp. $m<0$) are related each other by similitude. More precisely, we have
\[
\psi_{m'}(\xi)=\psi_m(2^{|m|-|m'|}\xi)\quad \mbox{if $m\cdot m'>0$.} 
\]
This scaling relation is true also for $\widetilde{\psi}_m$. 
Hence, for any $\alpha\in (\integer_+)^{2d}$, there exists a constant $C_\alpha>0$, which does not depend on $m$, such that 
\[
\|D^\alpha \psi_m\|_{L^\infty}<C_\alpha\cdot 2^{-|\alpha|\cdot m}\quad \mbox{and}\quad 
\|D^\alpha \widetilde{\psi}_m\|_{L^\infty}<C_\alpha\cdot 2^{-|\alpha| \cdot m}.
\]
Note that the support of the function $\psi_{m}$  (resp. $\widetilde{\psi}_{m}$) is contained in the disk on 
$E^*_+\oplus E^*_-$ with radius  $2^{m+1}$ (resp. $2^{m+2}$)  and center at the origin.

Next we define $C^\infty$ functions
$\psi_{n,k,m}:E^*\to [0,1]$ and $\widetilde{\psi}_{n,k,m}:E^*\to [0,1]$ 
for $n\ge 0$ and $k,m\in \integer$
respectively by
\[
\psi_{n,k,m}(\xi)= \chi_{n,k}(\xi)\cdot \psi_{m}(2^{-n/2}\xi^+,2^{-n/2}\xi^-)
\]
and
\[
\widetilde{\psi}_{n,k,m}(\xi)=\widetilde{\chi}_{n,k}(\xi)\cdot \widetilde{\psi}_{m}
(2^{-n/2}\xi^+, 2^{-n/2} \xi^-)
\]
where $\xi=(\xi_0,\xi^+,\xi^-)$. 
Then the family  $\{\psi_{n,k,m}\mid n\ge 0, m,k\in \integer\}$ is a $C^{\infty}$ partition of unity on $E^*$ and we have $\widetilde{\psi}_{n,k,m}\equiv 1$ on the support of ${\psi}_{n,k,m}$. From the estimates on the derivatives of $\chi_{n,k}$ and $\psi_m$ (resp. $\widetilde{\chi}_{n,k}$ and $\widetilde{\psi}_m$), we see that,  for any multi-index $\alpha\in (\integer_+)^{2d+1}$, there exists a constant $C_\alpha>0$, which does not depend on $n,k$ nor $m$, such that 
\begin{align*}
&\|D^\alpha \psi_{n,k,m}\|_{L^\infty}<C_\alpha\cdot 2^{-|\alpha|\cdot (n/2) -|\alpha|_\dag \cdot |m|}
\intertext{and}
&\|D^\alpha \widetilde{\psi}_{n,k,m}\|_{L^\infty}<C_\alpha\cdot 2^{-|\alpha|\cdot (n/2) -|\alpha|_\dag \cdot |m|},
\end{align*}
where we set 
\begin{equation}\label{eqn:sdag}
|\alpha|_\dag=|\alpha|-\alpha_0\quad \mbox{ for $\alpha=(\alpha_0, \alpha^+_1, \cdots, \alpha^+_d, \alpha^-_1,\cdots, \alpha^-_d)\in (\integer_+)^{2d+1}$.}
\end{equation}

\subsection{Partitions of unity on $T^*_\disk E=\disk\times E^*$}
As we noted in Remark \ref{rem:conv}, we consider   the set 
\[
\cN=\{(n,k)\in \integer_+\oplus \integer\mid \text{$\chi_{n,k}$ does not vanish completely.}\}
\]
as the index set of the partition of unity $\{\chi_{n,k}\}$. 
Below we introduce a $C^\infty$ partition of unity on $\disk\times E^*$ whose index set is 
\[
\Gamma=\bigl\{(n,k,m,\bk)
\in \cN \oplus \integer \oplus \integer^{2d+1} \mid\;\supp \rho_{n,\bk} \cap \disk \neq \emptyset\bigr\}.
\]
To refer the components of $\gamma=(n,k,m,\bk)\in \Gamma$, we set 
\[
n(\gamma)=n,\;\;k(\gamma)=k, \;\; m(\gamma)=m\quad \text{and}\quad \bk(\gamma)=\bk.
\]
For simplicity, we put
\[
\rho_{\gamma}=\rho_{n(\gamma), \bk(\gamma)}, \quad \Ze(\gamma)=\Ze(n(\gamma), \bk(\gamma))\;\;\mbox{ and }\;\;z(\gamma)=z(n(\gamma), \bk(\gamma)).
\] 
Note that, from the condition in the definition of $\Gamma$ above, 
$\|z(\gamma)\|$ for $\gamma\in \Gamma$ are uniformly bounded by some constant which depends only on $d$.  

Recall the diffeomorphism $\Phi_y:E\to E$  defined for $y\in E$ by (\ref{eqn:Phi}).
For each $\gamma\in \Gamma$, we consider the linear map
\[
\Phi_\gamma=((D\Phi_{z(\gamma)})_{\zero})^{*}:T_{z(\gamma)}E^*\to T_{\zero}E^*,
\]
which is characterized by the conditions
\[
\Phi_\gamma(\alpha_0(z(\gamma)))=\alpha_0(\zero)\quad \mbox{ and}\quad  \Phi_\gamma|_{E^*_+\oplus E^*_-}=id.
\] 
We then define the $C^\infty$ functions
$
\psi_\gamma:E^*\to [0,1]$ and $\widetilde{\psi}_\gamma:E^*\to [0,1]$ by
\[
\psi_{\gamma}= \psi_{n(\gamma), k(\gamma), m(\gamma)} \circ 
\Phi_\gamma
\quad \text{and}\quad
\widetilde{\psi}_{\gamma}= \widetilde{\psi}_{n(\gamma), k(\gamma), m(\gamma)}\circ 
\Phi_\gamma.
\]
Finally we define the family of $C^\infty$ functions $p_\gamma:T^*E\to [0,1]$ for $\gamma\in \Gamma$ by 
\[
p_\gamma(x,\xi)=\rho_\gamma(x)\cdot \psi_\gamma(\xi)\quad 
\text{for }(x,\xi)\in T^*E=E\times E^*. 
\]
This family of functions is a $C^\infty$ partition of unity on $\disk \times E^*$. In fact, for given $(n,k)\in \cN$ and $\bk\in \integer^{2d+1}$, we have
\[
\sum_{\gamma:n(\gamma)=n;k(\gamma)=k;\bk(\gamma)=\bk} p_{\gamma}(x,\xi)=
\rho_{n,\bk}(x)\cdot \chi_{n,k}(\xi)\quad\text{ for } (x,\xi)\in T^*E
\]
 and hence 
\[
\sum_{\gamma\in \Gamma} p_\gamma(x,\xi)\equiv 1 \quad \text{ for $(x,\xi)\in \disk\times E^*$.}
\]

\subsection{Boundedness of the families $\psi_\gamma$ and $\widetilde{\psi}_\gamma$} \label{ss:bdd}
One important property of the families $\psi_\gamma$ and $\widetilde{\psi}_\gamma$ is that they are bounded up to some scaling and translation in the following sense.
For integers $n\ge 0$ and $m$, we consider the linear map $J_{n,m}:E^*\to E^*$ defined by
\[
J_{n,m}(\xi_0,\xi^+,\xi^-)=(2^{n/2}\xi_0, 2^{n/2+|m|}\xi^+,2^{n/2+|m|}\xi^-).
\]
For $\gamma\in \Gamma$, let $A_\gamma:E^*\to E^*$ be the translation defined by
\[
A_\gamma(\xi)=\xi+k(\gamma)\cdot 2^{n(\gamma)/2}\cdot \alpha_0(z(\gamma)),
\]
which moves the origin to the center of the support of $\psi_\gamma$.
\begin{lemma}\label{lm:scaling}
The set of functions 
\[
\psi_\gamma\circ A_\gamma\circ J_{n(\gamma), m(\gamma)}\quad \mbox{ and }\quad \widetilde{\psi}_\gamma\circ A_\gamma\circ J_{n(\gamma), m(\gamma)}\quad \mbox{ for $\gamma\in \Gamma$}
\]
is bounded in $\mathcal{D}(E^*)$, that is, their supports are contained  in a bounded subset in~$E^*$ and   their $C^s$ norms are uniformly bounded for each $s\ge 0$. \end{lemma}
 \begin{proof}From the definitions, we have
 \begin{align*}
 \psi_\gamma\circ &A_\gamma\circ J_{n(\gamma),m(\gamma),k(\gamma)}(\xi_0,\xi^+,\xi^-)\\
 &=\rho(\xi_0)\cdot \psi_{m(\gamma)}(2^{|m(\gamma)|}\xi_+-\xi_0\cdot \pi_+^*\alpha(z(\gamma)), 
 2^{|m(\gamma)|}\xi_--\xi_0\cdot \pi_-^*\alpha(z(\gamma))).
 \end{align*}
 We also have the same formula with $\psi_\gamma$ and $\psi_{m(\gamma)}$ replaced by  $\widetilde{\psi}_\gamma$ and $\widetilde{\psi}_{m(\gamma)}$ respectively. 
 Therefore the claim of the lemma follows from the properties of the functions $\psi_m$ and $\widetilde{\psi}_m$ that is mentioned previously and the fact that $\|z(\gamma)\|$ for $\gamma\in \Gamma$ are bounded. 
 \end{proof}
For $n\in \integer_+$, $m\in \integer$ and $\mu>0$, 
we define the function $b_{n,m}^\mu:E\to \real_+$ by 
\begin{equation}\label{eqn:bnm}
b_{n,m}^\mu(x)=|\det J_{n,m}| \cdot \langle J_{n,m}(x)\rangle^{-\mu},
\end{equation}
where (and henceforth) we set
\[
\langle y\rangle =(1+\|y\|^2)^{1/2}.
\]
For brevity, we set $
b_{\gamma}^\mu=b_{n(\gamma),m(\gamma)}^\mu$ for $\gamma\in \Gamma$. Then the last lemma implies
\begin{corollary}
\label{cor:localized}For each $\mu>0$, there exists a constant $C_*>0$ such that 
\[
|\Fourier^{-1}\psi_\gamma(x)|\le C_*\cdot b_{\gamma}^\mu(x)\quad \mbox{ and }\quad 
|\Fourier^{-1}\widetilde{\psi}_\gamma(x)|\le C_*\cdot b_{\gamma}^\mu(x)
\]
for all $x\in E$ and $\gamma\in \Gamma$, where $\Fourier$ denotes the Fourier transform.
\end{corollary}
\begin{proof}
From Lemma \ref{lm:scaling}, we see that 
\[
\Fourier^{-1}(\psi_\gamma\circ J_{n,m})=|\det J_{n,m}|^{-1}\cdot (\Fourier^{-1}\psi_\gamma)\circ J_{n,m}^{-1}
\]
is bounded in the Schwartz space $\mathcal{S}(E)$. This implies the claim above.
\end{proof}

\section{The  Hilbert spaces $\cB^\beta_\nu$}\label{sec:Hil}
In this section, we define the  Hilbert spaces $\cB^\beta_\nu$ in Theorem \ref{th:mainlocal}.
\subsection{Decomposition of functions using pseudodifferential operators}
For a $C^\infty$ function $p:T^*E\to \real$ on the cotangent bundle $T^*E=E\times E^*$ with compact support, the adjoint of the pseudodifferential operator $p_{\gamma}(x,D)$ with symbol $p_\gamma$ is
the operator $
p(x,D)^*:L^2(E)\to L^2(E)$
given by
\[
p(x,D)^* u(x)=(2\pi)^{-(2d+1)} \int e^{i\langle \xi, x-y\rangle } p(y,\xi) u(y) dy d\xi.
\]
\begin{remark} The notation $p(x,D)^*$ should be read as a single symbol and the letter $x$ and $D$ in it have no meaning as variable. 
We refer to \cite{Hormander3, Taylor} for the general definition and properties of pseudodifferential operators. 
But we actually need no knowledge on pseudodifferential operators in the following argument, since we consider only simple cases as we will see below. 
\end{remark}

For $u\in L^2(\disk)$ and $\gamma\in \Gamma$, we set $
u_\gamma=p_\gamma(x,D)^* u$. 
From the definition of the function $p_\gamma$, we may write it in a simpler form as follows. For a $C^\infty$ function $\psi:E^*\to \real$ with compact support, let us consider the operator 
$\psi(D):L^2(E)\to L^2(E)$ defined by
\begin{align*}
\psi(D) u(x)&=(2\pi)^{-(2d+1)} \int e^{i\langle \xi, x-y\rangle } \psi(\xi) u(y) dy d\xi\\
&=\Fourier^{-1}(\psi\cdot \Fourier u)(x)=(\Fourier^{-1} \psi)* u(x).
\end{align*}
Then we may write $u_\gamma$ as 
\[
u_\gamma=p_\gamma(x,D)^* u
=\psi_{\gamma}(D) (\rho_\gamma\cdot u)=
(\Fourier^{-1} \psi_\gamma)*(\rho_\gamma\cdot u).
\]

Note that we have 
 $u=\sum_{\gamma\in \Gamma} u_\gamma$ for $u\in L^2(\disk)$
because $\{p_\gamma\}_{\gamma\in \Gamma}$ is a partition of unity on $T_{\disk}^*E$. 
Also observe that each $u_\gamma$ is localized near the support of $\rho_\gamma$ by Corollary \ref{cor:localized} and its 
Fourier transform is supported in  
that of $\psi_\gamma$ from the definition.

\subsection{The definition of the Hilbert space $\cB^\beta$} 
For $\beta>0$ and $\nu\ge 2d+2$, we set 
\[
\|u\|_{\beta,\nu}=\left(\sum_{\gamma\in \Gamma} 2^{2\beta m(\gamma)}\| d_\gamma^\nu \cdot u_\gamma\|_{L^2}^2\right)^{1/2}\quad \text{ for $u\in C^\infty(\disk)$,}
\]
where $\dt_\gamma^\nu=(\dt_\gamma)^\nu$ and $\dt_\gamma:E\to \real$ is the function defined by
\[
\dt_\gamma(x)=\langle 2^{n(\gamma)/2}(x-z(\gamma))\rangle =
\left(1+2^{n(\gamma)}\|x-z(\gamma)\|^2\right)^{1/2}.
\]
Then we have 
\begin{lemma}\label{lm:wd} For $0<\beta<s$ and $\nu\ge 2d+2$, there exists a constant $C>0$ such that  
$
 (1/C) \|u\|_{W^{-s}}\le \|u\|_{\beta,\nu}\le C \|u\|_{W^s}
$ for all $u\in C^{\infty}(\disk)$.
\end{lemma}
In particular, $\|\cdot \|_{\beta,\nu}$ is a norm on $C^\infty(\disk)$ associated to a unique inner product $( \cdot, \cdot )_{\beta,\nu}$. 
We give the proof of Lemma \ref{lm:wd} in the appendix at the end of this paper,  because it requires some estimates that will be given in the following sections. 
 Now we  define the Hilbert space $\cB^\beta_\nu$ as follows
\begin{definition}
For $0<\beta<(r-1)/2$ and $\nu\ge 2d+2$, the Hilbert space~$\cB^\beta_\nu$ is  the completion of the space $C^\infty(\disk)$ with respect to the norm $\|\cdot \|_{\beta,\nu}$, equipped with the extension of the inner product $( \cdot, \cdot )_{\beta,\nu}$.
\end{definition}
Clearly it follows from Lemma \ref{lm:wd} that  
\[
W^s(\disk)\subset \cB^\beta_\nu\subset W^{-s}(\disk)\quad \mbox{for $s>\beta$.}
\]

\begin{remark}\label{re:ori} If we take an appropriate  $C^\infty$ function $P_{\beta}:T^*E\to \real$ that approximates  $\sum_{\gamma} 2^{\beta m(\gamma)} p_{\gamma}$ and consider the norm $\|u\|_{\beta}=\| P_\beta(x,D)^* u\|_{L^2}$ in the place of the norm $\|\cdot \|_{\beta,\nu}$ in the definition above, we  get a Hilbert space similar to  $\cB^\beta_\nu$. 
Such definition of the Hilbert space should look much simpler and, actually, this is what we had in mind in the beginning for the definition of the Hilbert spaces in Theorem \ref{th:mainlocal}. But we adopt the rather involved definition of $\cB^\beta_\nu$ above because it fits directly to the argument in the proof  and 
because we like to avoid technical difficulties\footnote{The main difficulty is that the symbol $P_\beta$ that we want to consider does not belong to the standard classes of symbols  in the (classical) theory of pseudodifferential operators. It would be very interesting and preferable if our results is formulated, proved or improved in terms of pseudodifferential operators. }
 related to pseudodifferential operators. 
\end{remark}


\section{The auxiliary operator $\cM(G,g)$}\label{sec:aux}
In this section, we introduce the operator
$\cM(G,g):\bB^\beta_{\nu}\to \bB^\beta_{\nu'}$ between Hilbert spaces, which  is an  extension of  the operator $\cL(G,g):\cB^\beta_\nu\to \cB^\beta_{\nu'}$ in the sense that there exists an isometric embedding $\iota:\cB^\beta_\nu\to \bB^\beta_\nu$ and that the following diagram commutes:
\begin{equation}\label{cd:cMcLrel}
\begin{CD}
\bB^\beta_\nu @>{\cM(G,g)}>> \bB^\beta_{\nu'}\\
@AA{\iota}A @AA{\iota}A\\
\cB^\beta_\nu @>{\cL(G,g)}>> \cB^\beta_{\nu'}
\end{CD}
\end{equation}

\subsection{The definition of the operator $\cM$}
For $\beta>0$ and $\nu\ge  2d+2$, we consider the Hilbert space
$\bB^\beta_\nu\subset (L^2(E))^\Gamma$ defined by
\[
\bB^\beta_\nu =
\left\{\bu=(u_{\gamma})_{\gamma\in \Gamma}\;\left|\;\;  \widetilde{\psi}_{\gamma}(D) u_\gamma=u_\gamma, \;\; 
\sum_{\gamma\in \Gamma} 2^{2\beta m(\gamma)}\|d_\gamma^\nu \cdot u_\gamma\|^2_{L^2}<\infty \right.
\right\}
\] 
and equipped with the norm 
\[
\|\bu\|_{\beta,\nu}=\sum_{\gamma\in \Gamma} 2^{2\beta m(\gamma)}\|d_\gamma^\nu \cdot u_\gamma\|_{L^2}^2.
\] 
Then the injection $
\iota: \cB^\beta_\nu\to \bB^\beta_\nu$,  $\iota(u)=
\left(p_{\gamma}(x,D)^* u\right)_{\gamma\in \Gamma}$,
is an isometric embedding. (Notice that we have $\widetilde{\psi}_{\gamma}(D) \psi_\gamma(D) u=(\widetilde{\psi}_{\gamma}\cdot \psi_\gamma)(D) u={\psi}_{\gamma}(D) u$.)

Suppose that $v=\cL(G,g) u$ for $u\in L^2(\disk)$ and 
set $u_\gamma=p_{\gamma}(x,D)^* u$ and $v_\gamma=p_{\gamma}(x,D)^* v$ for $\gamma\in \Gamma$. Then we have
\begin{equation}\label{eqn:uvsum}
v_{\gamma'}=\sum_{\gamma\in \Gamma} \cL_{\gamma\gamma'}u_{\gamma},
\end{equation}
where the operator $\cL_{\gamma\gamma'}=\cL_{\gamma\gamma'}(G,g):L^2(E)\to L^2(E)$ is defined by
\begin{equation}\label{eqn:cLgg}
\cL_{\gamma\gamma'}w=p_{\gamma'}(x,D)^*( \cL(G,g)(\widetilde{\psi}_\gamma(D) w)).
\end{equation}
\begin{remark}
Since $\widetilde{\psi}_\gamma(D) u_\gamma=u_\gamma$ in the setting above, 
the operation $\widetilde{\psi}_\gamma(D)$ in (\ref{eqn:cLgg}) 
is not necessary for (\ref{eqn:uvsum}) to hold. But this operation makes sense when we regard $\cL_{\gamma\gamma'}$ as an operator on $L^2(E)$. 
\end{remark}

We define the operator  
$
\cM(G,g):\bB^\beta_\nu\to \bB^\beta_{\nu'}$
 formally by
\begin{equation}\label{eqn:defcM}
\cM(G,g)((u_{\gamma})_{\gamma\in \Gamma})=\left(\sum_{\gamma\in \Gamma} \cL_{\gamma\gamma'}(u_\gamma)\right)_{\gamma'\in \Gamma}.
\end{equation}
Then, by (\ref{eqn:uvsum}), the diagram (\ref{cd:cMcLrel}) commutes at the formal level at least. 
In the following sections, we will prove 
\begin{theorem}\label{th:reduced} 
There exist  constants $\lambda_*>0$ and $\Lambda_*>0$ such that, for $G:V'\to V$ in $\cH(\lambda_*,\Lambda_*)$ and $g\in\cC^r(V')$,  the formal definition of $\cM(G,g)$ above 
 gives a bounded operator $\cM(G,g):\bB^\beta_\nu\to \bB^\beta_{\nu'}$ for $0<\beta<(r-1)/2$ and $\nu,\nu'\ge 2\beta+2d+2$, and the  diagram (\ref{cd:cMcLrel}) commutes. 

Further, for any $\epsilon>0$ and $0<\beta<(r-1)/2$, there exist constants $\nu_*\ge 2\beta+2d+2$, $C_*>0$ and 
a family of norms $\|\cdot\|^{(\lambda)}$ on $\bB^\beta_{\nu_*}$ for $\lambda>0$, which are all equivalent to the norm $\|\cdot\|_{\beta,\nu_*}$, such that,  if $G:V'\to V$ belongs to $\cH(\lambda,\Lambda)$ for $\lambda\ge \lambda_*$ and  $\Lambda\ge \Lambda_*$ with $\Lambda\ge d\lambda$ and if $g\in \cC^r(V')$,
there exists a compact operator $\mathcal{K}(G,g):\bB^\beta_{\nu_*}\to \bB^\beta_{\nu_*}$ such that the operator norm of $\cM(G,g)-\mathcal{K}(G,g):\bB^\beta_{\nu_*}\to \bB^\beta_{\nu_*}$ with respect to the norm $\|\cdot\|^{(\lambda)}$ is bounded by $
C_* \cdot \|g\|_{*}\cdot  2^{-(1-\epsilon)\min\{\Lambda/2,\beta \lambda\}}$.

\end{theorem}

Since the operator $\iota$ in (\ref{cd:cMcLrel}) is an isometric embedding,
Theorem \ref{th:mainlocal} follows from Theorem \ref{th:reduced} immediately. 

\subsection{The operator $\cL_{\gamma\gamma'}$} 
\label{sec:Ell}
The operator  $\cL_{\gamma\gamma'}:L^2(E)\to L^2(E)$ defined in the last subsection can be regarded as an integral operator 
\[
\cL_{\gamma\gamma'}u(x')=
\int \kappa_{\gamma\gamma'}(x',x) u(x) dx
\]
with the smooth kernel
\begin{align}
\kappa_{\gamma\gamma'}&(x',x)=\int \Fourier^{-1}\psi_{\gamma'}(x'-y) \cdot \rho_{\gamma'}(y) \cdot g(y)
\cdot \Fourier^{-1}\tilde{\psi}_{\gamma}(G(y)-x) dy \notag\\
&=(2\pi)^{-2(2d+1)}\int e^{i\langle\xi, x'-y\rangle+i\langle\eta,G(y)-x \rangle } \rho_{\gamma'}(y)  g(y) \psi_{\gamma'}(\xi)\tilde{\psi}_{\gamma}(\eta) d\xi d\eta dy.\label{eqn:kappa}
\end{align}
As a simple estimate on this kernel, we have
\begin{lemma}\label{lm:kest0} For each $\mu>0$, there exists a constant $C_*>0$ such that 
\[
|\kappa_{\gamma\gamma'}(x',x)|\le C_*\cdot \|g\|_{L^\infty}\cdot \int_{\Ze(\gamma')} b_{\gamma'}^\mu(x'-y)\cdot b_{\gamma}^\mu(G(y)-x) dy
\]
for $(x,x')\in E\times E$, $\gamma, \gamma'\in \Gamma$ and for $G:V'\to V$ in $\mathcal{H}$ and $g\in \cC^r(V')$.
\end{lemma}
\begin{proof} 
The claim follows if we apply Corollary \ref{cor:localized} to the integration with respect to the variable $\xi$ and $\eta$ in (\ref{eqn:kappa}).
\end{proof}
This uniform estimate is quite useful. 
But we need to improve this estimate in some cases. 
In the case where  $DG^*_y(\supp \widetilde{\psi}_\gamma)$ for $y\in \supp \rho_{\gamma'}$ are apart from $\supp \psi_{\gamma'}$, it is natural to expect that the operator norm of $\cL_{\gamma\gamma'}$ is small. To justify this idea, we  use the fact that the term $
e^{i\langle\xi, x'-y\rangle+i\langle\eta,G(y)-x \rangle }$
in (\ref{eqn:kappa}) oscillates fast in such case and therefore the integration with respect to the variable $y$ in (\ref{eqn:kappa}) can be regarded as an oscillatory integral. 

Let us  recall a technique in estimating  oscillatory integrals. (See \cite[\S 7.7]{Hormander1} for more details.) Consider an integral of the form
\begin{equation}\label{eqn:oscl}
\int h(x) e^{i f(x)} dx
\end{equation}
where $h(x)$ is a continuous function supported on a compact subset in $E$ and $f(x)$ a real-valued continuous function defined on a neighborhood of the support of $h$. Take a few vectors $v_1,v_2,\dots,v_k$ in $E$ and regard them as constant vector fields on $E$. Assume that the functions $f$ and $h$ are so smooth that $v_i f$, $v_i v_j f$ and $v_i h$ for  $1\le i,j\le k$ exist and are continuous on a neighborhood of the support of $h$. Assume also that  
\[
v_1(f)^2+v_2(f)^2+\cdots +v_k(f)^2\neq 0\quad \text{on the support of $h$.}
\]
Then we can apply integration by parts to obtain
\[
\int h(x) e^{i f(x)} dx =
\int Lh(x) e^{i f(x)} dx
\]
where
\[
Lh=\sum_{j=1}^{k} 
v_j\left(\frac{i\cdot h\cdot v_j(f)}{\sum_{\ell=1}^{k}v_\ell(f)^2}\right).
\]
This formula tells that if the term $e^{i f(x)}$ oscillates fast in the directions spanned by the vectors $v_1,v_2,\cdots, v_k$, the term $Lh(x)$ will be small and so will be the oscillatory integral (\ref{eqn:oscl}). 

Assuming more smoothness of the functions $f$ and $h$, we may repeat the operation above and obtain the formula 
\begin{equation}\label{eqn:intbypart}
\int h(x) e^{i f(x)} dx =
\int L^\ell h(x) e^{i f(x)} dx.
\end{equation}
Basically we get better estimate if we exploit  this formula for larger $\ell$. This is the point where differentiablity of the flow gets into our argument. 

Below we give a simple estimate on the kernel $\kappa_{\gamma\gamma'}$ applying the formula (\ref{eqn:intbypart}). First we introduce some definitions. 
For integers $n,k,n',k'$ such that $(n,k), (n',k')\in \cN$, we set
\[
\Delta(n,k,n',k')=
\log_2^+ \left(2^{-n'/2}\cdot d(\supp \widetilde\chi_{n, k},\supp\chi_{n', k'})\right)
\]
where $\log_2^+ t=\max\{0, \log t/\log 2\}$. Also we put
\[
\widetilde{\Delta}(n,k,n',k')=
\begin{cases}
0,&\quad\text{if $|n-n'|\le 1$;}\\
\Delta(n,k,n',k'),&\quad\text{otherwise.} 
\end{cases}
\]
Since $\pi_0^*(\supp \chi_{n,k})\subset \pi_0^*(\supp \widetilde\chi_{n, k})\subset [-2^{n+2}, 2^{n+2}]$, we have that 
\begin{equation}\label{eqn:dnk0}
\widetilde{\Delta}(n,k,n',k')\le \Delta(n,k,n',k')\le \max\{n,n'\}-n'/2+2
\end{equation}
in general. If $|n-n'|\ge 2$ and $\max\{n,n'\}\ge 10$, we have also that
\begin{equation}\label{eqn:dnk}
\Delta(n,k,n',k')=\widetilde{\Delta}(n,k,n',k')\ge \max\{n,n'\}-n'/2-3.
 \end{equation}
Hence it holds, in general, that 
 \begin{equation}\label{eqn:nnd}
|n-n'|\le  2\Delta(n,k,n',k')+10.
 \end{equation}
\begin{remark}\label{rem:card}
For each $(n,k)\in \cN$, the cardinality of $(n',k')\in \cN$ such that $\Delta(n,k,n',k')=0$ (resp. $\Delta(n',k', n,k)=0$) is bounded by an absolute constant. 
\end{remark}
Looking into  the definition of $\Delta(n,k,n',k')$ more closely, we see that,  for each $s>1$,  there exists a constant $C_*=C_*(s)>0$ such that 
\begin{equation}\label{eqn:chisum}
\sum_{(n',k')\in \cN} 2^{-s\Delta(n,k,n',k')}<C_* \quad \text{for any $(n,k)\in \cN$}
\end{equation}
and that
\begin{equation}\label{eqn:chisum2}
\sum_{(n,k)\in \cN} 2^{-s\Delta(n,k,n',k')}<C_* \quad \text{for any $(n',k')\in \cN$.}
\end{equation}
For $(\gamma,\gamma')\in \Gamma\times \Gamma$, we will write $\Delta(\gamma,\gamma')$ and $\widetilde{\Delta}(\gamma,\gamma')$  respectively for
\[
\Delta(n(\gamma),k(\gamma), n(\gamma'), k(\gamma'))\quad \text{and}\quad
\widetilde{\Delta}(n(\gamma),k(\gamma), n(\gamma'), k(\gamma')).
\]

\begin{lemma}\label{lm:kest1}
For each $\mu>0$, there exists a constant $C_*>0$ such that 
\[
|\kappa_{\gamma\gamma'}(x',x)|\le C_*\cdot 2^{-r_*\cdot  \Delta(\gamma,\gamma')}\|g\|_{*}\cdot \int_{\Ze(\gamma')} b_{\gamma'}^\mu(x'-y)\cdot b_{\gamma}^\mu(G(y)-x) dy
\]
for any $(\gamma,\gamma')\in \Gamma\times \Gamma$ and any $(x,x')\in E\times E$. 
The constant $C_*$ does not depend on $G:V'\to V$ in $ \mathcal{H}$ nor on $g\in \cC^r(V')$. 
\end{lemma}
\begin{proof}
We suppose $\Delta(\gamma,\gamma')>0$, since the conclusion is a consequence of Lemma \ref{lm:kest0} otherwise. 
By definition, the condition $\Delta(\gamma,\gamma')>0$ implies 
\[
|\pi^*_0(\eta-\xi)|\ge 2^{n(\gamma')/2+\Delta(\gamma,\gamma')}
\quad \mbox{whenever $\xi\in \supp \psi_{\gamma'}$ and $\eta\in \supp \widetilde{\psi}_\gamma$.}
\]
Apply the formula (\ref{eqn:intbypart}) to the integral  with respect to $y$ in (\ref{eqn:kappa}), setting  $\ell=r_*$, $k=1$ and $\{v_j\}_{j=1}^{k}=\{v_0\}$.  Then we obtain the expression
\begin{equation}\label{eqn:ker2}
\kappa_{\gamma\gamma'}(x',x)=\int 
\left(\int e^{i \langle \xi, x'-y\rangle-i\langle \eta, G(y)-x\rangle} R(y, \xi, \eta) d\eta d\xi \right)dy
\end{equation}
where
\begin{equation}\label{eqn:RR}
R(y, \xi, \eta)= 
\frac{ i^{r_*}\cdot
 v_0^{r_*}(\rho_{\gamma'}(y) g(y))\cdot  \psi_{\gamma'}(\xi) \cdot \widetilde{\psi}_{\gamma}(\eta)}{(2\pi)^{2(2d+1)}(\pi^*_0(\eta-\xi))^{r_*}}.
\end{equation}
From Lemma \ref{lm:scaling}, there exists a constant $C_\alpha>0$  for each  $\alpha\in \integer_+^{2d+1}$, which does not depend on $\gamma$, such that
\[
\|D^\alpha\psi_\gamma\|_{L^\infty}<C_\alpha 2^{-|\alpha| n(\gamma)/2-|\alpha|_\dag |m(\gamma)|},\;\;
\|D^\alpha\widetilde{\psi}_\gamma\|_{L^\infty}<C_\alpha 2^{-|\alpha| n(\gamma)/2-|\alpha|_\dag |m(\gamma)|}
\]
where $|\alpha|_\dag$ is that defined in (\ref{eqn:sdag}). 
This and the estimate on $|\pi^*_0(\eta-\xi)|$ above imply that, for any multi-indices $\alpha, \beta\in \integer_+^{2d+1}$, 
there exists a constant $C_{\alpha\beta}>0$, which does not depend on $G\in \mathcal{H}$, $g\in \cC^r(V')$ nor on $(\gamma,\gamma')\in \Gamma\times \Gamma$, such that
\[
\|\partial^\alpha_{\xi} \partial^\beta_{\eta} R\|_{L^\infty} \le C_{\alpha \beta}\cdot \|g\|_{*}
\cdot 2^{-r_*\cdot  \Delta(\gamma, \gamma')-|\alpha| n(\gamma)/2-
|\alpha|_\dag |m(\gamma)|- |\beta| n(\gamma')/2-|\beta|_\dag 
|m(\gamma')|}.
\]
Again from Lemma \ref{lm:scaling}, the $(2d+1)$-dimensional volumes of the supports of $\psi_\gamma$ and $\widetilde{\psi}_\gamma$ are bounded by
$C_*\cdot 2^{(2d+1) n(\gamma)/2+ 2d |m(\gamma)|}$
with $C_*$ a constant independent of $\gamma$. 
Therefore we have
\begin{align*}
&\left|(x'-y)^\alpha \cdot (G(y)-x)^\beta\cdot \int e^{i \langle \xi, x'-y\rangle-i\langle \eta, G(y)-x\rangle} R(y, \xi, \eta) d\eta d\xi \right|\\
&\qquad =\left| \int e^{i \langle \xi, x'-y\rangle-i\langle \eta, G(y)-x\rangle} \partial^\alpha_\xi \partial^\beta_\eta R(y, \xi, \eta) d\eta d\xi \right|\\
&\qquad \le C_{\alpha \beta}\cdot C_*\cdot  \|g\|_{*}\cdot  2^{(2d+1) n(\gamma)/2+ 2d |m(\gamma)|}\\
&\qquad \qquad \quad \cdot 2^{-r_*\cdot  \Delta(\gamma, \gamma')-|\alpha| n(\gamma)/2-
|\alpha|_\dag |m(\gamma)|- |\beta| n(\gamma')/2-|\beta|_\dag 
|m(\gamma')|}
\end{align*}
for any multi-indices $\alpha$ and $\beta$. This implies that
the integral with respect to $\xi$ and $\eta$ in the bracket $(\cdot)$ in (\ref{eqn:ker2})  is bounded  by
\[
 C_*
\cdot 2^{-r_*\cdot  \Delta(\gamma, \gamma')}\|g\|_{*}\cdot b_{\gamma'}^\mu(x'-y)\cdot b_{\gamma}^\mu(G(y)-x)
\]
in absolute value.
Since the integral  vanishes when $y\notin \supp \rho_{\gamma'}\subset Z(\gamma')$, we obtain the claim of the lemma. 
\end{proof}

\section{Preliminary discussion to the proof of Theorem  
\ref{th:reduced}}\label{sec:pre}

In this section, we give preliminary discussion to  the proof of Theorem~\ref{th:reduced}.
For brevity, we henceforth write $\cM$ and $\cL$ respectively for $\cM(G,g)$ and $\cL(G,g)$, though we keep in mind  dependence of $\cM$ and $\cL$ on $G$ and $g$.

\subsection{The compact, central and hyperbolic part of $\cM$}\label{ssec:decoM}
In the proof of Theorem \ref{th:reduced}, we divide  the operator $\cM$ into five parts and consider each parts separately. To this end,
we divide the product set $\Gamma\times \Gamma$ into five disjoint subsets $\cR(j)$ for $0\le j\le  4$ and define the corresponding part $\cM_j:\bB^\beta_{\nu}\to \bB^\beta_{\nu'}$ of $\cM$  formally by
\begin{equation}\label{eqn:defM}
\cM_j((u_{\gamma})_{\gamma\in \Gamma})=\left(\sum_{\gamma: (\gamma,\gamma')\in \cR(j)} \cL_{\gamma\gamma'}(u_\gamma)\right)_{\gamma'\in \Gamma}.
\end{equation}

The definition of the part $\cM_0$ is simple. 
Let $K\ge 0$ be a large constant, which will be determined in the course of the proof, and set 
\[
\cR(0)=\{(\gamma,\gamma')\in \Gamma\times \Gamma\mid \max\{n(\gamma), |m(\gamma)|, n(\gamma'), |m(\gamma')|\}\le K\}.
\]
The corresponding part $\cM_0$ defined by (\ref{eqn:defM}) for $j=0$ is called the compact part of $\cM(G,g)$. This is because  we have
\begin{proposition}\label{lm:m0}
The formal definition of the operator $\cM_{0}$ gives a compact operator $\cM_{0}:\bB^\beta_\nu\to \bB^\beta_{\nu'}$ for any $\nu, \nu'\ge 2d+2$. 
\end{proposition}
\begin{proof}
For $\gamma\in \Gamma$, let $L^2(E; d_\gamma^\nu)$ be the Hilbert space of  functions $u\in L^2(E)$ such that $\|d_\gamma^\nu\cdot u_\gamma\|_{L^2}<\infty$, equipped with the obvious norm.  
Then  
\begin{equation}\label{eqn:Lgg}
\cL_{\gamma\gamma'}:L^2(E; d_\gamma^\nu) \to L^2(E;d_{\gamma'}^{\nu'})
\end{equation}
 is a compact operator, because its kernel (\ref{eqn:kappa}) is  smooth and decays rapidly as we saw in Lemma \ref{lm:kest0}. Since $\cR(0)$ contains only finitely many elements by the definition of $\Gamma$, the statement follows immediately. 
\end{proof}

The part $\cM_0$ will turn out to be the compact operator $\mathcal{K}(G,g)$ in the latter statement of Theorem~\ref{th:reduced}.

The definition of the part $\cM_1$ is also simple. 
Let $0<\delta<1/10$ be a constant that we will fix soon below. For given $\lambda>0$, we set
\begin{align*}
&\cR(1)=\cR(1;\lambda)\\
&=\{(\gamma,\gamma')\in \Gamma\times \Gamma\setminus \cR(0)\mid \max\{|m(\gamma)|, |m(\gamma')|\}\le \delta \lambda,\, |n(\gamma)-n(\gamma')|\le 1\}.
\end{align*}
The corresponding part $\cM_1$ is called the central part  of $\cM$.
The remaining part is called hyperbolic part and will be divided into three parts.

\subsection{Setting of constants}\label{ssec:remarks}
In the proof, we set up constants as follows.
We henceforth suppose that $0<\beta<(r-1)/2$ and $\epsilon>0$  in the statement of Theorem \ref{th:reduced} are fixed.  
We first choose $0<\delta<1/10$ so small that
\[
(2\beta +5d+2)\delta <\epsilon.
\]
Then we choose $\nu_*$, $\lambda_*$ and $\Lambda_*$ in the conclusion of Theorem \ref{th:reduced} so large that 
\[
\nu_*\ge 6(\beta/\delta+d+1)
\]
and that
\[
\lambda_*>40, \quad 2^{\delta \lambda_*-10}\ge 10^2\sqrt{2d+1}, \quad \Lambda_*\ge d\lambda_*.
\]
The conditions in the choice above are  technical ones.  The readers should not care about them too much at this stage. We present them only to emphasize that the choices are explicit.

Once we set up the  constants $\delta$, $\nu_*$, $\lambda_*$ and $\Lambda_*$ as above, we take $\lambda\ge \lambda_*$ and $\Lambda\ge \Lambda_*$ such that $\Lambda\ge d\lambda$ and then take an arbitrary diffeomorphism \hbox{$G:V'\to V$} in $\cH(\lambda, \Lambda)$ and an arbitrary function $g$ in $\cC^r(V')$. This is the setting in which most of the argument in the following sections is developed. 

The readers should be aware that the choice of the constant $K>0$ in the definition of $\cR(0)$ is not mentioned above. 
We will choose the constant $K$ in the course of the proof and the choice will depend on the diffeomorphism $G$ and the function~$g$ besides $\lambda$ and~$\Lambda$. 
This does not cause any problem because Proposition~\ref{lm:m0} holds regardless of the choice of $K$. 
In the proof, we understand that the constant $K$ is taken so large that the argument holds true and will not mention the choice of $K$ too often.

In the proof, it is important to distinguish the class of constants that are independent of the diffeomorphism $G:V'\to V$ in $\cH(\lambda,\Lambda)$, the function $g:V'\to \real$ in $\cC^r(V')$ and the choice of $\lambda$ and $\Lambda$.
To this end, we  use a generic symbol $C_*$ for such class of constants.  On the contrary, we  use a generic symbol $C(G,g)$ (resp. $C(G)$) for   
constants that may depend on $G$ and $g$ (resp. on $G$) and also on  $\lambda$ and $\Lambda$ (resp. on $\lambda$). 
Notice that the real value of constants denoted by $C_*$, $C(G,g)$ and $C(G)$ may change from places to places in the argument. 

\subsection{Norms on $\bB^\beta_\nu$}
In the proof, we consider the following family of norms on $\bB^\beta_\nu$ for $\lambda>0$,  rather than the original norm $\|\cdot \|_{\beta,\nu}$ in the definition:
\[
\|\bu\|_{\beta,\nu}^{(\lambda)}=\left(\sum_\gamma 
\wt(m(\gamma))^2\cdot  \|d_\gamma^\nu \cdot u_\gamma\|_{L^2}^2\right)^{1/2}
\quad \text{ for $\bu=(u_\gamma)_{\gamma\in \Gamma}\in \bB^\beta_\nu$,}
\]
where
\begin{equation}\label{eqn:wl}
\wt(m)=\begin{cases}
2^{\beta (m+2\lambda)},&\text{if $m> \delta \lambda$};\\
1 ,&\text{if $|m|\le \delta \lambda$};\\
2^{\beta (m-2\lambda)},&\text{if $m< -\delta \lambda$}.\\
\end{cases}
\end{equation}
This family of norms are all equivalent to the original norm $\|\cdot\|_{\beta,\nu}$ because
\begin{equation}\label{eqn:wt}
2^{\beta (m-2\lambda)}\le \wt(m)\le 2^{\beta (m+2\lambda)}.
\end{equation}
The family of norms  $\|\cdot \|_{\beta,\nu_*}^{(\lambda)}$ will turn out to be the norms $\|\cdot\|^{(\lambda)}$ in the latter statement of Theorem \ref{th:reduced}. 

\begin{remark}In reading the proof in the following sections, it is a good idea  to ignore the pairs $(\gamma,\gamma')$ with $|n(\gamma)-n(\gamma')|\ge 2$ provisionally. Lemma~\ref{lm:kest1} tells that the operators $\cL_{\gamma\gamma'}$ for such pairs are very small and basically negligible. Also it may be helpful to consider the case where $G$ is a linear map and $g$ is a constant function first. Then the reader will find that a good part of the argument is rather obvious or  simple in such case.  
\end{remark}

\section{The hyperbolic parts of the operator $\cM$ (I)}\label{sec:body}
In this section and the following two sections, we consider the hyperbolic part of the operator $\cM$. 
We divide it into three parts, namely, $\cM_2$, $\cM_3$ and $\cM_4$, and estimate the operator norms of each part separately. The rough idea in this division is as follows. 
From the definition of the operator $\cL_{\gamma\gamma'}$, we naturally expect that the operator norm of  $\cL_{\gamma\gamma'}$ should be small if either
\begin{itemize}
\item[(A)]  $G(z(\gamma'))$ is apart from  $z(\gamma)$, or 
\item[(B)]  $DG^*_y(\supp \widetilde{\psi}_\gamma)$ for  $y\in \supp \rho_{\gamma'}$ are apart from $\supp \psi_{\gamma'}$. 
\end{itemize}
Roughly, $\cM_3$ and $\cM_4$ consist of components $\cL_{\gamma\gamma'}$ for pairs $(\gamma, \gamma')$ in the case  (A) and (B) respectively. 
We will in fact  prove that the operator norms of $\cM_3$ and $\cM_4$ are small in Section \ref{sec:tail1} and \ref{sec:tail2}. 
The remaining components $\cL_{\gamma\gamma'}$ are assigned to
the part $\cM_2$. The operator $\cM_2$ gives raise to the factor $2^{-\beta \lambda}$ in the claim of Theorem~\ref{th:reduced}.  

\subsection{The operator $\cM_2$}
We  first define the operator $\cM_2$ as follows. 
\begin{definition}Let $\cR(2)$ be the set of pairs $(\gamma, \gamma')\in \Gamma\times \Gamma\setminus (\cR(0)\cup \cR(1))$ such that
$n=n(\gamma)$, $k=k(\gamma)$, $m=m(\gamma)$, $n'=n(\gamma')$, $k'=k(\gamma')$ and  $m'=m(\gamma')$ satisfy at least one of  the following two conditions:
\begin{itemize}
\setlength{\itemsep}{4pt}
\item[(a)] 
$m'<m-\lambda+ 10 \widetilde{\Delta}(n,k,n', k')+20$, and
\item[(b)] $|n-n'|\le 1$ and either $m'< -\delta \lambda\le m$ or $m'\le \delta \lambda<m$.
\end{itemize}
Let $\cM_2$ be the part defined  formally by (\ref{eqn:defM}) with $j=2$.
\end{definition}
For the operator $\cM_2$, we have
\begin{proposition}\label{pp:m2} 
The formal definition of the operator $\cM_2$ in fact gives a bounded operator $\cM_2:\bB^\beta_\nu\to \bB^\beta_{\nu'}$ for any $\nu,\nu'\ge 2d+2$. Further, for any $\nu,\nu'\ge 2d+2$, there is a constant $C_*>0$  such that we have
\[
\left\|\cM_2(\bu)
\right\|_{\beta,\nu'}^{(\lambda)}
\le C_*\cdot \|g\|_{*}\cdot   2^{-\beta \lambda}\cdot 
\left\|\bu\right\|_{\beta,\nu}^{(\lambda)}\quad \text{for  $\bu\in \bB^\beta_\nu$,}
\]
for $G:V'\to V$ in $\cH(\lambda,\Lambda)$ and $g\in \cC^r(V')$ provided  $\lambda\ge \lambda_*$ and $\Lambda\ge \Lambda_*$.
\end{proposition} 

\noindent{\it Proof.}\;
For a combination $(n,k,m,n',k',m')\in (\cN\oplus \integer)^2$, we set 
\begin{equation}\label{eqn:K}
K_{n,k,m,n',k',m'}=2^{-r_*\cdot  \Delta(n,k,n',k')}\cdot  \|g\|_*\cdot 
\frac{\wt(m')}{\wt(m)}.
\end{equation}
We need the following sublemma of combinatorial nature, whose proof is postponed for a while. 
\begin{sublemma}\label{subl:1} There exists a constant $C_*>0$ such that
\begin{align}\label{eqn:ksumest1}
&\sup_{(n',k',m')\in \cN\oplus \integer}\left(\sum_{n,k,m:n',k',m'}K_{n,k,m,n',k',m'} \right)<C_*\|g\|_*
\cdot  2^{-\beta\lambda}
\intertext{and}
&\sup_{(n,k,m)\in \cN\oplus \integer}\left(\sum_{n',k',m':n,k,m}K_{n,k,m,n',k',m'} \right)<C_*\|g\|_*\cdot  2^{-\beta\lambda}\label{eqn:ksumest2}
\end{align}
where $\sum_{n',k',m':n,k,m}$ {\rm (resp. $\sum_{n,k,m:n',k',m'}$)} denotes the sum over $(n',k',m')$ 
{\rm (resp. $(n,k,m)$)} in $\cN\oplus \integer$ such that the combination $(n,k,m,n',k',m')$ satisfies 
\begin{equation}\label{eqn:max}
\max\{n, |m|, n', |m'|\}> K,
\end{equation}
and at least one  of the conditions (a) and (b)  in the definition of $\cR(2)$. 
\end{sublemma}
For $(n,k,m)\in \cN \oplus \integer$, we set 
\begin{equation}\label{eqn:v}
v_{n,k,m}(x)=\left(\sum_{\gamma:n,k,m}d_\gamma^{2\nu}(x)\cdot |u_\gamma (x)| ^2\right)^{1/2}
\end{equation}
where $\sum_{\gamma:n,k,m}$ denotes the sum over $\gamma\in \Gamma$ such that $n(\gamma)=n$, $k(\gamma)=k$ and $m(\gamma)=m$. 
Then we have, by Schwarz inequality, that
\begin{equation}\label{eqn:sch}
\sum_{\gamma:n,k,m}|u_\gamma (x)|
 \le \left(\sum_{\gamma:n,k,m} d_{\gamma}^{-2\nu}(x)\right)^{1/2}
 \cdot v_{n,k,m}(x)\le  C_* \cdot v_{n,k,m}(x).
\end{equation}
From  Lemma \ref{lm:kest1} for $\mu=\nu'+2d+2$, we have the following estimate on the kernel $\kappa_{\gamma\gamma'}$ of the operator $\cL_{\gamma\gamma'}$:
\[
|d_{\gamma'}^{\nu'}(x')\cdot\kappa_{\gamma\gamma'}(x',x)|\le C_*\cdot 2^{-r_*\cdot  \Delta(\gamma,\gamma')}\|g\|_{*}\cdot \int_{\Ze(\gamma')} b_{\gamma'}^{2d+2}(x'-y)\cdot b_{\gamma}^\mu(G(y)-x) dy.
\]
Hence, by Young inequality, we obtain 
\[
\left\|\sum_{\gamma:n,k,m} d_{\gamma'}^{\nu'} \cL_{\gamma\gamma'} u_\gamma\right\|_{L^2} \!\! \le C_* \left|K_{n,k,m,n',k',m'} \frac{\wt(m)}{\wt(m')}\right|
\left\| b^\mu_{n,m}* v_{n,k,m}|_{G(Z(\gamma'))} \right\|_{L^2}
\]
for  $\gamma'\in \Gamma$ such that $n(\gamma')=n'$, $k(\gamma')=k'$ and $m(\gamma')=m'$. 
Since the intersection multiplicity of  $Z(\gamma')$
for $\gamma'\in \Gamma$ such that $n(\gamma')=n'$, $k(\gamma')=k'$ and $m(\gamma')=m'$ is bounded by some constant depending only on~$d$, it follows
\begin{equation}\label{eqn:ck}
\sum_{\gamma':n',k',m'}\left\|\sum_{\gamma:n,k,m} d_{\gamma'}^{\nu'} \cL_{\gamma\gamma'} u_\gamma\right\|_{L^2}^2  \le C_* \left|K_{n,k,m,n',k',m'} \frac{\wt(m)}{\wt(m')}\right|^2
\left\|  v_{n,k,m} \right\|_{L^2}^2.
\end{equation}
For $\bu=(u_\gamma)_{\gamma\in \Gamma}\in \bB^\beta_\nu$, 
 we have  by definition that 
\begin{align*}
&\left(\| \cM_2(\bu)
\|_{\beta,\nu'}^{(\lambda)}\right)^2=\sum_{\gamma'}   \wt(m')^2
\cdot \| d_{\gamma'}^{\nu'}\cdot \cM_2(\bu)_{\gamma'}\|_{L^2}^2\\
&\qquad=\sum_{n',k',m'}\;\;\sum_{\gamma':n',k',m'} \wt(m')^2
\left\|\,d_{\gamma'}^{\nu'}\cdot  \left(\sum_{n,k,m:n',k',m'}\;\sum_{\gamma:n,k,m} 
\cL_{\gamma\gamma'} u_\gamma\right)\right\|^2_{L^2}.
\end{align*}
From (\ref{eqn:ksumest1}) and Schwarz inequality, this is bounded by
\[
\sum_{n',k',m'}\;\sum_{\gamma':n',k',m'}
 \wt(m')^2
\sum_{n,k,m:n',k',m'}\frac{C_*\|g\|_*\cdot  2^{-\beta\lambda}}{K_{n,k,m,n',k',m'}}\left\|\sum_{\gamma:n,k,m} 
d_{\gamma'}^{\nu'} \cL_{\gamma\gamma'} u_\gamma\right\|^2_{L^2}
\]
and hence by 
\begin{align*}
&\sum_{n',k',m'}\;
\sum_{n,k,m:n',k',m'}
\frac{C_*\|g\|_*\cdot  2^{-\beta\lambda}\cdot \wt(m')^2}{K_{n,k,m,n',k',m'}}
\sum_{\gamma':n',k',m'}\left\|\sum_{\gamma:n,k,m} 
d_{\gamma'}^{\nu'} \cL_{\gamma\gamma'} u_\gamma\right\|^2_{L^2}\\
&\le 
\sum_{n,k,m}\;\;\sum_{n',k',m':n,k,m}
C_*\|g\|_*\cdot 2^{-\beta\lambda}\cdot  K_{n,k,m,n',k',m'}\cdot 
 \wt(m)^2\cdot
\|v_{n,k,m}\|^2_{L^2}\\
&\le 
C_*\|g\|_*^2\cdot  2^{-2\beta\lambda}\cdot
\sum_{n,k,m} 
 \wt(m)^2\cdot
\|v_{n,k,m}\|^2_{L^2}= C_*\|g\|_*^2\cdot  2^{-2\beta\lambda}\cdot(\|\bu\|_{\beta,\nu}^{(\lambda)})^2,
\end{align*}
where the first inequality follows from (\ref{eqn:ck}) and the second from  (\ref{eqn:ksumest2}). Thus the conclusion of Proposition \ref{pp:m2} holds. 

We now complete the proof by proving  Sublemma \ref{subl:1}.
\begin{proof}[Proof of Sublemma \ref{subl:1}] 
In the argument below, we consider combinations $(n,k,m,n',k',m')\in (\cN\times \integer)^2$ satisfying (\ref{eqn:max}) and at least one of the conditions (a) and (b) in the definition of $\cR(2)$. 
And we will further restrict ourselves to the cases (I)  $|n-n'|\le 1$ and (II) $|n-n'|\ge 2$ in turn and prove the claims (\ref{eqn:ksumest1}) and (\ref{eqn:ksumest2}) with the sums replaced by the partial sums restricted to such cases.
This is of course enough for the proof of the sublemma. 

Let us first consider the case (I). 
Suppose that the condition (a) in the definition of $\cR(2)$ holds in addition.
Since $\widetilde{\Delta}(n,k,n',k')=0$ in the case (I) by definition, we have $m'<m-\lambda+20<m-2\delta\lambda$ from the choice of $\delta$ and $\lambda_*$. Hence, recalling the definition of $w^{(\lambda)}(m)$ in (\ref{eqn:wl}),  we have 
\begin{equation}\label{eqn:genmmb}
K_{n,k,m,n',k',m'}\le 2^{\beta(m'-m)-r_* \Delta(n,k,n',k')}\|g\|_*.
\end{equation}
Next, suppose that  the condition (b) in the definition of $\cR(2)$ holds.
 Then we have  $m'<m$ and hence 
\begin{equation}
K_{n,k,m,n',k',m'}\le 2^{-2\beta\lambda +\beta(m'-m-2\delta \lambda)-r_* \Delta(n,k,n',k')}\|g\|_*. \label{eqn:genmmc}
\end{equation} 
Therefore, considering each of these two subcases (a) and (b) separately and using (\ref{eqn:chisum}) and (\ref{eqn:chisum2}), we obtain 
the required inequalities for the partial sums. 

Let us consider the case  (II).  
Note that  the condition (a) in the definition of $\cR(2)$ holds for combinations $(n,k,m,n',k',m')$ in this case and we have  $\widetilde{\Delta}(n,k,n',k')={\Delta}(n,k,n',k')$ from the definition.  Let us consider three subcases:
\begin{itemize}
\item[(i)] 
The subcase where  $\max\{n,n'\}\le   K/100$
 and   $m$ and $m'$ are on the same side of the interval $[-\delta\lambda, \delta\lambda]$. In this subcase,  we have 
 (\ref{eqn:genmmb}), which can be written as
\[
K_{n,k,m,n',k',m'}\le 2^{\beta(m'-m-10\Delta(n,k,n',k') )-
(r_*-10\beta )  \Delta(n,k,n',k')}\|g\|_*.
\]
\item[(ii)]
The subcase where   
$\max\{n,n'\}\le   K/100$
 and  $m$ and $m'$ are {\em not} on the same side of the interval  $[-\delta\lambda, \delta\lambda]$. In this subcase, we have  $
m'-m\le -K/2$, because $\max\{|m|, |m'|\}\ge K$ from  (\ref{eqn:max}) and because $\widetilde{\Delta}(n,k,n',k')\le K/50$ from (\ref{eqn:dnk0}).
\item[(iii)]
The subcase where $\max\{n,n'\}>   K/100$. In this subcase, we have
 $\widetilde{\Delta}(n,k,n',k')\ge K/200-3$
from (\ref{eqn:dnk}).
\end{itemize}
Note that we have 
\begin{equation*}
K_{n,k,m,n',k',m'}\le 
2^{\beta(m'-m)- r_* \Delta(n,k,n',k')+4\beta\lambda}\|g\|_*
\end{equation*}
(that holds in general) 
in the latter two subcases. 
Consider each of the three subcases above separately.
Then, by  using (\ref{eqn:chisum}), (\ref{eqn:chisum2}) and the condition (a) in the definition of $\cR(2)$, it is easy to obtain 
the required inequalities for the partial sums, provided that we take sufficiently large constant $K$. 
\end{proof}

\subsection{A dichotomy in the remaining case}
In this subsection, we prove a lemma which tells roughly that each pair $(\gamma, \gamma')$ that belongs to neither of $\cR(j)$ for $j=0,1,2$ falls into either of the situation (A) or (B) mentioned in the beginning of this section.  
First of all, we note that a pair $(\gamma,\gamma')\in \Gamma\times \Gamma$ belongs to neither of  $\cR(0)$, $\cR(1)$ or $\cR(2)$
  if and only if 
 $n=n(\gamma)$, $k=k(\gamma)$, $m=m(\gamma)$, $n'=n(\gamma')$, $k'=k(\gamma')$ and $m'=m(\gamma')$ satisfy the conditions
\begin{itemize}
\item[(R1)]  $\max\{n,n', |m|, |m'|\} > K$, 
\item[(R2)] $\max\{|m|, |m'|\}>\delta \lambda$ if $|n-n'|\le 1$, 
\item[(R3)] $m'\ge m-\lambda+10 \widetilde{\Delta}(n,k,n',k')+20$, and
\item[(R4)] neither $m'< -\delta \lambda\le m$ nor $m'\le \delta \lambda<m$ if $|n-n'|\le 1$.
\end{itemize}
For convenience in the later argument, we list the following immediate consequences of (R1)-(R4):
\begin{itemize}
\item[(R5)] $m'\ge m-\lambda+20$, 
\item[(R6)] either $m<0$ or $m'> 0$, 
\item[(R7)] if $|n-n'|\le 1$, we have $\max\{-m,m'\}\ge \delta \lambda$, 
\item[(R8)] if $|n-n'|\ge 2$, we have
\[
\max\{-m,m'\}\ge 2 \max\{n,n'\}\;\;\text{and}\;\;
\max\{-m,m'\}\ge K/100.
\]
\end{itemize}
\begin{proof}[Proof of (R5)--(R8)]
 (R5) follows from (R3), and (R6) follows from (R7) and (R8). 
(R7) follows from (R2) and (R4).  
If $\max\{n,n'\}\ge K/100$,  (R8) follows from (R3) and (\ref{eqn:dnk}). Otherwise we have
$\max\{|m|, |m'|\}\ge K$ from (R1) and hence $\max\{-m,m'\}\ge K/2$ from (R5), which implies (R8).
\end{proof}
Next we give a few definitions in order to state the next lemma. 
For a pair $(\gamma, \gamma')\in \Gamma\times \Gamma$ that belongs to neither of $\cR(j)$ for $j=0,1,2$, we set 
\[
D(\gamma,\gamma')=D(n,m,n',m')\quad \text{and}\quad
 \widetilde{D}(\gamma,\gamma')=\widetilde{D}(n,m,n',m')
\]
where $n=n(\gamma)$, $m=m(\gamma)$, $n'=n(\gamma')$ and $m'=m(\gamma')$ and\footnote{Because of (R6), we do not consider the case ($m\ge0$ and $m'\le 0$). } 
\[
D(n,m,n',m')=
\begin{cases}
m'+n'/2,&\text{if $m\ge 0$, $m'> 0$;}\\
-m+n/2+\lambda,&\text{if $m< 0$, $m'< 0$;}\\
\max\{-m+n/2+\lambda,m'+n'/2\},&\text{if $m< 0$, $m'\ge 0$,}
\end{cases}
\]
and
\[
\widetilde{D}(n,m,n',m')\!=\!
\begin{cases}
m'+n'/2-n+\lambda,&\text{if $m\ge 0$, $m'> 0$;}\\
-m-n/2,&\text{if $m< 0$, $m'< 0$;}\\
\max\{-m-n/2,m'+n'/2-n+\lambda\},&\text{if $m< 0$, $m'\ge 0$.}
\end{cases}
\]
Let $\Pi_{z}:E^*\to E^*_+\oplus E^*_-$ 
 be the projection along  the line $\langle \alpha_0(z)\rangle$ spanned by $\alpha_0(z)$. Then we have, from the definition of $\alpha_0$, that
\begin{equation}\label{eqn:Pi}
\|\Pi_{z}(\xi)-\Pi_{z'}(\xi)\|\le |\pi_0^*(\xi)|\cdot  \|z-z'\|\quad \text{ for $\xi\in E^*$ and $z,z'\in E$.}
\end{equation}
Recall $z(\gamma)$ and $Z(\gamma)$ defined in Subsection \ref{ss:poe}. We show
\begin{lemma}\label{lm:division}
If $
d(G(\Ze(\gamma')), z(\gamma))\le 2^{\widetilde{D}(\gamma,\gamma')-10}$ 
for a pair $(\gamma, \gamma')\in \Gamma\times \Gamma$ that belongs to neither of $\cR(j)$ for $j=0,1,2$,
we have
\begin{equation}\label{eqn:cldiv}
d(\Pi_{z(\gamma')}(\supp{\psi}_{\gamma'}),\Pi_{z(\gamma')}(DG^{*}_{y}(\supp \widetilde{\psi}_{\gamma})))\ge 
2^{D(\gamma,\gamma')-10}
\end{equation}
for all $ y\in \Ze(\gamma')$. 
Further, if  
\begin{equation}\label{eqn:quarter}
\max\{|m(\gamma)|, |m(\gamma')|\}\le \max\{n(\gamma), n(\gamma')\}/4
\end{equation}
in addition, we have (\ref{eqn:cldiv}) for all $y\in E$ such that $\|y-z(\gamma')\|< 2^{-n(\gamma)/3}$. 
\end{lemma}

\begin{proof}[Proof of Lemma \ref{lm:division}]
Take $(\gamma, \gamma')\in (\Gamma\times \Gamma)\setminus \cup_{j=0}^2 \cR(j)$ and set
 $n=n(\gamma)$, $k=k(\gamma)$, $m=m(\gamma)$, $n'=n(\gamma')$, $k'=k(\gamma')$ and $m'=m(\gamma')$. We first prove the following claim.
\begin{sublemma}\label{subl:a} If $w\in \Ze(\gamma')$ satisfies $
d(G(w), z(\gamma))\le 2^{\widetilde{D}(n,m,n',m')-8}$,
we have that $
d(\Pi_{z(\gamma')}(\supp{\psi}_{\gamma'}),\Pi_{z(\gamma')}(DG^{*}_{w}(\supp \widetilde{\psi}_{\gamma})))\ge 
2^{D(n,m,n',m')-8}$. 
\end{sublemma}
\begin{proof} We prove the claim only in the case $m\ge 0$ and $m'> 0$. The proofs in the other cases are similar and left to the readers. 
Note that we have
\[
\Pi_{z(\gamma)}(\supp \widetilde{\psi}_\gamma)
=\Pi^*_{+,-}\left(\supp \widetilde{\psi}_{n,k,m}\right)
\]
from the relation $\Phi_z(\alpha_0(z))=\alpha_0(0)$. (Recall the definition of $\widetilde{\psi}_\gamma$.) Thus $\Pi_{z(\gamma)}(\supp \widetilde{\psi}_\gamma)$ is contained in the disk in $E_+^*\oplus E_-^*$ with center at the origin and radius $2^{m+n/2+2}$.
By (\ref{eqn:Pi}), 
the Hausdorff distance between the subsets
$\Pi_{G(w)}(\supp \widetilde{\psi}_\gamma)$ and 
$\Pi_{z(\gamma)}(\supp \widetilde{\psi}_\gamma)$ is bounded by 
\[
2^{n+2}\cdot d(G(w),z(\gamma))\le 2^{n+2+\widetilde{D}(n,m,n',m')-8}= 
2^{m'+n'/2+\lambda-6}.
\]
Hence the subset $\Pi_{G(w)}(\supp\widetilde{\psi}_\gamma)$ is contained in the disk $\disk_{+,-}^*(R)$ in the subspace $E^*_+\oplus E^*_-$ with center at the origin and radius 
\[
R=2^{m'+n'/2+\lambda-5}\ge 2^{m+n/2+2}+2^{m'+n'/2+\lambda-6}
\]
where the inequality is a consequence of the condition (R3) and (\ref{eqn:nnd}). 

Since $G$ preserves the contact form $\alpha_0$, we have that
\[
\Pi_{w}(DG^*_w(\supp\widetilde{\psi}_{\gamma})) =DG^*_w(\Pi_{G(w)}(\supp\widetilde{\psi}_{\gamma})). 
\]
Note that the condition (H2) and (H3) in the definition of $\cH(\lambda,\Lambda)$ implies that $\|(DG^*_z)(\xi)\|\ge 2^{\lambda}
\|\xi\|$ for $\xi\in (E_+^*\oplus E_-^*)\setminus \cone^*_-(1/10)$.
Therefore the  subset $\Pi_{w}(DG^*_w(\supp\widetilde{\psi}_{\gamma}))$ is contained in $\disk_{+,-}^*
(2^{-\lambda} R)\cup \cone^*_-(1/10)$.

Again by (\ref{eqn:Pi}),  the Hausdorff distance between 
 $\Pi_{z(\gamma')}(DG^*_w(\supp\widetilde{\psi}_{\gamma}))$  and  $\Pi_{w}(DG^*_w(\supp\widetilde{\psi}_{\gamma}))$ is bounded by 
 \[
2^{n+2} \cdot d(w,z(\gamma'))\le \sqrt{2d+1}\cdot 2^{n-n'/2+3}.
\] 
If we set 
\begin{align*}
R'&=2^{-\lambda} R+ 10^2 \sqrt{2d+1}\cdot 2^{n-n'/2+3},
\end{align*}
we find that  $\disk_{+,-}^*
(R')\cup \cone^*_-(2/10)$ contains the $\sqrt{2d+1}\cdot 2^{n-n'/2+3}$ neighborhood of 
$\disk_{+,-}^*(2^{-\lambda} R)\cup \cone^*_-(1/10)$ by  elementary geometric consideration. 
Therefore 
$\Pi_{z(\gamma')}(DG^*_w(\supp\widetilde{\psi}_{\gamma}))$ is contained in $\disk_{+,-}^*
(R')\cup \cone^*_-(2/10)$.

On the other hand, the subset 
$\Pi_{z(\gamma')}(\supp{\psi}_{\gamma'})=\Pi^*_{+,-}(\supp \psi_{n',k',m'})$ is contained in $\cone^*_+(6/10)$ and bounded away from the disk $\disk_{+,-}^*(2^{m'+n'/2-1})=\disk_{+,-}^*(2^{-\lambda+4}R)$ by definition. Thus the claim follows if we prove
\begin{equation*}
10^2 \sqrt{2d+1}\cdot 2^{n-n'/2+3} \le 2^{-\lambda}R=2^{m'+n'/2-5}.
\end{equation*}
If $|n-n'|\le 1$, this follows from (R7) and the choice of $\lambda_*$. Otherwise this  follows from (R8), provided that $K$ is sufficiently large. 
\end{proof}

Now we prove Lemma \ref{lm:division} by using the sublemma above. 
Let us first consider the case where 
(\ref{eqn:quarter}) holds.
Note that  we have $\max\{n, n'\}\ge K$ from (R1) and $|n'-n|\le 1$ from (R8). Corollary~\ref{cor:local} tells that
\begin{align*}
\|DG_y^*(\xi)-DG_{z(\gamma')}^*(\xi)\|&<
C(G,g) (2^{n+2}\cdot (2^{-n/3})^2+2^{n/2+|m|}\cdot 2^{n/3})\\
&\le 
C(G,g)\cdot  2^{(5/12)\max\{n,n'\}}< 2^{D(n,m,n',m')-10}
\end{align*}
for $\xi\in \supp \widetilde{\psi}_\gamma $ and  $y\in E$ such that $d(y, z(\gamma'))<2^{-n/3}$. (Note that the last condition on $y$ holds if  $y\in Z(\gamma')$.)
Clearly the claim of the lemma follows from this and the sublemma. 

\begin{remark} 
The argument above is  one of the  key steps in our argument, in which we used   a consequence, Corollary~\ref{cor:local},  of the fact that the flow preserves a contact structure. 
\end{remark}

Next we consider the case where 
(\ref{eqn:quarter}) does {\em not} hold. By virtue of the sublemma, it is enough to show 
\begin{equation}\label{eqn:diam}
\mathrm{diam}\, G(\Ze(\gamma'))\le 2^{\widetilde{D}(n,m,n',m')-10},
\end{equation}
since this and the assumption of the lemma imply that all $w\in Z(\gamma')$ satisfy the condition $d(G(w), z(\gamma))\le 2^{\widetilde{D}(n,m,n',m')-8}$. Note that we have $\mathrm{diam}\, G(\Ze(\gamma'))\le C(G)\cdot 2^{-n'/2}$. If $|n-n'|\le 1$, we have $\max\{-m,m'\} \ge K/5$ from (R1) and (R5), and hence (\ref{eqn:diam}) holds provided that we take large $K$ according to $G$. 
Otherwise  (\ref{eqn:diam}) follows from (R8)  immediately. 
\end{proof}


\section{The hyperbolic parts of the operator $\cM$ (II)}
\label{sec:tail1}
Let $\cR(3)$ be the set of pairs $(\gamma, \gamma')\in\Gamma\times \Gamma\setminus \cup_{i=0}^2 \cR(i)$ such that 
\begin{equation}\label{eqn:r3}
d(G(\Ze(\gamma')),z(\gamma)) > 2^{\widetilde{D}(\gamma,\gamma')-10}.
\end{equation} 
We consider the part $\cM_3$ defined formally by (\ref{eqn:defM}) for  $j=3$. This part corresponds to the case (A)  mentioned in the beginning of Section \ref{sec:body}. Below we prove
\begin{proposition}\label{pp:m3} 
The formal definition of the operator $\cM_3$ in fact gives a bounded operator $\cM_3:\bB^\beta_\nu\to \bB^\beta_{\nu'}$ for any $\nu,\nu'\ge 2\beta+2d+2$. Further  there is a constant $C_*>0$ such that we have
\[
\left\| \cM_3(\bu)
\right\|_{\beta,\nu_*}^{(\lambda)}
\le C_*\|g\|_{L^\infty}  \cdot   2^{-\beta\lambda} \cdot 
\left\|\bu\right\|_{\beta,\nu_*}^{(\lambda)}\quad \text{ for $\bu\in \bB^\beta_{\nu_*}$}
\]
for  $G:V'\to V$ in $\cH(\lambda,\Lambda)$ and  $g\in \cC^r(V')$ provided  $\lambda\ge \lambda_*$ and $\Lambda\ge \Lambda_*$.
\end{proposition}

\noindent{\it Proof.}\;
The structure of the proof is similar to that of Proposition~\ref{pp:m2}, though we consider combinations $(n,k,m,n',k',m')$ in $(\cN\oplus \integer)^2$ that satisfy the conditions (R1)-(R4) for this time. 
We set
\[
K_{n,k,m,n',k',m'}=2^{-(\nu-2d-2)\cdot (\widetilde{D}(n,m,n',m')+n/2)-r_* \Delta(n,k,n',k')}\cdot\|g\|_*\cdot  \frac{\wt(m')}{\wt(m)}.
\]
And we use the following sublemma of combinatorial nature in the place of Sublemma \ref{subl:1}, whose proof is postponed for a while.
\begin{sublemma}\label{subl:2} There exists a constant $C_*>0$ such that
\begin{align*}
&\sup_{(n',k',m')\in \cN\oplus \integer}\left(\sum_{n,k,m\mid n',k',m'}\!\!\!\!\!\!\!K_{n,k,m,n',k',m'} \right)<
C_*\|g\|_* \cdot   2^{-(\nu-2\beta-2d-2)\delta \lambda+4\beta  \lambda}
\intertext{and}
&\sup_{(n,k,m)\in \cN\oplus \integer}\left(\sum_{n',k',m'\mid n,k,m}\!\!\!\!\!\!\!K_{n,k,m,n',k',m'} \right)<
C_*\|g\|_* \cdot 2^{-(\nu-2\beta-2d-2)\delta \lambda+4\beta  \lambda}
\end{align*}
where $\sum_{n',k',m'|n,k,m}$ {\rm(resp. $\sum_{n,k,m|n',k',m'}$)} denotes the sum over $(n',k',m')$ 
{\rm (resp. $(n,k,m)$)} in $\cN\oplus \integer$ such that $(n,k,m,n',k',m')$ satisfies  (R1)-(R4). 
\end{sublemma}

We continue with the proof of Proposition \ref{pp:m3}.
We first see that it holds
\begin{equation}\label{eqn:dgw}
d_\gamma(x)^{-1}
 \le  2^{-\widetilde{D}(\gamma,\gamma')-n(\gamma)/2+11} \cdot \langle 2^{n(\gamma)/2}( G(y) -x) \rangle
\end{equation}
for $(\gamma,\gamma')\in \cR(3)$, $y\in \Ze(\gamma')$ and $x\in E$.
If  $\|x-z(\gamma)\|\ge 2^{\widetilde{D}(\gamma,\gamma')-11}$, the claim is trivial. Otherwise we have, from the definition of $\cR(3)$, that 
\[
\|G(y) -x\| \ge \|G(y)-z(\gamma)\|-\|x-z(\gamma)\|\ge  2^{\widetilde{D}(\gamma,\gamma')-11}
\]
and hence the right hand side of (\ref{eqn:dgw}) is not smaller than $1\ge d_\gamma(x)^{-1}$. 

From the inquality (\ref{eqn:dgw})  and  Lemma \ref{lm:kest1} for  $\mu=
\max\{\nu,\nu'\}+4d+4$, we obtain the following estimate on the kernel $\kappa_{\gamma\gamma'}(x',x)$ of the operator $\cL_{\gamma\gamma'}$:
\begin{align*}
|d_{\gamma'}^{\nu'}(x') \kappa_{\gamma\gamma'}(x',x)& d_{\gamma}^{-\nu+2d+2}(x)|\\
&\le
C_*\|g\|_*\cdot  2^{-(\nu-2d-2)(\widetilde{D}(n,m,n',m')+n/2)-r_* \Delta(n,k,n',k')}\\
&\qquad \cdot 
\int_{Z(\gamma')} b_{\gamma'}^{2d+2}(x'-y)\cdot b_\gamma^{\mu-\nu+2d+2}(G(y)-x) dy
\end{align*}
for $(\gamma,\gamma')\in \cR(3)$, where  $n=n(\gamma)$, $k=k(\gamma)$, $m=m(\gamma)$, $n'=n(\gamma')$, $k'=k(\gamma')$ and $m'=m(\gamma')$.

Since the intersection multiplicity of  $Z(\gamma')$
for $\gamma'\in \Gamma$ such that $n(\gamma')=n'$, $k(\gamma')=k'$ and $m(\gamma')=m'$ is bounded by some constant depending only on~$d$, we have, from the estimate  above and  Young inequality,  that
\begin{align*}
&\sum_{\gamma':n',k',m'}
\biggl\|
\mathop{\sum{}^{\dag}}_{\gamma:n,k,m;\gamma'} d_{\gamma'}^{\nu'}\cL_{\gamma\gamma'} u_\gamma\biggr\|^2_{L^2}\\
&\qquad\qquad \le C_* \cdot \biggl|  K_{n,k,m, n',k',m'}\cdot  \frac{\wt(m)}{\wt(m')}\biggr|^2   \cdot 
\biggl\| 
\sum_{\gamma:n,k,m}d_\gamma^{\nu-2d-2} |u_\gamma|  \biggr\|_{L^2}^2
\end{align*}
for $\bu=(u_\gamma)_{\gamma\in \Gamma}\in \bB^\beta_\nu$, where $\sum_{\gamma:n,k,m;\gamma'}^\dag$ denotes the sum over $\gamma\in \Gamma$ such that $n(\gamma)=n$, $k(\gamma)=k$ and $m(\gamma)=m$ and that $(\gamma,\gamma')\in \cR(3)$, while $\sum_{\gamma:n,k,m}$ denotes the sum over $\gamma\in \Gamma$ such that $n(\gamma)=n$, $k(\gamma)=k$ and $m(\gamma)=m$. 
Applying Schwarz inequality as in (\ref{eqn:sch}), we get
\begin{align*}
\sum_{\gamma':n',k',m'}&\left\|\sum_{\gamma:n,k,m:\gamma'}\!\!\!\!\!{}^\dag\;\; d_{\gamma'}^{\nu'}\cL_{\gamma\gamma'} u_\gamma\right\|_{L^2}^2  \le C_* 
\left| K_{n,k,m, n',k',m'}  \frac{\wt(m)}{\wt(m')}\right|^2
\left\|  v_{n,k,m} \right\|_{L^2}^2
\end{align*}
where $v_{n,k,m}$ is defined by (\ref{eqn:v}).
Once we have this estimate, we can proceed just as in the last part of the proof of Proposition~\ref{pp:m2}, using Sublemma \ref{subl:2} in the place of Sublemma \ref{subl:1}, and conclude  that
\[
\left\| \cM_3(\bu)
\right\|_{\beta,\nu'}^{(\lambda)}
\le C_*\|g\|_{*}  \cdot  2^{-(\nu-2\beta-2d-2)\delta \lambda+4\beta\lambda} \cdot 
\left\|\bu\right\|_{\beta,\nu}^{(\lambda)}\quad \text{ for $\bu\in \bB^\beta_{\nu}$.}
\]
This  implies not only that $\cM_3:\bB^\beta_{\nu}\to \bB^\beta_{\nu'}$ is bounded but also the latter claim of the proposition because $-(\nu_*-2\beta-2d-2)\delta \lambda+4\beta\lambda<-\beta\lambda$ from the choice of $\nu_*$. We finish the proof by proving  Sublemma \ref{subl:2}.
\begin{proof}[Proof of Sublemma \ref{subl:2}]
In the argument below, we consider combinations $(n,k,m,n',k',m')$ in $(\cN\times \integer)^2$ that satisfy the conditions (R1)-(R4).
From the definition of $\widetilde{D}(n,m,n',m')$, we have that
\begin{align*}
\widetilde{D}(n,m,n',m') +n/2+|n-n'|/2 \ge \max\{-m, m'\}
\end{align*} 
and hence, by (\ref{eqn:nnd}),  that
\begin{align*}
2\widetilde{D}(n,m,n',m') +n+2\Delta(n,k,n',k')+10 \ge 2\max\{-m, m'\}\ge m'-m.
\end{align*} 
Using this and (\ref{eqn:wt}), we see that $K_{n,k,m,n',k',m'}$ is bounded by 
\begin{equation}\label{eqn:bK}
C_* \|g\|_*\cdot 2^{-(\nu-2\beta-2d-2)
(\widetilde{D}(n,m,n',m')+n/2)-(r_*-2\beta) \Delta(n,k,n',k')+4\beta \lambda}
\end{equation}
Below we proceed as in the proof of Sublemma \ref{subl:1}: We restrict our attention  to the  cases (I) $ |n'-n|\le 1$  and (II) $|n'-n|\ge 2$ in turn, and prove the claims with the sums replaced by the partial sums restricted to such cases. 

Let us first consider the case (I).  In this case, we have, from  (\ref{eqn:bK}), that 
\[
K_{n,k,m,n',k',m'}
\le C_* \|g\|_*\cdot 2^{-(\nu-2\beta-2d-2)
\max\{-m,m'\}-(r_*-2\beta) \Delta(n,k,n',k')+4\beta \lambda}.
\]
Note that there exists a constant $C_*>0$ such that 
\begin{align}
&\sum_{m:\mbox{\tiny (R5), (R7)}}\;\; 2^{-(\nu-2\beta-2d-2)
\max\{-m,m'\}}<C_*\qquad \mbox{ for any $m'$}\label{eqn:r571}
\intertext{
and }
&\sum_{m':\mbox{\tiny (R5), (R7)}} 2^{-(\nu-2\beta-2d-2)
\max\{-m,m'\}}<C_*\qquad \mbox{ for any $m$}
\label{eqn:r572}
\end{align}
where $\sum_{m:\mbox{\tiny (R5), (R7)}}$ (resp. $\sum_{m':\mbox{\tiny (R5), (R7)}}$ ) denotes the sum over $m$ (resp. $m'$) satisfying both of the  conditions $\max\{-m, m'\}\ge \delta\lambda $ and $ m'\ge m-\lambda+20$ that come from (R5) and (R7) respectively. Therefore, taking (\ref{eqn:chisum}) and (\ref{eqn:chisum2}) into account, we obtain the required  estimates for the partial sums.

Let us consider the case (II). In this case, we have
\begin{align*}
\widetilde{D}(n,m,n',m')+n/2&\ge 
\max\{-m,m'\}-\max\{n,n'\}\ge  \max\{-m,m'\}/2
\end{align*}
from (R8). Hence it follows from the bound (\ref{eqn:bK}) above that
\[
K_{n,k,m,n',k',m'}
\le C_* \|g\|_*\cdot 2^{-(\nu-2\beta-2d-2)
\max\{-m,m'\}/2-(r_*-2\beta) \Delta(n,k,n',k')+4\beta \lambda}.
\]
Using this estimate and taking (R5) and (R8) 
and also  (\ref{eqn:chisum}) and (\ref{eqn:chisum2}) into account, we obtain the  required inequalities for the partial sums. (In this case, we can make the constant $C_*$ in the required inequality arbitrarily small by taking large $K$.)
\end{proof}


\section{The hyperbolic parts of the operator $\cM$ (III)}
\label{sec:tail2}
In this section we consider the remainder  of the hyperbolic part.
We set $\cR(4)=\Gamma\times \Gamma\setminus (\cup_{i=0}^3 \cR(i))$ and let $\cM_4$ be the part defined  formally by  (\ref{eqn:defM}) for $j=4$.  This part corresponds to the case (B)  mentioned in the beginning of Section \ref{sec:body}. Below we prove
\begin{proposition}\label{pp:m4} 
The formal definition of the operator $\cM_4$ in fact gives a bounded operator $\cM_4:\bB^\beta_\nu\to \bB^\beta_{\nu'}$ for $\nu,\nu'\ge 2d+2$.
Further, there exists a constant $C_*>0$ such that we have 
\[
\left\| \cM_4(\bu)
\right\|_{\beta,\nu_*}^{(\lambda)}
\le C_* \cdot \|g\|_*\cdot 2^{-\beta \lambda} 
\left\|\bu\right\|_{\beta,\nu_*}^{(\lambda)}\quad \text{for $\bu\in \bB^\beta_\nu$}
\]
for $G:V'\to V$ in $\cH(\lambda,\Lambda)$ and $g\in \cC^r(V')$ provided  $\lambda\ge \lambda_*$ and $\Lambda\ge \Lambda_*$.
\end{proposition} 
In the proof, we need the following lemma which gives a delicate estimate on the kernel $\kappa_{\gamma\gamma'}$ of $\cL_{\gamma\gamma'}$ that results from  applications of integration by parts.  
We will prove it after finishing the proof of  Proposition \ref{pp:m4}. 
\begin{lemma}\label{lm:kest2} 
For  $\mu\ge 2d+2$ and $\mu'>0$, there exist a constant $C_*>0$, and another constant $C(G,g)$ that may depend on $G$ and $g$, such that 
\begin{align*}
|\kappa_{\gamma\gamma'}(x',x)|\le C(n&(\gamma),m(\gamma),n(\gamma'),m(\gamma')) \\
& 
\cdot 2^{-r_*\cdot\Delta(\gamma,\gamma')}  \int b^\mu_{\gamma'}(x'-y)
\cdot d_{\gamma'}^{-2d-2}(y)  \cdot b^{\mu}_{\gamma}(G(y)-x) dy
\end{align*}
for $(\gamma, \gamma')\in \cR(4)$ and $x,x'\in E$, where we set
\begin{align}\label{eqn:other0}
C(n,m,n',m')=&C_*   \|g\|_* 2^{-\mu' (D(n,m,n',m')-n'/2)}\\
&\qquad 
+C(G,g)   2^{-(r-1) (D(n,m,n',m')-n'/3)}\notag
\end{align}
in the case  $\max\{|m|,|m'|\}\le \max\{n,n'\}/4$ and set
\begin{equation}\label{eqn:other}
C(n,m,n',m')=C(G,g)    2^{-(r-1)(D(n,m,n',m')-n'/2)}
\end{equation}
 otherwise. 
\end{lemma} 
\begin{proof}[Proof of Proposition \ref{pp:m4}]\;
The structure of the proof is again similar to that of Proposition \ref{pp:m2} though we consider  combinations $(n,k,m,n',k',m')$  in $(\cN\oplus \integer)^2$ that satisfy (R1)-(R4), as in the proof of Proposition \ref{pp:m3}.  
We fix  $\mu\ge \nu'+2d+2$ and $\mu'> 6\beta/\delta+2\beta$ and let  $C(n,m,n',m')$ be that in Lemma \ref{lm:kest2} for such $\mu$ and $\mu'$. For this time, we set
\[
K_{n,k,m,n',k',m'}=  C(n,m,n',m')\cdot 2^{-r_*\Delta(n, k, n', k')} \cdot 
\frac{\wt(m')}
{\wt(m)}.
\]
We use the following sublemma, whose proof is postponed for a while. 
\begin{sublemma}\label{subl:3}
There exists a constant $C_*>0$ such that  
\[
\sup_{(n,k,m)\in \cN\oplus \integer}\left(\sum_{n',k',m'|n,k,m} K_{n,k,m,n',k',m'}\right)
< C_* \|g\|_* \cdot  2^{-(\mu'-2\beta) \delta \lambda +4\beta\lambda}  
\]
and 
\[
\sup_{(n',k',m')\in \cN\oplus \integer}\left(\sum_{n,k,m|n',k',m'} K_{n,k,m,n',k',m'}\right)
< C_* \|g\|_*\cdot  2^{-(\mu'-2\beta) \delta \lambda +4\beta\lambda}
\]
where $\sum_{n',k',m'|n,k,m}$ {\rm (resp. $\sum_{n,k,m|n',k',m'}$)} denotes the sum over $(n',k',m')$ {\rm (resp. $(n,k,m)$)} in $\cN\oplus \integer$ such that $(n,k,m,n',k',m')$ satisfies  (R1)-(R4). 
\end{sublemma}

By the same argument that we deduce (\ref{eqn:ck}) from Lemma \ref{lm:kest1} 
in the proof of Proposition \ref{pp:m2}, we can deduce the following estimate from Lemma \ref{lm:kest2}:
\begin{align*}
&\sum_{\gamma':n',k',m'}\left\|\sum_{\gamma:n,k,m;\gamma'}\!\!\!\!\!\!\!{}^{\dag\dag}\;\; d_{\gamma'}^{\nu'}\cL_{\gamma\gamma'} u_\gamma\right\|^2_{L^2}\le C_*  
\left| K_{n,k,m, n',k',m'} \frac{\wt(m)}{\wt(m')}\right|^2   
\left\| 
v_{n,k,m}   \right\|_{L^2}^2
\end{align*}
for $\bu=(u_\gamma)_{\gamma\in \Gamma}\in \bB^\beta_\nu$, where $v_{n,k,m}$ is defined by (\ref{eqn:v}) and $\sum_{\gamma:n,k,m;\gamma'}^{\dag\dag}$ denotes the sum over $\gamma\in \Gamma$ such that $n(\gamma)=n$, $k(\gamma)=k$ and $m(\gamma)=m$ and that $(\gamma,\gamma')\in \cR(4)$.
But, once we have this estimate, we can proceed just as in the last part of the proof of Proposition \ref{pp:m2}, using Sublemma \ref{subl:3} instead of Sublemma \ref{subl:1}, and  conclude that
\[
\left\| \cM_4(\bu)
\right\|_{\beta,\nu'}^{(\lambda)}
\le C_*\|g\|_{*}  \cdot   2^{-(\mu'-2\beta)\delta \lambda+4\beta\lambda} \cdot 
\left\|\bu\right\|_{\beta,\nu}^{(\lambda)}\quad \text{ for $\bu\in \bB^\beta_{\nu}$.}
\]
Since we have   $
-(\mu'-2\beta) \delta \lambda +4\beta\lambda 
<-\beta\lambda$
from the choice of $\mu'$, this implies the conclusion of the proposition. 
\end{proof}

Below we prove Sublemma \ref{subl:3} and Lemma \ref{lm:kest2}.

\begin{proof}[Proof of Sublemma \ref{subl:3}]
In the argument below, we consider combinations 
$(n,k,m,n',k',m')$ in $(\cN\oplus \integer)^2$ that satisfy  (R1)-(R4). 
First we  restrict our attention to the case where $\max\{|m|,|m'|\}> \max\{n,n'\}/4$. Note that $C(n,m,n',m')$ in the definition of $K_{n,k,m,n',k',m'}$ is given by (\ref{eqn:other}) in this case. From (R1) and (R5), we have 
$\max\{-m,m'\}\ge K/5$. 
Since
\[
{D}(n,m,n',m')-n'/2\ge \max\{-m,m'\}-|n'-n|/2-\lambda,
\]
we obtain
\begin{align*}
K_{n,k,m,n',k',m'}&\le C(G,g) 
2^{-(r-1)(\max\{-m,m'\}-|n'-n|/2)+\beta(m'-m)-r*\Delta(n,k,n',k')}
\\
&\le C(G,g) 2^{-(r-1-2\beta)\max\{-m,m'\}-(r*-(r-1))\Delta(n,k,n',k')}
\end{align*}
by using (\ref{eqn:nnd}) and (\ref{eqn:wt}).
Therefore, by using variants\footnote{In (\ref{eqn:r571}) and (\ref{eqn:r572}), we replace the exponent
$(\nu-2\beta-2d-2)$ by $(r-1-2\beta)$ and the condition $\max\{-m, m'\}\ge \delta\lambda $ by $\max\{-m, m'\}\ge K/5$. As the result, the right hand sides should be $C_*\lambda \cdot 2^{-K/5}$. } of the inequalities (\ref{eqn:r571}) and (\ref{eqn:r572}), and also by using (\ref{eqn:chisum}) and (\ref{eqn:chisum2}), we obtain the inequalities in Sublemma \ref{subl:3} with the sum restricted to this case, provided that we take sufficiently large $K$ according to $\lambda$ and $G$.

We next restrict  our attention  to  the case $\max\{|m|,|m'|\}\le \max\{n,n'\}/4$. 
In this case, we have  $\max\{n,n'\}\ge K$ from (R1) and $|n-n'|\le 1$ from (R8). Since we have
\[
C(n,m,n',m')\le C_*  \|g\|_*  2^{-\mu'\max\{-m,m'\}}
+C(G,g)  2^{-(r-1) (\max\{-m,m'\}+n/6)},
\]
we see that $K_{n,k,m,n',k',m'}$ is bounded by
\begin{align*}
& 2^{4\beta\lambda-r_*\cdot \Delta(n,k,n',k')}\cdot  \\
& \left(C_*  \|g\|_*  2^{-(\mu'-2\beta)\max\{-m,m'\}}
+C(G,g) 2^{-(r-1)K/6-(r-1-2\beta) \max\{-m,m'\}}\right).
\end{align*}  
Therefore, again by using variants\footnote{Here we just change the exponent 
$(\nu-2\beta-2d-2)$ in (\ref{eqn:r571}) and (\ref{eqn:r572}) appropriately and the right hand sides should be the same as those in  (\ref{eqn:r571}) and (\ref{eqn:r572}). }
 of the inequalities (\ref{eqn:r571}) and (\ref{eqn:r572}), and also by using (\ref{eqn:chisum}) and (\ref{eqn:chisum2}), we obtain the  required  inequalities with the sum restricted to this case, provided that we take sufficiently large $K$ according to $\lambda$ and $G$.
\end{proof}

\begin{proof}[Proof of Lemma \ref{lm:kest2}]
Recall that we set $v_0=\partial/\partial x_0$ and 
take unit vectors  $v_1,v_2,\dots ,v_{2d}$ in $E$ so that  $\{v_j\}_{j=0}^{2d}$ is an orthonormal basis of $E$.
We first prove the lemma in  the case where $\max\{|m|,|m'|\}> \max\{n,n'\}/4$.
Let us write the integration with respect to the variable $y$ in (\ref{eqn:ker2}) as 
\begin{equation}\label{eqn:ker2ab}
\kappa_{\gamma\gamma'}(x',x;\xi,\eta)=\int 
e^{i f(y;x,x';\xi,\eta)} R(y, \xi, \eta) dy
\end{equation}
where $R(y, \xi, \eta)$ is that given in (\ref{eqn:RR}) and we set
\[
f(y;x,x';\xi,\eta)=\langle \xi, x'-y\rangle-i\langle \eta, G(y)-x\rangle.
\]
If we apply the formula (\ref{eqn:intbypart}) of integration by parts along the set of vectors $\{v_j\}_{j=0}^{2d}$ for $k$ times ($0\le k\le r-1$), we will get the expression
\[
\kappa_{\gamma\gamma'}(x',x;\xi,\eta)=\int 
e^{i f(y;x,x';\xi,\eta)} R_k(y, \xi, \eta) dy
\]
where $R_k(\cdot)$ should be defined inductively by
\begin{equation}\label{eqn:indR}
R_0(y,\xi,\eta)=R(y,\xi,\eta),\qquad 
R_{k}=\sum_{j=0}^{2d} 
v_j\left(\frac{i\cdot R_{k-1}\cdot v_j(f)}{\sum_{\ell=0}^{2d}v_\ell(f)^2}\right).
\end{equation}
By induction on $0\le k\le r-1$, we  show the following claim.
\begin{sublemma} \label{sublm:ddd}
For any multi-indices $\alpha, \beta\in (\integer_+)^{2d+1}$ and $\alpha'\in (\integer_+)^{2d+1}$ such that 
$|\alpha'|\le (r-1)-k$, we have
\begin{align*}
\|D^\alpha_\xi D^\beta_\eta D^{\alpha'}_y R_{k}\|_{L^\infty}\le& C(G,g,\alpha,\beta,\alpha')\cdot 2^{-k\cdot  (D(\gamma,\gamma')-n(\gamma')/2)-r_*\cdot\Delta(\gamma,\gamma')}\\
&\quad \cdot 2^{-|\alpha|n(\gamma')/2-|\alpha|_\dag |m(\gamma')|-|\beta|n(\gamma)/2-|\beta|_\dag|m(\gamma)|+|\alpha'|n(\gamma')/2}.
\end{align*}
\end{sublemma}
\begin{proof}
By using the estimates on the derivatives of $\rho_{\gamma'}$, $\psi_{\gamma'}$ and $\widetilde{\psi}_\gamma$, we can check the claim for $k=0$. To continue,  let us first check the inequality 
\begin{equation}\label{eqn:mnc1}
\max\{n(\gamma)/2+|m(\gamma)|/2, n(\gamma')/2+|m(\gamma')|/2\}
\le D(\gamma, \gamma')+C(G).
\end{equation}
If $|n(\gamma)-n(\gamma')|\le 1$, we can check this inequality by using (R5) in the definition of $D(\gamma, \gamma')$. Otherwise we have, from (R3) and (\ref{eqn:dnk}), that
\[
m(\gamma')\ge m(\gamma)+5\max\{n(\gamma), n(\gamma')\}-C(G)
\]
and  hence we obtain the same inequality by a crude estimate. 

Note that we have (\ref{eqn:cldiv})  for all $y\in \supp \rho_{\gamma'}\subset Z(\gamma')$ from  Lemma \ref{lm:division}. This implies that  we have 
\begin{equation}\label{eqn:vfest}
\sum_{\ell=0}^{2d} v_\ell(f)^2 =\sum_{\ell=0}^{2d} |\langle DG^*_y(\eta)-\xi, v_\ell\rangle|^2=\|DG_y^*(\eta)-\xi\|^2\ge 2^{2(D(\gamma,\gamma')-10)}
\end{equation}
if  $R(y, \xi, \eta)\neq 0$. For $\alpha\in (\integer_+)^{2d+1}$ with $|\alpha|\ge 2$, we have a simple estimate
\begin{equation}\label{eqn:N1}
|D_y^\alpha f|\le C(G)\cdot \|\eta\|\le C(G)\cdot  2^{n(\gamma)}\quad \mbox{if  $R(y, \xi, \eta)\neq 0$.}
\end{equation}
In the case $|\alpha|=2$, we have, from Corollary \ref{cor:local} and (\ref{eqn:mnc1}), that
\begin{equation}\label{eqn:N2}
 |D_y^\alpha f|\le C(G)\cdot
  2^{n(\gamma)/2+ |m(\gamma)|}\le C(G)\cdot 2^{D(\gamma,\gamma')}\quad \mbox{if $R(y, \xi, \eta)\neq 0$. }
\end{equation}
Using these estimates in the inductive definition (\ref{eqn:indR})  of $R_k(\cdot)$, we obtain the claim of the sublemma by induction on $k$.  
\end{proof}

By the same argument as in the last part of the proof of Lemma~\ref{lm:kest1}, we see that the claim of the sublemma above for $k=r-1$  implies 
\begin{align*}
|\kappa_{\gamma\gamma'}(x',x)|\le C(G,g,\mu)  &  \cdot2^{-(r-1)(D(n,m,n',m')-n'/2)-r_*\cdot\Delta(\gamma,\gamma')}\\
& 
\cdot   \int_{Z(\gamma')} b^\mu_{\gamma'}(x'-y)  \cdot b^{\mu}_{\gamma}(G(y)-x) dy
\end{align*}
for  $x,x'\in E$. Clearly this implies the claim of the lemma  in the case $\max\{|m|,|m'|\}> \max\{n,n'\}/4$.

Next we prove the lemma in  the case $\max\{|m|,|m'|\}\le \max\{n,n'\}/4$. 
In this case, we have $\max\{n(\gamma), n(\gamma')\}> K$ from (R1), and $|n(\gamma')-n(\gamma)|\le 1$ from (R8). In particular, we have  $n(\gamma)\ge K-1$ and $n(\gamma')\ge K-1$. 
It follows from the definition of $D(\gamma,\gamma')$ and (R5) that
\begin{equation}\label{eqn:R4rel}
D(\gamma,\gamma')\ge |m(\gamma')|+n(\gamma')/2\ge n(\gamma')/2\quad \text{for  $(\gamma, \gamma') \in \cR(4)$.}
\end{equation} 
This implies that the diameter of $\supp \psi_{\gamma'}$ is not much larger than $2^{D(\gamma,\gamma')}$, that is, we have
\[
\mathrm{diam}\,  \left(\supp \psi_{\gamma'}\right)\le C_0\cdot 2^{|m(\gamma')|+n(\gamma')/2}\le C_0 \cdot 2^{D(\gamma,\gamma')}
\]
for some constant $C_0>0$ that depends only on $d$. 
From this fact,  we can construct a $C^\infty$ partition of unity 
\[
\left\{\left. \phi_{\gamma\gamma'}^{(\ell)}:E^*\to[0,1]\;\right|\; \ell=0,1,2,\dots\right\}
\]
for each pair $(\gamma, \gamma')\in \cR(4)$ so that the following three conditions hold:
\begin{itemize}
\item[(P1)] $\supp \phi_{\gamma\gamma'}^{(0)}$ is contained in the $2^{D(\gamma,\gamma')-11}$-neighborhood of $\supp \psi_{\gamma'}$,
\item[(P2)] for $\ell\ge 1$, the distance between $\supp \psi_{\gamma'}$ and $\supp \phi_{\gamma\gamma'}^{(\ell)}$ is bounded from below by $2^{D(\gamma,\gamma')+\ell-13}$, and
\item[(P3)] the family of functions $\phi_{\gamma\gamma'}^{(\ell)}$ for $ (\gamma, \gamma')\in \cR(4)$ and $\ell\ge 0$ are uniformly bounded up to scaling in the the following sense:  all the functions
\[
\phi_{\gamma\gamma'}^{(\ell)}\circ A_{\gamma'}\circ
\widetilde{J}_{D(\gamma,\gamma')+\ell}:E^*\to [0,1],\quad (\gamma, \gamma')\in \cR(4), \;\;\ell \ge 0
\]
are bounded in $\mathcal{D}(E^*)$, that is they are supported in a bounded subset of $E^*$ and their $C^s$ norms are uniformly bounded for each $s>0$, where $A_{\gamma}$ is the translation that  defined in Subsection \ref{ss:bdd} and $
\widetilde{J}_{t}:E^*\to E^*$ is defined by $\widetilde{J}_{t}(\xi)= 2^{t}\cdot \xi$. 
\end{itemize}
We will give one way of the construction of $\phi_{\gamma\gamma'}^{(\ell)}$ as above in Remark \ref{rem:consth} at the end of this proof.

Using the partitions of unity as above, we decompose the kernel (\ref{eqn:kappa}) as 
\begin{equation}\label{eqn:kdec}
\kappa_{\gamma\gamma'}(x',x)=(2\pi)^{-3(2d+1)}\sum_{\ell=0}^{\infty}\kappa_{\gamma\gamma'}^{(\ell)}(x',x)
\end{equation}
where  $\kappa_{\gamma\gamma'}^{(\ell)}(x',x)$ on the right hand side is defined by
\begin{equation}\label{eqn:kap}
\kappa_{\gamma\gamma'}^{(\ell)}(x',x)=\int e^{if(x',y',y,x;\xi',\xi,\eta)}\, R^{(\ell)}(y',y;\xi', \xi,\eta) dydy'd\xi d\xi' d\eta,
\end{equation}
with setting 
\begin{align*}
f(x',y',y,x;\xi',\xi,\eta)&=\langle\xi, x'-y'\rangle
+\langle\xi', y'-y\rangle+\langle\eta,G(y)-x \rangle,\\
R^{(\ell)}(y',y;\xi', \xi,\eta) &=(2\pi)^{-3(2d+1)}\rho_{\gamma'}(y') \widetilde{\rho}_{\gamma'}(y)   g(y)  \psi_{\gamma'}(\xi)  \phi_{\gamma\gamma'}^{(\ell)}(\xi')  \widetilde{\psi}_{\gamma}(\eta),\\
\widetilde{\rho}_{\gamma'}(y) &=\chi
\bigl(2^{n(\gamma')/3+1} \|y-z(\gamma')\|\bigr).
\end{align*}
To check that (\ref{eqn:kdec})  holds   pointwise (at least), we use the fact that
\begin{align*}
(2\pi)^{-(2d+1)}\sum_{\ell=0}^\infty\int e^{i\langle \xi', y-y'\rangle} \phi_{\gamma\gamma'}^{(\ell)}(\xi') d\xi'&=(2\pi)^{-(2d+1)}\int e^{i\langle \xi', y-y'\rangle} d\xi'\\
&=\delta(y-y')
\end{align*}
in the sense of distribution and that   $
\widetilde{\rho}_{\gamma'}\cdot{\rho}_{\gamma'}\equiv {\rho}_{\gamma'}$. (We have $\widetilde{\rho}_{\gamma'}\equiv 1$ on the disk in $E$ with center at $z(\gamma')$ and radius $2^{-n(\gamma')/3-1}$, which contains the support of ${\rho}_{\gamma'}$, provided that  $K$ is large enough.)

We write the integration with respect to $y$ and $y'$ in (\ref{eqn:kap}) as
\begin{equation}\label{eqn:kapp}
\kappa_{\gamma\gamma'}^{(\ell)}(x',x; \xi',\xi, \eta)=\int e^{if(x',y',y,x;\xi',\xi,\eta)}\, R^{(\ell)}(y',y;\xi', \xi,\eta) dydy'.
\end{equation}
Below we are going to estimate $\kappa_{\gamma\gamma'}^{(\ell)}(x',x; \xi',\xi, \eta)$ by applying integration by parts to the integral with respect to the variables $y$ and $y'$ in (\ref{eqn:kap}) in two steps. (The argument below is formally  parallel to that in the case $\max\{|m|,|m'|\}> \max\{n,n'\}/4$ given above.)
To this end, we extend the formula (\ref{eqn:intbypart}) of integration by parts to oscillatory integrals on $E\times E$ in an obvious manner. And we regard $y$ and $y'$ as the former and latter variable on $E\times E$ respectively. 

As the first step, we apply the formula of integration by parts along the single vector $(v_0,v_0)$ for $r_*$ times if $\Delta(\gamma,\gamma')>0$ (and  we do nothing otherwise.) As the result, we obtain the expression
\begin{equation}\label{eqn:krl}
\kappa_{\gamma\gamma'}^{(\ell)}(x',x; \xi',\xi, \eta)=\int e^{if(x',y',y,x;\xi',\xi,\eta)}\, R_0^{(\ell)}(y',y;\xi', \xi,\eta) dydy'
\end{equation}
where 
\[
 R_0^{(\ell)}(y',y;\xi', \xi,\eta)=
\frac{ i^{r_*}\cdot
 \hat{v_0}^{r_*}(\rho_{\gamma'}(y')\cdot \widetilde{\rho}_{\gamma'}(y)\cdot  g(y))\cdot  \psi_{\gamma'}(\xi) \cdot \phi_{\gamma\gamma'}^{(\ell)}(\xi')\cdot  \widetilde{\psi}_{\gamma}(\eta)}{(2\pi)^{3(2d+1)}(\pi^*_0(\eta-\xi))^{r_*}}
\]
in the case $\Delta(\gamma,\gamma')>0$ (and just set $R_0^{(\ell)}=R^{(\ell)}$ otherwise). 
Here we wrote  $\hat{v_0}$ for the directional derivative along the vector $(v_0,v_0)$ in $E\times E$.
In the second step, we consider the two cases  $\ell=0$ and $\ell>0$ separately. In fact, the first term in (\ref{eqn:other0}) comes from the case $\ell>0$ and the second from the case $\ell=0$ as we will see below.

Let us first consider the case $\ell=0$. 
In this case, we will apply the formula of integration by parts  along the set of vectors 
$\{(v_i,0)\}_{i=0}^{2d}$  (or, in other words, we will apply the formula (\ref{eqn:intbypart}) of integration by parts to the integral with respect to the variable $y$ along the vectors $\{v_i\}_{i=0}^{2d}$)  for $(r-1)$ times. 
Let us write the result of such application of  integration by parts for \underline{$k$ times} as
\[
\kappa_{\gamma\gamma'}^{(0)}(x',x)=\int e^{if(x',y',y,x;\xi',\xi,\eta)}\, R_k^{(0)}(y',y;\xi', \xi,\eta) dydy'
\]
where $R_k^{(0)}(\cdot)$ for $1\le k\le r-1$ should be defined inductively by
\[
R_{k}^{(0)}=\sum_{j=0}^{2d} 
v_j\left(\frac{i\cdot R^{(0)}_{k-1}\cdot v_j(f)}{\sum_{\ell=0}^{2d}v_\ell(f)^2}\right)\quad \mbox{for $k=1,2,\cdots, r-1$.}
\] 
(Here and below we suppose that $v_j$ are the directional derivatives along $v_j$  with respect to the variable $y$.) 

By induction on $0\le k\le r-1$, we show the following claim.
\begin{sublemma} \label{sublm:dddd}
For any multi-indices $\alpha_1, \alpha_2, \beta,\alpha',\beta'\in (\integer_+)^{2d+1}$ such that 
$|\alpha'|\le (r-1)-k$, we have
\begin{align*}
\|D^{\alpha_1}_\xi D^{\alpha_2}_{\xi'} D^\beta_{\eta}
 D^{\alpha'}_y D^{\beta'}_{y'} R^{(0)}_{k}\|_{L^\infty}
 \le& C(G,g,\alpha_1,\alpha_2,\beta, \alpha',\beta',k)\\
 &\quad \cdot 2^{-k\cdot  (D(\gamma,\gamma')-n(\gamma')/3)-r_*\cdot\Delta(\gamma,\gamma')}\\
&\quad \cdot 2^{-|\alpha_1|n(\gamma')/2-|\alpha_1|_\dag |m(\gamma')|-|\beta|n(\gamma)/2-|\beta|_\dag|m(\gamma)|}\\
&\quad \cdot 2^{-|\alpha_2| D(\gamma,\gamma')+|\alpha'|n(\gamma')/3
+|\beta'|n(\gamma)/2
}.
\end{align*}
\end{sublemma}
\begin{proof}
The proof is parallel to that of Sublemma \ref{sublm:ddd}. 
We have
\begin{equation}\label{eqn:dya}
\|D_y^{\alpha'}\widetilde{\rho}_{\gamma'}\|_{L^\infty}\le C_*(\alpha')\cdot 
2^{|\alpha'| n(\gamma')/3}
\end{equation}
and also
\begin{equation}\label{eqn:dya2}
\|D_\xi^{\alpha} \phi_{\gamma\gamma'}^{(\ell)}\|_{L^\infty}\le C_*(\alpha)\cdot 2^{-(D(\gamma,\gamma')+\ell)}
\end{equation}
from the condition (P3) in the choice of $ \phi_{\gamma\gamma'}^{(\ell)}$. 
Using these estimates, we can check the claim for $k=0$. 

From Lemma \ref{lm:division} and (P1), we have   
\[
d(\Pi_{z(\gamma')}(\supp{\phi}^{(0)}_{\gamma\gamma'}),\Pi_{z(\gamma')}(DG^{*}_{y}(\supp \widetilde{\psi}_{\gamma})))\ge 
2^{D(\gamma,\gamma')-11}
\]
for all $y\in \supp \widetilde{\rho}_{\gamma'}$. This implies that 
\[
\sum_{\ell=0}^{2d}v_\ell(f)^2  =\sum_{\ell=0}^{2d} |\langle DG^*_y(\eta)-\xi, v_\ell\rangle|^2=\|DG_y^*(\eta)-\xi\|^2\ge 2^{2(D(\gamma,\gamma')-11)}
\]
if  $R^{(0)}(y',y;\xi', \xi,\eta) \neq 0$. 
Also we have (\ref{eqn:N1}) and (\ref{eqn:N2}) with the assumption
$R(y,\xi,\eta)\neq 0$ replaced by  $R^{(0)}(y',y;\xi', \xi,\eta) \neq 0$. Using these estimates, we can prove the required estimate by  induction  on $k$ as in the proof of Sublemma \ref{sublm:ddd}. 
\end{proof}
By the same argument as in the last part of the proof of Lemma~\ref{lm:kest1} using in addition the fact that the $(2d+1)$-dimensional volume of the support of $\phi^{(0)}_{\gamma\gamma'}$ is bounded by $C_*\cdot 2^{(2d+1)D(\gamma,\gamma')}$, we see that the claim of the sublemma above for $k=r-1$  implies 
\begin{align*}
|\kappa_{\gamma\gamma'}^{(0)}(x',x)|\le 
C&(G,g)  \cdot 2^{-r_*\cdot \Delta(\gamma,\gamma')-(r-1) 
(D(\gamma,\gamma')-n(\gamma')/3)}\\
&\cdot \int \left(\int_{\Ze(\gamma')} b_{\gamma'}^\mu(x'-y') b_{\gamma\gamma',0}^{\mu+2d+2}(y'-y) dy' \right)b_{\gamma}^\mu(G(y)-x)  dy
\end{align*}
where (and also henceforth) we set 
\[
b_{\gamma\gamma',\ell}^{\mu}(x)=
2^{(2d+1)(D(\gamma,\gamma')+\ell)} 
\left\langle 2^{D(\gamma,\gamma')+\ell}\cdot x\right\rangle^{-\mu}\quad \text{for $\mu>0$ and $\ell\ge 0$}.
\]

Next let us consider the case $\ell\ge 1$. In this case,  we will apply the formula of integration by parts  to the integral (\ref{eqn:krl}) along the set of vectors 
$\{(0,v_i)\}_{i=0}^{2d}$  (or, in other words, we will apply the formula (\ref{eqn:intbypart}) of integration by parts to the integral with respect to the variable $y'$ along the vectors $\{v_i\}_{i=0}^{2d}$)  for $\mu'$ times. 
Let us write the result of such application of  integration by parts for \underline{$k$ times} as
\[
\widetilde{\kappa}_{\gamma\gamma'}^{(\ell)}(x',x)=\int e^{if(x',y',y,x;\xi',\xi,\eta)}\, \widetilde{R}_k^{(\ell)}(y',y;\xi', \xi,\eta) dydy'
\]
where $\widetilde{R}_k^{(\ell)}(\cdot)$ for $1\le k\le \mu'$ should be defined inductively by $\widetilde{R}_0^{(\ell)}=R_0^{(\ell)}$ and 
\[
\widetilde{R}_{k}^{(\ell)}=\sum_{j=0}^{2d} 
v'_j\left(\frac{i\cdot \widetilde{R}^{(\ell)}_{k-1}\cdot v'_j(f)}{\sum_{\ell=0}^{2d}v'_\ell(f)^2}\right)\quad \mbox{for $k=1,2,\cdots, r-1$.}
\] 
(Here and below we suppose that $v'_j$ are the directional derivatives along $v_j$ with respect to the variable $y'$.) 

By induction on $0\le k\le \mu'$, we can show the following claim.
\begin{sublemma} 
For any multi-indices $\alpha_1, \alpha_2, \beta,\alpha',\beta'\in (\integer_+)^{2d+1}$, we have
\begin{align*}
\|D^{\alpha_1}_\xi D^{\alpha_2}_{\xi'} D^\beta_{\eta}
 D^{\alpha'}_y D^{\beta'}_{y'} \widetilde{R}_{k}\|_{L^\infty}
 \le& C_*(\alpha_1,\alpha_2,\beta, \alpha',\beta',k)\cdot \|g\|_* \\
 &\quad \cdot 2^{-k\cdot  (D(\gamma,\gamma')-n(\gamma')/2)-r_*\cdot\Delta(\gamma,\gamma')}\\
&\quad \cdot 2^{-|\alpha_1|n(\gamma')/2-|\alpha_1|_\dag |m(\gamma')|-|\beta|n(\gamma)/2-|\beta|_\dag|m(\gamma)|}\\
&\quad \cdot 2^{-|\alpha_2| D(\gamma,\gamma')+|\alpha'|n(\gamma')/3
+|\beta'|n(\gamma)/2
}
\end{align*}
where the coefficient $C_*(\alpha_1,\alpha_2,\beta, \alpha',\beta',k)$ does not depend on $G$ nor $g$. 
\end{sublemma}
\begin{proof} 
Recall (\ref{eqn:dya}) and (\ref{eqn:dya2}) in the proof of Sublemma \ref{sublm:dddd}. 
We can show the claim in the case $k=0$ using these estimate. Note that the condition (P2) in the definition of $\phi_{\gamma\gamma'}^{(\ell)}$ implies that
\[
\sum_{\ell=0}^{2d}v'_\ell(f)^2\ge 2^{2(D(\gamma,\gamma')+\ell-13)}.
\]
Also we have $D^\alpha_y f\equiv 0$ for any $\alpha$ with $|\alpha|\ge 2$.  Using these estimates, we can show the required estimates by induction on $k$. 
\end{proof}
By the same argument as in the last part of the proof of Lemma~\ref{lm:kest1} using in addition the fact that the $(2d+1)$-dimensional volume of the support of $\phi^{(\ell)}_{\gamma\gamma'}$ is bounded by $C_*\cdot 2^{(2d+1)(D(\gamma,\gamma')+\ell)}$, we see that the claim of the sublemma above for $k=\mu'$  implies that  there exists a constant $C_*$, which is independent of $G$, $g$,  $\lambda$ and $\Lambda$, such that 
\begin{align*}
&|\kappa_{\gamma\gamma'}^{(\ell)}(x',x)|\le C_*\cdot \|g\|_*\cdot  2^{-r_*\cdot\Delta(\gamma,\gamma')-\mu' (D(\gamma,\gamma')+\ell-n(\gamma')/2)}\\
&
\qquad\qquad\qquad \cdot
\int\left(\int_{\Ze(\gamma')} b_{\gamma'}^\mu(x'-y') b_{\gamma\gamma',\ell}^{\mu+2d+2}(y'-y)  dy'\right) b_{\gamma}^\mu(G(y)-x) dy.
\end{align*}

From the inequality  (\ref{eqn:R4rel}), there exists a constant $C_*>0$ such that 
\[
\int_{\Ze(\gamma')} b_{\gamma'}^\mu(x'-y') \cdot b_{\gamma\gamma',\ell}^{\mu+2d+2}(y'-y) dy'\le C_* \cdot  d_{\gamma'}^{-2d-2}(y)\cdot b_{\gamma'}^\mu(x'-y).
\]
Therefore we conclude  the inequality in Lemma \ref{lm:kest2}, 
by putting this inequality in the estimates on $\kappa_{\gamma\gamma'}^{(\ell)}(x',x)$ that we obtained above in the case $\ell=0$ and $\ell \ge 1$. 
\end{proof}

\begin{remark}\label{rem:consth}
We can construct  the partitions of unity $\{\phi_{\gamma\gamma'}^{(\ell)}\}_{\ell\ge 0}$ with the properties (i), (ii) and (iii) in the proof above as follows: 
Let $K_{\gamma\gamma'}^{(\ell)}$ be the $2^{D(\gamma,\gamma')+\ell-12}$-neighborhood  of $\supp\; \psi_{\gamma'}$. Also we define $\phi_0:E\to \real$ by
\[
\phi_0(\eta)=
\left(\int \chi(\|\xi\|) d\xi\right)^{-1} \cdot \chi(\|\eta\|)
\]
where $\chi$ is the function taken in Subsection \ref{ss:poe}, 
so that it is supported on the disk with radius $5/3$ and satisfies $\int \phi_0 d\eta=1$.  
Then we set
\[
H_{\gamma\gamma'}^{(\ell)}(\xi)=
{2^{-(2d+1)(D(\gamma,\gamma')+\ell-13)}}
\int_{K_{\gamma\gamma'}^{(\ell)}}\phi_0
\left(2^{-(D(\gamma,\gamma')+\ell-13)}\cdot \|\xi-\eta\|\right) d\eta
\]
The function $H_{\gamma\gamma'}^{(\ell)}$ is supported on $K_{\gamma\gamma'}^{(\ell+1)}$ and satisfies $H_{\gamma\gamma'}^{(\ell)}\equiv 1$ on $K_{\gamma\gamma'}^{(\ell-1)}$. 
From  (\ref{eqn:R4rel}), the required properties are fulfilled if we set 
\[
\phi_{\gamma\gamma'}^{(0)}(\xi)=H_{\gamma\gamma'}^{(0)}(\xi)\quad \text{and}\quad 
\phi_{\gamma\gamma'}^{(\ell)}(\xi)=H_{\gamma\gamma'}^{(\ell)}(\xi)-H_{\gamma\gamma'}^{(\ell-1)}(\xi)\quad \text{for $\ell\ge 1$.}
\]
\end{remark}

\section{The central part of the operator $\cM$ }\label{sec:center}

In this section, we consider the central part $\cM_1$  defined in Subsection~\ref{ssec:decoM}. Our goal is to prove the following proposition.
\begin{proposition} 
\label{th:m1}
The formal definition of $\cM_1$  in fact gives a bounded operator $\cM_1:\bB^\beta_\nu\to \bB^\beta_{\nu'}$ for any $\nu,\nu'\ge 2d+2$. Further, for the case $\nu=\nu'=\nu_*$, there exists a constant $C_*>0$ such that we have
\begin{equation}\label{eqn:m1}
\|\cM_1(\bu)\|_{\beta,\nu_*}^{(\lambda)}\le C_* \cdot \|g\|_{*}\cdot 2^{-(1-\epsilon)
\Lambda/2} \|\bu\|_{\beta,\nu_*}^{(\lambda)}\quad \text{for\/  $\bu\in\bB^\beta_{\nu_*}$},
\end{equation}
for $G\in \cH(\Lambda,\lambda)$ and $g\in\cC^r(V')$, provided  $\Lambda\ge \Lambda_*$, $\lambda\ge \lambda_*$ and $\Lambda\ge d\lambda$. 
\end{proposition}

Clearly Theorem~\ref{th:reduced} follows from Proposition \ref{lm:m0}, \ref{pp:m2}, \ref{pp:m3}, \ref{pp:m4} and Proposition \ref{th:m1} above,  if we set $\mathcal{K}(G,g)=\cM_0$ and $\|\cdot \|^{(\lambda)}=\|\cdot \|^{(\lambda)}_{\beta,\nu_*}$.

\subsection{Reduction of the claim}

For integers $n,n'\ge 0$, we set
\[
\cR^{(n,n')}(1)=\{(\gamma,\gamma')\in \cR(1)\mid n(\gamma)=n,\; n(\gamma')=n'\}
\]
and let $\cM_1^{(n,n')}:\bB^\beta_\nu\to \bB^\beta_{\nu'}$ be the operator defined formally by (\ref{eqn:defM}) with $\cR(j)$ replaced by $\cR^{(n,n')}(1)$. Then $\cM_1$ is formally the sum of $\cM_1^{(n,n')}$ for $(n,n')\in \integer_+\times \integer_+$ such that $\max\{n,n'\}>K$ and $|n'-n|\le 1$. 
From the definition of the norm $\|\cdot\|^{(\lambda)}_{\beta,\nu_*}$,  Proposition \ref{th:m1} follows if we prove the claim (\ref{eqn:m1}) with $\cM_1$ replaced by $\cM_1^{(n,n')}$ and with the constant $C_*$ independent of $n$ and $n'$.

Let $\talpha_0$ be the contact form defined by 
\[
\widetilde{\alpha}_0= dx_0+x^{-}\cdot dx^{+}.
\]
Then it holds  $H_0^*(\alpha_0)=\widetilde{\alpha}_0$ for the  diffeomorphism $H_0:E\to E$ defined by
\[
H_0(x_0, x^+, x^-)=(x_0+2^{-1}x^+\cdot x^-, \; 2^{-1/2} x^+,\; 2^{-1/2} x^-).
\]
In the proof below, we regard the diffeomorphism $G:V'\to V$ in $\cH(\lambda,\Lambda)$ as the composition of two contact diffeomorphisms
\[
\begin{CD}
(V',\alpha_0) @>{H_0^{-1}}>> (H_0^{-1}(V'),\talpha_0) @>{G\circ H_0}>> (E,\alpha_0).
\end{CD}
\]
Also we will introduce a Hilbert space $\tB$ and regard $\cM_1^{(n,n')}:\bB^\beta_\nu\to \bB^\beta_{\nu'}$ as the composition of two operators $\cP^{(n')}$ and $\cQ^{(n)}$,
\[
\begin{CD}
\bB^\beta_\nu @>{\cQ^{(n)}}>> \tB @>{\cP^{(n')}}>> \bB^\beta_{\nu'},
\end{CD}
\]
which are associated to the diffeomorphisms $H_0^{-1}$ and $G\circ H_0$ respectively.
As we will see in the next subsection, the operator $\cP^{(n')}$ does nothing harmful and hence
Proposition \ref{th:m1} is  reduced to a proposition on the operator $\cQ^{(n)}$. 

\begin{remark}
The reason for taking this roundabout way is that we need to "straighten" the contact form $\alpha_0$ along the subspace  $E_0\oplus E^+$ so that we can use the formula (\ref{eqn:intbypart}) of integration by parts appropriately in the last part of the proof. 
\end{remark}

We define the transfer operators 
\[
P:C^r(H_0^{-1}(V'))\to C^r(V')\quad \text{and}\quad 
Q:C^r(V)\to C^r(H_0^{-1}(V'))
\]
by $
P u= u\circ H_0^{-1}$ and $Q u= \hat{g}\cdot ( u\circ \hG)$
respectively, where (and henceforth) we set $
\hat{g}=g\circ H_0$ and $\hG=G\circ H_0$. Obviously we have $\cL=P\circ Q$. 

The definition of the Hilbert space $\tB$ is somewhat  similar to that of $\bB^\beta_\nu$. 
We consider the set $\Sigma=\cN \oplus (\integer_+)$  as the index set instead of $\Gamma$. To refer the components of an element 
$\sigma=(n,k,\ell)\in \Sigma$, we set $
n(\sigma)=n$, $k(\sigma)=k$ and $\ell(\sigma)=\ell$. 
For each $\sigma\in \Sigma$, we define 
the functions $\Psi_{\sigma}:E^*\to [0,1]$ and $\widetilde{\Psi}_{\sigma}:E^*\to [0,1]$ by
\[
\Psi_{\sigma}(\xi)=\chi_{n(\sigma),k(\sigma)}(\xi)\cdot \chi_{\ell(\sigma)}(2^{-n(\sigma)/2-2\delta \lambda} \|\pi_{-}^*(\xi)\|)
\]
and
\[
\widetilde{\Psi}_{\sigma}(\xi)=\widetilde{\chi}_{n(\sigma),k(\sigma)}(\xi)\cdot 
\widetilde{\chi}_{\ell(\sigma)}(2^{-n(\sigma)/2-2\delta \lambda} \|\pi_{-}^*(\xi)\|)
\]
respectively, where $\pi_{-}^*(\xi)=\xi^-$ for $\xi=(\xi_0, \xi^+, \xi^-)$ and the functions $\chi_{n,k}$, $\widetilde{\chi}_{n,k}$, $\chi_n$ and $\widetilde{\chi}_n$ are those defined in Section~\ref{sec:part}. 
By definition, the family $\{\Psi_{\sigma}\}_{\sigma\in \Sigma}$ is a partition of unity on $E^*$ and we have $\Psi_{\sigma}\cdot 
\widetilde{\Psi}_{\sigma}\equiv \Psi_{\sigma}$ for each $\sigma\in \Sigma$. Note that the functions $\Psi_{\sigma}(\xi)$ and $\widetilde{\Psi}_{\sigma}(\xi)$ do not depend on the component $\xi^+$ and hence their inverse Fourier transforms are
not functions in the usual sense but the tensor products of the Dirac $\delta$-function on $E_+$ at the origin and  rapidly decaying functions on $E_0\oplus E_-$. For $\mu\ge 2d+2$, there exists a constant $C_*>0$ such that 
\begin{equation}\label{eqn:Fp}
|\Psi_{\sigma}(D)u(x)|= |\Fourier^{-1}\Psi_{\sigma}* u(x)|
\le C_* \cdot  |b^\mu_\sigma*|u| (x)|
\end{equation}
where $b^\mu_\sigma$ is the finite measure on $E$ defined 
by\begin{equation}\label{eqn:bsig}
b^\mu_\sigma(x)= 
\frac{2^{d(n(\sigma)/2+\ell(\sigma)+2\delta\lambda)+n(\sigma)/2}}{  \langle 2^{n(\sigma)/2+\ell(\sigma)+2\delta\lambda} x^- \rangle^{\mu}\cdot \langle 2^{n(\sigma)/2} x_0 \rangle^{\mu}}\cdot \delta(x^+)
\end{equation}
for $x=(x_0,x^+,x^-)$.  
For $\sigma\in \Sigma$, we set\footnote{The factors $(\delta\lambda)^{1/2}$ and $2^{-\Lambda}$ in the definition of $\widetilde{w}(\sigma)$ are not very essential. We  put those factors in order to make the statements a little simpler.}
\[
\widetilde{w}(\sigma)=
\begin{cases}
(\delta \lambda)^{1/2}, &\text{if $\ell(\sigma)=0$};\\
2^{- \Lambda-\ell(\sigma)}, &\text{if $\ell(\sigma)>0$.}\end{cases}
\]
Then we define the Hilbert space $\tB$ as the linear space 
\[
\tB=\left\{(v_{\sigma})_{\sigma\in \Sigma}\;\left|\; 
v_\sigma\in L^2(E),\;
\widetilde{\Psi}_{\sigma}(D)v_{\sigma}=v_{\sigma}, \;  \sum_{\sigma}\widetilde{w}(\sigma)^2 \|v_{\sigma}\|_{L^2}^2<\infty\right.\right\}
\]
equipped with the norm $\|\cdot\|_{\tB}$ defined by
\[
\|\mathbf{v}\|_{\tB}=
\left(\sum_{\sigma\in \Sigma} \widetilde{w}(\sigma)^2 \|v_{\sigma}\|_{L^2}^2
\right)^{1/2}\qquad \text{for } \mathbf{v}=(v_{\sigma})_{\sigma\in \Sigma}\in \tB.
\]

For $n\ge 0$, $\sigma\in \Sigma$ and $\gamma\in \Gamma$, we define the operators 
\[
\cP^{(n)}_{\sigma\gamma}:L^2(E)\to L^2(E)\quad \mbox{ and }\quad \cQ^{(n)}_{\gamma\sigma}:L^2(E)\to L^2(E)
\]
 respectively by
\[
\cP_{\sigma\gamma}^{(n)}(v)=
\begin{cases}
p_{\gamma}(x,D)^*  (P(\widetilde{\Psi}_{\sigma}(D) v)), & \text{if $|m(\gamma)|\le \delta \lambda$ and $n(\gamma)=n$;}\\
0, &\text{otherwise,}
\end{cases}
\]
and
\[
\cQ_{\gamma\sigma}^{(n)}(u)=
\begin{cases}
\Psi_{\sigma}(D)( Q( \widetilde{\psi}_{\gamma}(D) u)), &\text{if $|m(\gamma)|\le \delta \lambda$ and $n(\gamma)=n$;}\\
0,&\text{otherwise.}
\end{cases}
\]
Then we define  
$
\cP^{(n)}: \tB \to \bB^\beta_{\nu'}$ and $ 
\cQ^{(n)}: \bB^\beta_{\nu}\to \tB$ 
formally by
\[
\cP^{(n)}((v_{\sigma})_{\sigma\in \Sigma})=
\left(\sum_{\sigma\in \Sigma} \cP^{(n)}_{\sigma\gamma}(v_{\sigma})\right)_{\gamma\in \Gamma}
\]
and
\[
\cQ^{(n)}((u_{\gamma})_{\gamma\in \Gamma})=
\left(\sum_{\gamma\in \Gamma} \cQ^{(n)}_{\gamma\sigma}(u_{\gamma})\right)_{\sigma\in \Sigma}.
\]
By the definitions,  we have $
\cM_{1}^{(n,n')}=\cP^{(n')}\circ \cQ^{(n)}$ at the formal level. Therefore, in order to prove  Proposition \ref{th:m1}, it is enough to show the following two propositions. 
\begin{proposition}\label{pp:cP}
The formal definition of the operator $\cP^{(n)}$ for $n\ge K$ in fact gives a bounded operator $\cP^{(n)}:\tB\to  \bB^\beta_{\nu'}$ for each $\nu'\ge 2d+2$. Further, for $\nu'\ge 2d+2$, there exists a constant $C_*>0$ such that we have 
\[
\|\cP^{(n)}(\mathbf{v})\|^{(\lambda)}_{\beta,\nu'}\le C_*\|\mathbf{v}\|_{\tB}\quad \text{ for all\/ $\mathbf{v}\in \tB$ and $n\ge K$,}
\]
provided $\Lambda\ge \Lambda_*$ and $\lambda\ge \lambda_*$. 
\end{proposition}
\begin{proposition}\label{pp:cQ}
The formal definition of the operator $\cQ^{(n)}$  for $n\ge K$ in fact gives a bounded operator $\cQ^{(n)}:\bB^\beta_{\nu}\to \tB$ for each $\nu\ge 2d+2$. For each $\nu\ge 2d+2$, the operator norms of $\cQ^{(n)}: \bB^\beta_{\nu}\to \tB$  are bounded uniformly for $n\ge K$. 
Further, for the case $\nu=\nu_*$, there exists a constant $C_*>0$ such that  we have
\[
\|\cQ^{(n)}(\bu)\|_{\tB}\le C_* \|g\|_*\cdot 2^{-(1-\epsilon)\Lambda/2}
\|\bu\|^{(\lambda)}_{\beta,\nu_*}\quad \text{ for all\/ $\bu\in \bB^\beta_{\nu_*}$ and $n\ge K$}
\]
for $G:V'\to V$ in $\cH(\Lambda,\lambda)$ and $g\in\cC^r(V')$, provided  $\lambda\ge \lambda_*$, $\Lambda\ge \Lambda_*$ and $\Lambda\ge d\lambda$. 
\end{proposition}

In the following subsections, we prove Proposition \ref{pp:cP} and \ref{pp:cQ}. We henceforth consider a fixed $n\ge K$ and write $\cP$, $\cP_{\sigma\gamma}$, $\cQ$ and $\cQ_{\gamma\sigma}$ respectively for $\cP^{(n)}$, $\cP_{\sigma\gamma}^{(n)}$, $\cQ^{(n)}$ and $\cQ_{\gamma\sigma}^{(n)}$ for simplicity, though we keep dependence of them  on $n$ in mind. 
Notice that we will write  $C_*$, $C(G)$ and $C(G,g)$ only for constants that do not depend on $n$.

\subsection{The operator $\cP$}
In this subsection, we consider the operator $\cP=\cP^{(n)}$  and  prove Proposition \ref{pp:cP}. 
The structure of the proof is similar to that of Proposition \ref{pp:m2}. 
Fix some integers $\mu\ge \nu'+2d+2$ and $\mu'>2\Lambda/(\delta\lambda)$. 
For $\sigma\in \Sigma$ and $k\in \integer$ such that $(n,k)\in \cN$, we set 
\begin{align*}
K_{\sigma, k}&=2^{-\mu' (\Delta(n(\sigma), k(\sigma), n,k)+\delta \lambda +\ell(\sigma))}\cdot   (1/\widetilde{w}(\sigma))
\intertext{
if $\ell(\sigma)>0$ and $n/2\le n(\sigma)/2+\ell(\sigma)$, and otherwise we set}
K_{\sigma, k}&=2^{-2\mu'\cdot  \Delta(n(\sigma), k(\sigma), n,k)} \cdot  
(1/\widetilde{w}(\sigma)).
\end{align*}

We use the following sublemma, whose proof is postponed for a while. 
\begin{sublemma}\label{subl:4} There exists a constant $C_*>0$ such that 
\[
\sup_{\sigma} 
\left(\sum_{k:(n,k)\in \cN} K_{\sigma, k}\right)\le 
\frac{C_*}{ \sqrt{\delta \lambda}}
\quad\text{and}\quad
\sup_{k:(n,k)\in \cN}
\left( \sum_{\sigma} K_{\sigma, k}\right) \le \frac{C_*}{ \sqrt{\delta \lambda}}.
\]
\end{sublemma}

Consider a pair $(\sigma,\gamma)\in \Sigma\times \Gamma$ such that $n(\gamma)=n$ and $|m(\gamma)|\le \delta \lambda$. We regard the operator $\cP_{\sigma\gamma}$  as an integral operator
\[
\cP_{\sigma\gamma} u(x')= (2\pi)^{-2(2d+1)}\int \kappa_{\sigma\gamma}(x',x) u(x) dx
\]
with the kernel
\begin{equation}\label{eqn:kgs}
\kappa_{\sigma\gamma}(x',x)=
\int  e^{i\langle \xi, x'-y\rangle+i\langle \eta, H_0^{-1}(y)-x\rangle} \rho_{\gamma}(y)
\psi_{\gamma}(\xi) \widetilde{\Psi}_{\sigma}(\eta) dy d\xi d\eta.
\end{equation}
In order to apply the formula (\ref{eqn:intbypart}) of integration by parts to this kernel, we prepare two estimates. The first  is that 
\begin{align*}
d\left(\pi^*_0((DH_0)_{y}^*(\supp\widetilde{\Psi}_{\sigma})),
\pi^*_0( \supp\psi_{\gamma})\right)&\ge d\left(\supp \widetilde{\chi}_{n(\sigma), k(\sigma)}, \supp \chi_{n(\gamma), k(\gamma)}\right)\\
&\ge 2^{\Delta(n(\sigma),k(\sigma), n(\gamma),k(\gamma))+n(\gamma)/2}
\end{align*}
for all $y\in E$ 
when $\Delta(n(\sigma),k(\sigma), n(\gamma),k(\gamma))> 0$. 
This follows immediately from the definitions. The second  is that 
\begin{equation}\label{eqn:sepsupp}
d\left((DH_0^{-1})_{y}^*(\supp\widetilde{\Psi}_{\sigma}), \supp\psi_{\gamma}\right)\ge 2^{n(\sigma)/2+\delta\lambda+\ell(\sigma)}
\end{equation}
for all $y\in \supp \rho_\gamma$ 
if $\ell(\sigma)>0$  and 
$
n(\gamma)/2\le n(\sigma)/2+\ell(\sigma)$. This can be checked as follows. By the assumption $\ell(\sigma)>0$, the  support of $\widetilde{\Psi}_{\sigma}$ is contained in the region
\[
\{ \xi=(\xi_0, \xi^+, \xi^-)\mid |\xi_0|\le 2^{n(\sigma)+1}, \|\xi_-\| \ge 2^{n(\sigma)/2+\ell(\sigma)+2\delta\lambda-1}\}.
\]
Then, from the definition of $H_0$,  $(DH_0^{-1})_{y}^*(\supp\widetilde{\Psi}_{\sigma})$ should be contained in
\[
\{ \xi=(\xi_0, \xi^+, \xi^-)\mid |\xi_0|\le 2^{n(\sigma)+1}, \|\pi^*_-(\xi-\xi_0\cdot \alpha(y))\| \ge 2^{n(\sigma)/2+\ell(\sigma)+2\delta\lambda-1}\}.
\]
On the other hand, $\supp\psi_{\gamma}$ is contained in 
\[
\{ \xi=(\xi_0, \xi^+, \xi^-)\mid |\xi_0|\le 2^{n(\gamma)+2}, \|\pi^*_-(\xi-\xi_0\cdot \alpha(z(\gamma)))\| \le 2^{n(\gamma)/2+|m(\gamma)|+2}\}
\]
where $|m(\gamma)|\le \delta \lambda$. Since we assume $y\in \supp \rho_\gamma$, we have
\[
\|\alpha(y)-\alpha(z(\gamma))\|=\|y-z(\gamma)\|\le \sqrt{2d+1}\cdot 2^{-n(\gamma)+1}
\]
Therefore (\ref{eqn:sepsupp}) follows from the assumption $
n(\gamma)/2\le n(\sigma)/2+\ell(\sigma)$, provided that $\delta\lambda$ is sufficiently large and, if fact, the choice that we made in Subsection \ref{ssec:remarks} is quite enough. 

From the two estimates prepared above, we  obtain the following estimateson the 
kernel $\kappa_{\sigma\gamma}(x',x)$ of the operator $\cP_{\sigma\gamma}$.
\begin{lemma}\label{lm:Psg}
There exists a constant $C_*>0$  such that 
\begin{align*}
|\kappa_{\sigma\gamma}(x',x)|\le  C_*  \cdot 
K_{\sigma,  k(\gamma)}\cdot \widetilde{w}(\sigma) 
\int_{\Ze(\gamma)} b^\mu_{\gamma}(x'-y) b^\mu_{\sigma}(H_0(y)-x) dy
\end{align*}
for $(\sigma, \gamma)\in \Sigma\times \Gamma$ such that $|m(\gamma)|\le \delta\lambda$.
\end{lemma}
\begin{proof} Since the poof is almost parallel to the former part of the proof of 
Lemma \ref{lm:kest2}, we only outline the argument. We take an orthonormal basis $\{v_j\}_{j=0}^{2d}$ of $E$ including $v_0=\partial/\partial x_0$. First, to the integral with respect to $y$ in (\ref{eqn:kgs}),
we apply the formula (\ref{eqn:intbypart}) of integration by parts  
along the single vector $v_0$ for $2\mu'$ times  if $\Delta(n(\sigma),k(\sigma), n(\gamma),k(\gamma))> 0$, and do nothing otherwise. 
Second, to the result of the previous step,  we  apply the formula (\ref{eqn:intbypart}) of integration by parts  
along the set of vectors $\{v_j\}_{j=0}^{2d}$ for $\mu'$~times if  $\ell(\sigma)>0$  and  $n(\gamma)/2\le n(\sigma)/2+\ell(\sigma)$ (and again do nothing otherwise). For these two steps, we can  proceed just as in the former part of  proof of Lemma \ref{lm:kest2} with obvious changes. 
In the places where we used the estimates on the derivatives of $\widetilde{\psi}_\gamma$, we use the estimate that
\[
\|D_\xi^\alpha \widetilde{\Psi}_{\sigma}\|_{L^\infty}\le C_*(\alpha) \cdot 
2^{-|\alpha|n(\sigma)/2-|\alpha|_\dag (2\delta\lambda+\ell(\sigma)) }
\]
for $\alpha=(\alpha_0, \alpha_1^-, \cdots, \alpha_d^-)\in (\integer_+)^{d+1}$, where $|\alpha|_\dag=|\alpha|-\alpha_0$. (Notice that the function $\widetilde{\Psi}_{\sigma}(\xi)$ does not depend on the variable $\xi^+$ in $\xi=(\xi_0, \xi^+, \xi^-)$.) Also we use 
(\ref{eqn:sepsupp}) 
in the place where we used (\ref{eqn:cldiv}). Then, as the result, we obtain the claim of the lemma. 
\end{proof}

Once we have Lemma \ref{lm:Psg}, we can proceed as in the proof of Proposition~\ref{pp:m2}.   By Young inequality,  we get
\[
\left(\|\cP^{(n)}(\mathbf{v})\|_{\beta,\nu'}^{(\lambda)}\right)^2\le 
\sum_{k:(n,k)\in \cN}
\sum_{m: |m|\le \delta \lambda}\;\;\left\| \sum_{\sigma\in \Sigma} K_{\sigma,k}\cdot  \widetilde{w}(\sigma)\cdot |b_\sigma^\mu *v_{\sigma}| 
\right\|_{L^2}^2
\]
Then, by Schwarz inequality, Sublemma \ref{subl:4} and Remark \ref{rem:card}, we obtain
\[
\left(\|\cP^{(n)}(\mathbf{v})\|_{\beta,\nu'}^{(\lambda)}\right)^2 \le 
C_* \sum_{\sigma\in \Sigma}
  \widetilde{w}(\sigma)^2\cdot
 \left\|v_{\sigma}\right\|_{L^2}^2= C_* \|\mathbf{v}\|_{\tB}^2
\]
for $\mathbf{v}=(v_\sigma)_{\sigma\in \Sigma}\in \tB$. This is nothing but the claim of Proposition \ref{pp:cP}.
Now we finish the proof by proving Sublemma \ref{subl:4}.
\begin{proof}[Proof of Sublemma \ref{subl:4}]
We prove the former inequality. 
The latter can be proved similarly. If $\ell(\sigma)=0$, the sum $\sum_{k:(n,k)\in \cN}  K_{\sigma,k}$ is bounded by
\begin{align*}
C_* (\delta\lambda)^{-1/2} \sum_{k:(n,k)\in \cN} 2^{-2\mu'\cdot  \Delta(n(\sigma), k(\sigma), n,k)}\le C_* (\delta\lambda)^{-1/2}
\end{align*}
from  (\ref{eqn:chisum}).
If $\ell(\sigma)>0$, the sum  $\sum_{k:(n,k)\in \cN}  K_{\sigma,k}$ is bounded by
\[
C_*  2^{\Lambda+\ell(\sigma)} 
\left(\sum_{k}^{*} 2^{-\mu'\cdot  (\Delta(n(\sigma), k(\sigma), n,k)+\delta \lambda+\ell(\sigma))}+ \sum_{k}^{**} 2^{-2\mu'\cdot  \Delta(n(\sigma), k(\sigma), n,k)}\right)
\]
where $\sum_{k}^{*}$ (resp. $\sum_{k}^{**}$) denotes the sum over $k\in \integer$ such that $(n,k)\in \cN$ and 
$n\le n(\sigma)/2+\ell(\sigma)$ (resp. $n>n(\sigma)/2+\ell(\sigma)$).
For the first sum, we have
\[
\sum_{k}^{*} 2^{-\mu'\cdot  (\Delta(n(\sigma), k(\sigma), n,k)+\delta \lambda+\ell(\sigma))}<C_* 2^{-\mu' \delta \lambda-\mu'\ell(\sigma) }
\]
from (\ref{eqn:chisum}).  The conditions $\ell(\sigma)>0$ and $n/2>n(\sigma)/2+\ell(\sigma)$  imply  that 
$|n-n(\sigma)|>2$ and hence, by (\ref{eqn:dnk}),  that 
\[
\Delta(n(\sigma), k(\sigma), n,k)\ge n/2-3\ge K/2-3.
\]
Hence, 
for the second sum, we have
\begin{align*}
\sum_{k}^{**} 2^{-2\mu'  \Delta(n(\sigma), k(\sigma), n,k)}&\le 2^{-(K/3+3)\mu'} \sum_{k}^{**} 2^{-\mu'  \Delta(n(\sigma), k(\sigma), n,k)}\le C_* 2^{-(K/3+3)\mu'}
\end{align*}
where we used (\ref{eqn:chisum}) in the second inequality. 
Therefore, if we take sufficiently large $K$, we have 
\[
\sum_{k:(n,k)\in \cN}  K_{\sigma,k}\le C_* 2^{\Lambda -\mu'\cdot \delta \lambda}.
\]
Since we have $\mu'\cdot \delta \lambda>2\Lambda>\Lambda+\lambda$ from the choice of $\mu'$, this implies the former inequality in the case $\ell(\sigma)>0$. 
\end{proof}

\subsection{The operator $\cQ$} 
In the remaining part of this section, we consider the operator $\cQ=\cQ^{(n)}$ and prove Proposition \ref{pp:cQ}. 
Consider $(\gamma,\sigma)\in \Gamma\times \Sigma$ such that $n(\gamma)=n$ and $|m(\gamma)|\le \delta \lambda$. 
We regard the operator $\cQ_{\gamma\sigma}$ as an integral operator
\[
\cQ_{\gamma\sigma} u(x')= (2\pi)^{-2(2d+1)}\int \kappa_{\gamma\sigma}(x',x) u(x) dx
\]
with the kernel
\[
\kappa_{\gamma\sigma}(x',x)=\int  e^{i\langle \xi, x'-y\rangle+i\langle \eta, \hG(y)-x\rangle} \hat{g}(y) 
{\Psi}_{\sigma}(\xi) \widetilde{\psi}_{\gamma}(\eta)  dy d\xi d\eta.
\]
We can show the following estimate in the same way as Lemma \ref{lm:kest0} and \ref{lm:kest1}. 
\begin{lemma}\label{lm:kest3}
For $\mu\ge 2d+2$,
there exists a constant $C_*>0$ such that 
\[
|\kappa_{\gamma\sigma}(x',x)|\le 
C_*\cdot  \|g\|_{L^\infty} \int b^\mu_{\sigma}(x'-y)\cdot  b^\mu_{\gamma}(\hG(y)-x) dy
\]
for $x,x'\in E$ and $(\gamma,\sigma)\in \Gamma\times \Sigma$ and, further, that 
\begin{align*}
&|\kappa_{\gamma\sigma}(x',x)| \\
&\le C_* \|g\|_{*}\cdot 2^{-r_* (\Delta(n(\gamma), k(\gamma), n(\sigma), k(\sigma))+n(\sigma)/2)}   \cdot  \int b^\mu_{\sigma}(x'-y) b^\mu_{\gamma}(\hG(y)-x) dy
\end{align*}
whenever $\Delta(n(\gamma), k(\gamma), n(\sigma), k(\sigma))>0$.

\end{lemma}
\begin{remark}
Notice that we have the additional term $n(\sigma)/2$ in the second claim above compared with Lemma \ref{lm:kest1}. This is because there is no longer the term $\rho_{\gamma'}$ which produced the factor $2^{n(\gamma')/2}$ for each differentiation.  
\end{remark}

Let $S=S(n)$ be the set of pairs $(\gamma,\sigma)\in \Gamma\times \Sigma$ such that 
\[
n(\gamma)=n,\quad |m(\gamma)|\le \delta \lambda, \quad 
\ell(\sigma)=0 \quad\text{and}\quad \Delta(n(\gamma), k(\gamma), n(\sigma), k(\sigma))=0.
\]
We define the operator $\widehat{\cQ}=\widehat{\cQ}^{(n)}:\bB^\beta_\nu\to \tB$ formally by
\[
\widehat{\cQ}(\bu)=\left(\sum_{\gamma:(\gamma,\sigma)\in S} \cQ_{\gamma\sigma}(u_\gamma)\right)_{\sigma\in \Sigma}\quad \text{ for $\bu=(u_\gamma)_{\gamma\in \Gamma}\in \bB^\beta_\nu$.}
\]
This is actually the main part of the operator $\cQ$ and considered  in the following two subsections. 
The next lemma tells that the remainder part $\cQ -\widehat{\cQ}:\bB^\beta_\nu\to \tB$ of $\cQ$, defined by
\[
(\cQ -\widehat{\cQ})(\bu)=\left(\sum_{\gamma:(\gamma,\sigma)\notin S} \cQ_{\gamma\sigma}(u_\gamma)\right)_{\sigma\in \Sigma}\quad \text{ for $\bu=(u_\gamma)_{\gamma\in \Gamma}\in \bB^\beta_\nu$,}
\]
does not do harm. 
\begin{lemma}\label{lm:ccQ}
The formal definition of $(\cQ -\widehat{\cQ})$ above in fact gives a bounded operator $(\cQ -\widehat{\cQ}):\bB^\beta_\nu\to \tB$ for any $\nu\ge 2d+2$. Further, for $\nu\ge 2d+2$, there exists a constant $C_*>0$, which is independent of $n$, such that we have
\[
\|(\cQ -\widehat{\cQ})(\bu)\|_{\tB}\le 
C_* 2^{-\Lambda/2}\|g\|_*\|\bu\|_{\beta,\nu}^{(\lambda)}\quad \text{for $\bu\in \bB^\beta_{\nu}$}
\]
for any $G:V'\to V$ in $\cH(\Lambda,\lambda)$ and  $g\in\cC^r(V')$,  provided that $\Lambda\ge \Lambda_*$, $\lambda\ge \lambda_*$ and $\Lambda\ge d\lambda$.
\end{lemma}
\begin{proof} 
For $\sigma\in \Sigma$ and $k\in\integer$ such that $(n,k)\in \cN$, we set
\[
K_{k,\sigma}=
\begin{cases}
2^{-r_*(\Delta(n, k, n(\sigma), k(\sigma))+ n(\sigma)/2)} 
\|g\|_*  \widetilde{w}(\sigma),
& \!\text{if $\Delta(n, k, n(\sigma), k(\sigma))>0$;}\\
\|g\|_*\cdot \widetilde{w}(\sigma),
&\!\mbox{if}\left[\parbox{38mm}{$\Delta(n, k, n(\sigma), k(\sigma))=0$ and  $\ell(\sigma)>0$}\right];\\
0,& \text{otherwise.}
\end{cases}
\]
Then we have
\begin{sublemma} There exists a constant $C_*>0$ such that 
\[
\sup_{\sigma\in \Sigma}
\left( \sum_{k:(n,k)\in \cN} K_{k,\sigma} 
\right)\le C_*2^{-\Lambda} \|g\|_*,\quad
\sup_{k:(n,k)\in \cN}
\left( \sum_{\sigma} K_{k,\sigma}
\right)\le C_* 2^{-\Lambda}\|g\|_*.
\]
\end{sublemma}
\begin{proof} 
Note that we are suming $n\ge K$. 
From (\ref{eqn:dnk}), we always have
\[
\Delta(n,k, n(\sigma),k(\sigma))+n(\sigma)/2\ge \max\{ n, n(\sigma)\}/2-3.
\]
If $\Delta(n, k, n(\sigma), k(\sigma))>0$,
this implies
\[
K_{k,\sigma}\le 2^{-(K/2-3)} \cdot 2^{-(r_*-1)(\Delta(n, k, n(\sigma), k(\sigma))+ n(\sigma)/2)} 
\|g\|_*  \widetilde{w}(\sigma).
\]
Therefore, using (\ref{eqn:chisum}), we see that the sums in the claim above restricted to the case $\Delta(n, k, n(\sigma), k(\sigma))>0$ can be bounded by an arbitrarily small constant, if we take sufficiently large $K$. 
We can estimate the sums restricted to the case $\Delta(n, k, n(\sigma), k(\sigma))=0$ by using the definition of $\widetilde{w}(\sigma)$ and recalling Remark \ref{rem:card},   to obtain the claim of the sublemma. 
\end{proof}
By Lemma \ref{lm:kest3} and Young inequality, we have
\begin{align*}
\|(\cQ -\widehat{\cQ})(\bu)\|_{\tB}^2&\le C_*
 \sum_{\sigma} 
\left\|\sum_{m:|m|\le \delta \lambda}\sum_{k:(n,k)\in \cN} K_{k,\sigma} \cdot 
b^{\mu}_{n,m}*\left(\sum_{\gamma:n,k,m}|u_{\gamma}|\right)
\right\|^2_{L^2}
\end{align*}
for $\bu=(u_\gamma)_{\gamma\in \Gamma}\in \bB^\beta_\nu$, where $\sum_{\gamma:n,k,m}$ denotes the sum over $\gamma\in \Gamma$ such that $n(\gamma)=n$, $m(\gamma)=m$ and $k(\gamma)=k$. Hence, by Schwarz inequality, (\ref{eqn:sch}) and the inequalities above on the sums of $K_{k,\sigma}$, we obtain that 
\begin{align*}
\|(\cQ &-\widehat{\cQ})(\bu)\|_{\tB}^2\\
&\le C_*\delta \lambda\cdot 2^{-\Lambda} \|g\|_* \cdot \sum_{\sigma} \sum_{m:|m|\le\delta \lambda}\;\sum_{k:(n,k)\in \cN} K_{k,\sigma}\sum_{\gamma:n,k,m}\|d_\gamma^\nu u_{\gamma}\|_{L^2}^2\\
&\le C_*(\delta \lambda)^2 \cdot 2^{-2\Lambda} \|g\|_*^2\cdot 
\left(\|\bu\|_{\beta,\nu}^{(\lambda)}\right)^2.
\end{align*}
From the assumption $\Lambda\ge d\lambda$, this implies the conclusion of the lemma. 
\end{proof}

\subsection{The operator $\widehat{\cQ}$} 
In this subsection and the next, we consider the operator $\widehat{\cQ}=\widehat{\cQ}^{(n)}:\bB^\beta_\nu\to \tB$.
Using Lemma \ref{lm:kest3}, it is easy to check that the formal definition of $\widehat{\cQ}$ gives a bounded operator $\widehat{\cQ}:\bB^\beta_\nu\to \tB$ and the operator norm is bounded by $C_*(\delta \lambda)\|g\|_*$. 
This and Lemma \ref{lm:ccQ} imply the former statement of Proposition \ref{pp:cQ} on boundedness of $\cQ$.
To prove the latter statement, we  need more precise estimates. We begin with 
\begin{lemma}\label{lm:keypre}
If $(\gamma,\sigma)\in S$ and if $u\in L^2(E)$ satisfies $\widetilde{\psi}_\gamma(D) u=u$ and $\|d_\gamma^{\nu_*}u\|_{L^2}<\infty$,
we have $
\|\cQ_{\gamma\sigma}(u)\|_{L^2}\le  C_* 2^{-\Lambda/2+d\delta \lambda}\|g\|_{L^\infty} \|d_\gamma^{\nu_{*}} u\|_{L^2}$.
\end{lemma}
\begin{proof} By using Schwarz inequality and Young inequality, we have
\begin{align*}
\|\cQ_{\gamma\sigma}(u)\|_{L^2}^2
&\le
C_* \|g\|_{L^\infty}^2 
\left\||
\Fourier^{-1}\Psi_{\sigma}| * \left|(d_\gamma^{\nu_{*}} u)^2\circ \hG\right| 
\cdot |\Fourier^{-1}\Psi_{\sigma}| * \left|d_\gamma^{-{2\nu_{*}}}\circ \hG\right| \right\|_{L^1}
\\
&\le 
C_* \|g\|_{L^\infty}^2 
\| d_\gamma^{\nu_{*}} u\|_{L^2}^2  
\left\||\Fourier^{-1}\Psi_{\sigma}| * \left| d_\gamma^{-{2\nu_{*}}}\circ \hG\right| \right\|_{L^\infty}.
\end{align*}
For the last factor, we have that
\begin{align}\label{eqn:Fsp}
|&\Fourier^{-1}\Psi_{\sigma}| * \left| d_\gamma^{-{2\nu_{*}}}\circ \hG)\right| (x_0,x^+,x^-)\\
&\le C_* 2^{d(n(\sigma)/2+2\delta\lambda)+n(\sigma)/2}
\cdot \int d^{-2\nu_*}_\gamma(G(x_0+y_0, x^++y^+, x^-)) dy_0dy^+
\notag\\
&\le C_* \cdot 2^{d(n(\sigma)/2)+2\delta\lambda)+n(\sigma)/2}\cdot 2^{-\Lambda -(d+1) n(\sigma)/2}\le C_* 2^{-\Lambda+2d\delta\lambda}
\notag
\end{align}
where we used  (\ref{eqn:bsig}) in the first inequality and 
the condition (H4) in the definition of  $\cH(\lambda, \Lambda)$ in the second.  
We therefore obtain the estimate in the lemma. 
\end{proof}
The next lemma, which improves Lemma \ref{lm:keypre} above, is the core  of our argument on the central part.   
\begin{lemma}\label{lm:key}
There exists a constant $C_*>0$ such that 
\[
|\langle \cQ_{\gamma\sigma}(u), \cQ_{\gamma'\sigma}(u')\rangle_{L^2}|
\le C_* \cdot \frac{2^{-\Lambda+2d\delta \lambda}\cdot \|g\|_{L^\infty}^2\cdot 
\|d_\gamma^{\nu_{*}}u\|_{L^2}\cdot \|d_{\gamma'}^{\nu_{*}} u'\|_{L^2} 
}{\langle 2^{n/2-2\delta\lambda}\|z(\gamma)-z(\gamma')\|\rangle^{2d+2}}
\]
for any $(\gamma,\sigma), (\gamma',\sigma)\in S$ and  any $u,u'\in L^2(E)$ satisfying $\widetilde{\psi}_\gamma(D) u=u$,  $\widetilde{\psi}_{\gamma'}(D) u'=u'$, $\|d_\gamma^{\nu_{*}} u\|_{L^2}<\infty$ and $\|d_{\gamma'}^{\nu_{*}} u'\|_{L^2}<\infty$.
\end{lemma}

We first show that the latter claim of Proposition \ref{pp:cQ} follows from 
Lemma~\ref{lm:key}  and Lemma \ref{lm:ccQ}.
From Lemma \ref{lm:key} and Remark \ref{rem:card}, it follows
\begin{align*}
&\sum_{\gamma:(\gamma,\sigma)\in S} \;
\sum_{\gamma':(\gamma',\sigma)\in S} \;
 |\langle \cQ_{\gamma\sigma}(u_\gamma), \cQ_{\gamma'\sigma}(u_{\gamma'})\rangle_{L^2}|\\
 &\qquad \le 2^{-\Lambda+2d\delta\lambda}\cdot 
 \|g\|_{L^\infty}^2\cdot \sum_{\gamma:(\gamma,\sigma)\in S} \;
\sum_{\gamma':(\gamma',\sigma)\in S} \;
\frac{\|d_{\gamma}^{\nu_{*}} u_{\gamma}\|_{L^2}^2 }{\langle 2^{n/2-2\delta\lambda}\|z(\gamma)-z(\gamma')\|\rangle^{2d+2}}\\
&\qquad 
\le C_*\cdot \delta \lambda\cdot 2^{-\Lambda+(6d+2)\delta\lambda}
 \|g\|_{L^\infty}^2\cdot \sum_{\gamma:(\gamma,\sigma)\in S} 
  \;
\|d_{\gamma}^{\nu_{*}} u_{\gamma}\|_{L^2}^2 
\end{align*}
for $\sigma\in \Sigma$ with $\ell(\sigma)=0$.
Taking sum of both sides over $\sigma\in \Sigma$ with $\ell(\sigma)=0$ and recalling Remark \ref{rem:card} again, we obtain
\begin{align*}
\|\widehat{\cQ}(\bu)\|_{\tB}^2&\le \widetilde{w}(0)^2\cdot \sum_{\sigma\in \Sigma: \ell(\sigma)=0} \;
\sum_{\gamma:(\gamma,\sigma)\in S} \;
\sum_{\gamma':(\gamma',\sigma)\in S} \;
 |\langle \cQ_{\gamma\sigma}(u_\gamma), \cQ_{\gamma'\sigma}(u_{\gamma'})\rangle_{L^2}|\\
 &\le  C_*  \cdot (\delta \lambda)^2\cdot 2^{-\Lambda+(6d+2)\delta \lambda}  \|g\|_{L^\infty}^2
 \sum_{\gamma:m(\gamma)\le \delta\lambda} 
\|d_\gamma^{\nu_{*}}u_\gamma\|_{L^2}^2
\\
&\le C_* \cdot (\delta \lambda)^{2}\cdot 2^{-\Lambda+(6d+2)\delta\lambda}\|g\|_{L^\infty}^2 \cdot
 (\|\bu\|_{\beta,{\nu_{*}}}^{(\lambda)})^2
\end{align*}
for $\bu=(u_\gamma)_{\gamma\in \Gamma}\in \bB^\beta_{\nu_{*}}$. 
Therefore the operator norm of $\widehat{\cQ}:\bB^\beta_{\nu_{*}}\to \tB$ is bounded by $C_* 2^{-(1-\epsilon )\Lambda/2}$ from the choice of $\delta$ in Subsection \ref{ssec:remarks}, provided that we consider the norm $\|\cdot\|_{\beta, \nu_*}^{(\lambda)}$ on $\bB^\beta_{\nu_{*}}$. This, together with 
  Lemma \ref{lm:ccQ}, implies the latter claim of Proposition~\ref{pp:cQ}. 

\newcommand\K{\mathcal{K}}
\newcommand{\tK}{\widetilde{\K}}

We prove Lemma \ref{lm:key} in the remaining part of this subsection and in the next subsection. 
We consider $(\gamma,\sigma), (\gamma',\sigma)\in S$ and $u, u'\in L^2(E)$ satisfying the assumptions   in Lemma \ref{lm:key} and  prove the conclusion of Lemma \ref{lm:key} in each of the following four cases separately: \begin{itemize}
\item[(i)] $\|z(\gamma)-z(\gamma')\| \le 2^{-n/2+2\delta \lambda}$, 
\item[(ii)] $\|z(\gamma)-z(\gamma')\| \ge 2^{(-1/2+\tau)n}$ with $\tau=1/(5(d+1))$,
\item[(iii)] neither  (i) nor (ii), but $\|\pi_-(z(\gamma)-z(\gamma'))\|\le \|z(\gamma)-z(\gamma')\|/10$,
\item[(iv)] neither  (i) nor (ii), but  $\|\pi_-(z(\gamma)-z(\gamma'))\|> \|z(\gamma)-z(\gamma')\|/10$.
\end{itemize}
In the first case (i) is the case where the points $z(\gamma)$ and $z(\gamma')$ are  close to each other (relative to the size of the cube $Z(\gamma)$ and $Z(\gamma')$). In this case, the claim of Lemma \ref{lm:key} is an immediate consequence of Lemma \ref{lm:keypre}. 
The second case (ii) and the third case (iii) will turn out to be the cases where $\hG^{-1}(z(\gamma))$ and $\hG^{-1}(z(\gamma'))$ are far from each other. 
The proof of Lemma \ref{lm:key} in these two cases is not difficult and will be  given in the remainder part of this subsection. 
The fourth case (iv) is the most important case where $z(\gamma)$ and $z(\gamma')$ are not close to each other but $\hG^{-1}(z(\gamma))$ and $\hG^{-1}(z(\gamma'))$ may come close to each other by hyperbolicity of~$\hG$.  This last case will be  considered in the next subsection. 
\begin{proof}[Proof of Lemma \ref{lm:key} in the case {\rm (iii)}] 
We first show the following claim, which is a consequence of  geometric properties of $\hG$.  
\begin{sublemma}\label{subl:6}
There exists a constant $C_*>0$ such that, for given point $y\in E$,  either of the following two conditions holds: the condition that
\begin{equation}\label{eqn:spt}
\frac{1}{\langle 2^{n/2} z\rangle
d_\gamma(\hG(y-z))}\le   
\frac{ C_*}
{2^{n/2}\|z(\gamma')-z(\gamma)\|} \quad  \text{ for all $z\in E_0\oplus E_-$,}
\end{equation}
or 
the same condition with $\gamma$ and $\gamma'$ exchanged. 
\end{sublemma}
\begin{remark} If $\hG$ is not defined at $y-z$, we suppose that $d_\gamma(\hG(y-z))=\infty$ and  (\ref{eqn:spt}) holds trivially. 
\end{remark}
\begin{proof}
If $\|z\|\ge \|z(\gamma)-z(\gamma')\|/100$, the inequality in  (\ref{eqn:spt}) with  $C_*=100$ holds obviously and so does the same inequality with $\gamma$ and $\gamma'$ exchanged, because $d_\gamma(\cdot)\ge 1$.
Thus we may assume   
\begin{equation}\label{eqn:zg}
\|z\|<\|z(\gamma)-z(\gamma')\|/100\le 10^{-2}\cdot 2^{(-1/2+\tau)n}\le  10^{-2}\cdot 2^{(-1/2+\tau)K}
\end{equation}
in the conditions in the sublemma. 

Since we are considering the case (iii), we have
\[
\|\pi_{0,+}(z(\gamma)-z(\gamma')\|)\ge (9/10)\cdot  \|z(\gamma)-z(\gamma')\|\ge  9\cdot  \|\pi_-(z(\gamma)-z(\gamma'))\|.
\]
Recall that  $\hG=G\circ H_0$ for $G\in \cH(\lambda,\Lambda)$. Since 
(\ref{eqn:zg}) implies that we may suppose that $\hG$ is well-approximated by its linearization (at $z(\gamma)$, say) if we take large $K$, 
it is not difficult to check that we have either 
\[
d(\hG(y-z),z(\gamma))\ge C_*^{-1} \|z(\gamma)-z(\gamma')\|\quad \mbox{for all $z\in E_0\oplus E_-$ satisfying (\ref{eqn:zg})}
\]
or the same condition with $\gamma$ and $\gamma'$ exchanged, depending on the position of the point $\hG(y)$ relative to $z(\gamma)$ and $z(\gamma')$. This implies the conclusion of the sublemma.
\end{proof}

Let $Y$ be the set of points $y\in E$ for which (\ref{eqn:spt})  holds. 
Then  we have
\begin{align*}
&|\cQ_{\gamma\sigma}(u)(y)|=| \Fourier^{-1}\Psi_{\sigma}*(\hat{g} \cdot (u\circ \hG))(y)|\\
&\le \|g\|_{L^\infty} \int_{E_0\oplus E_-}
\frac{|\Fourier^{-1}\Psi_{\sigma}(z)|\cdot\langle 2^{n/2}z\rangle^{(2d+2)}\cdot  (d_\gamma^{2d+2}\cdot u) \circ \hG(y-z)}{\langle 2^{n/2}z\rangle^{(2d+2)} \cdot d_\gamma^{2d+2}(\hG(y-z))}dz 
\\
&\le \frac{C_*\cdot \|g\|_{L^\infty}}{(2^{n/2}\|z(\gamma')-z(\gamma)\|)^{2d+2}} \\
&\qquad\cdot  \int \left(|\Fourier^{-1}\Psi_{\sigma}(z)|\cdot \langle 2^{n/2} z\rangle^{2d+2}\right) \cdot   (d_\gamma^{2d+2}\cdot  u) \circ \hG(y -z) dz
\end{align*}
for $y\in Y$, where we used (\ref{eqn:spt})  in the latter inequality. 
Recall the estimates (\ref{eqn:Fp}) and (\ref{eqn:bsig}) on the factor $\Fourier^{-1}\Psi_{\sigma}(z)$ and check that $\Fourier^{-1}\Psi_{\sigma}(z)\cdot \langle 2^{n/2} z\rangle^{2d+2}$ enjoys the same estimates  
with the exponent $\mu$ replaced by $\mu-2d-2$.  
Then, by an argument parallel to that in the proof of Lemma \ref{lm:keypre}, we see 
\begin{align*}
&\|\cQ_{\gamma\sigma}(u)\cdot \mathbf{1}_{Y}\|_{L^2}^2\\
&\le \frac{C_*\cdot \|g\|_{L^\infty}^2}{(2^{n/2}\|z(\gamma')-z(\gamma)\|)^{2(2d+2)}} \\
&\qquad 
\cdot \|d_\gamma^{\nu_*}u\|_{L^2}^2 \cdot \int 
\left(|\Fourier^{-1}\Psi_{\sigma}(z)|\cdot \langle 2^{n/2} z\rangle^{2d+2}\right) \cdot   (d_\gamma^{2d+2-\nu_*}) \circ \hG(y -z) dz\\
&
\le \frac{C_*\cdot \|g\|_{L^\infty}^2}{(2^{n/2}\|z(\gamma')-z(\gamma)\|)^{2(2d+2)}}
\cdot  \|d_\gamma^{\nu_{*}}  u\|_{L^2}^2 \cdot 
2^{-\Lambda+2d\delta \lambda}.
\end{align*}
Exchanging $\gamma$ and $\gamma'$ in the argument above, we  obtain the same estimate for $\|\cQ_{\gamma'\sigma}(u')\cdot \mathbf{1}_{E\setminus Y}\|_{L^2}^2$. Since $|\langle \cQ_{\gamma\sigma}(u), \cQ_{\gamma'\sigma}(u')\rangle_{L^2}|$ is bounded by
\begin{align*}
\|\cQ_{\gamma\sigma}(u)\cdot \mathbf{1}_{Y}\|_{L^2}\cdot \|\cQ_{\gamma'\sigma}(u')\|_{L^2}+ \|\cQ_{\gamma'\sigma}(u')\cdot \mathbf{1}_{E\setminus Y}\|_{L^2}\cdot \|\cQ_{\gamma\sigma}(u)\|_{L^2},
\end{align*}
we conclude  Lemma \ref{lm:key} from the estimates above and Lemma \ref{lm:keypre}.
\end{proof}

\begin{proof}[Proof of Lemma \ref{lm:key} in the case {\rm (ii)}]
Note that we can show the claim of Sublemma \ref{subl:6} in the case (ii) easily {\em if 
we allow the constant $C_*>0$ in it to depend on~$G$}. Thus, following the argument above for the case (iii) and replacing  $2d+2$ by $2d+3$ there, we obtain the estimate
\[
|\langle \cQ_{\gamma\sigma}(u), \cQ_{\gamma'\sigma}(u')\rangle_{L^2}|
\le C(G) \frac{2^{-\Lambda+2d\delta \lambda} \|g\|_{L^\infty}^2\cdot 
\|d_\gamma^{\nu_{*}} u\|_{L^2} \cdot \|d_{\gamma'}^{\nu_{*}} u'\|_{L^2} 
}{\langle 2^{n/2}\|z(\gamma)-z(\gamma')\|\rangle^{2d+3}}
\]
with $C(G)$ a constant which depends on the diffeomorphism $G$. But this implies the lemma because $C(G)/\langle2^{n/2}\|z(\gamma)-z(\gamma')\|\rangle<C(G) 2^{-\tau n}<1$ in the case (ii), provided that we take large $K$ according to $G$. 
\end{proof}

\subsection{The main part of the proof of Lemma \ref{lm:key}}
\label{ss:key}
In this subsection, 
we prove Lemma \ref{lm:key} in the case (iv).
This completes the proof of Proposition \ref{th:m1} and hence that of the main theorem.  

If either $z(\gamma)$ or $z(\gamma')$ is not contained in the image $G(V')$ of $\hat{G}$, we have $d_{\gamma}(y)\ge C(G,g) 2^{n/2}$ and $ 
d_{\gamma'}(y)\ge C(G,g) 2^{n/2}$  for all $y\in \supp \hat{g}$ and hence we can prove the conclusion of Lemma \ref{lm:key} easily, taking large $K$ according to $G$ and $g$.
Therefore we henceforth suppose that $z(\gamma)$ and $z(\gamma')$ are contained in $G(V')$ and let $y(\gamma)$ and $y(\gamma')$  be the unique points in $V'$ such that $\hG(y(\gamma))=z(\gamma)$ and $\hG(y(\gamma'))=z(\gamma')$ respectively.

In order to extract the main part of $\langle \cQ_{\gamma\sigma}(u), \cQ_{\gamma'\sigma}(u')\rangle_{L^2}$, we consider the $C^\infty$ functions  $h,h':E\to [0,1]$ defined by
\begin{align}\label{eqn:defh}
h(y)&= \chi\left(
\frac{20\|\pi_{-}(D\hG_{y(\gamma)}(y-y(\gamma)))\|}
{\|\pi_-(z(\gamma)-z(\gamma'))\|}\right)\cdot \chi(2^{n/3}\cdot \|y-y(\gamma)\|)
\intertext{and}
h'(y)&=h(y-y(\gamma)+y(\gamma'))\\
&= \chi\left(
\frac{20\|\pi_{-}(D\hG_{y(\gamma)}(y-y(\gamma')))\|}
{\|\pi_-(z(\gamma)-z(\gamma'))\|}\right)\cdot \chi(2^{n/3}\cdot \|y-y(\gamma')\|)\notag
\end{align} 
where $\chi$ is the function defined in the beginning of Section \ref{sec:part}.
Since we have
\[
\|y(\gamma)-y(\gamma')\|\le C(G)\cdot \|z(\gamma)-z(\gamma')\|\le C(G)\cdot 2^{(-1/2+\tau)n}
\]
in the case (iv) and the right hand side is much smaller than $2^{-n/3}$ (provided that we take large $K$), the supports of $h$ and $h'$ are contained in the disk with center at $y(\gamma)$ and radius $2^{-n/3+1}$. In particular,  $\hG$ is well approximated by its linearization at  $y(\gamma)$ on that disk up to the error term bounded by  $C(G) (2^{-n/3})^2\ll 2^{-n/2}$.

From the definitions of the function $h$ and $h'$ above, we can show that 
\begin{align*}
&d_\gamma^{-1}\circ \hG(y)\le C_* 2^{-n/2}\|\pi_-(z(\gamma)-z(\gamma'))\|^{-1}\quad \text{for $y\in \supp(1-h)$}
\intertext{and that}
&d_{\gamma'}^{-1}\circ \hG(y)\le C_* 2^{-n/2}\|\pi_-(z(\gamma)-z(\gamma'))\|^{-1}\quad \text{for $y\in \supp(1-h')$.}
\end{align*}
In fact, if $\|y-y(\gamma)\|\ge 2^{-n/3}$, we have $
d_\gamma^{-1}\circ \hG(y)\le  C(G) 2^{-(n/2-n/3)}$
and  obtain the first inequality by a crude estimate,  taking sufficiently large $K$. Otherwise, the condition $y\in \supp(1-h)$ implies that we have
\[
20\|\pi_{-}(D\hG_{y(\gamma)}(y-y(\gamma)))\|\ge \|\pi_-(z(\gamma)-z(\gamma'))\|
\] 
and, hence, we obtain the first inequality again by using the linear approximation of $G$ mentioned above.  We can show the second inequality in the same manner. 

Set $v= {\Psi}_\sigma(D) (h\cdot Q (u))$ and $
v'= {\Psi}_\sigma(D) (h'\cdot Q (u'))$.
Then, in the same manner  as we proved the first inequality  in the proof of Lemma \ref{lm:keypre}, we see that
\begin{align*}
\|\cQ_{\gamma\sigma}(u)-v\|_{L^2}^2&=\|{\Psi}_\sigma(D) ((1-h)\cdot Q (u))\|_{L^2}^2\\
&\le \|g\|_{L^\infty}^2\cdot 
\| d_\gamma^{\nu_{*}} u\|_{L^2}^2  \cdot 
\left\|
|\Fourier^{-1}\Psi_{\sigma}|*\left|
(1-h)\cdot d_\gamma^{-{2\nu_{*}}}\circ \hG \right|\right\|_{L^\infty}.
\end{align*}
From the estimate on $d_\gamma^{-1}\circ \hG$ above, the last factor above is bounded by
\[
C_*\cdot 2^{-\Lambda+2d\delta \lambda}\cdot \langle 2^{n/2}\|\pi_-(z(\gamma)-z(\gamma'))\|\rangle^{-2\nu_*+d+2}.
\]
Hence we obtain
\[
\|\cQ_{\gamma\sigma}(u)-v\|_{L^2}^2\le 
C_* \cdot \frac{2^{-\Lambda+2d\delta \lambda}\cdot \|g\|_{L^\infty}^2\cdot 
\|d_\gamma^{\nu_{*}} u\|_{L^2}^2 
}{\langle 2^{n/2}\|z(\gamma)-z(\gamma')\|\rangle^{2{\nu_{*}}-d-2}}\;\;.
\]
Similarly we obtain the parallel estimate for $\|\cQ_{\gamma'\sigma}(u')-v'\|_{L^2}^2$.
Therefore, by Lemma \ref{lm:keypre} and the choice of $\nu_*$,  we obtain
\begin{align}\label{eqn:Qdif}
|\langle  \cQ_{\gamma\sigma}(u), &\cQ_{\gamma'\sigma}(u')\rangle_{L^2} -\langle v, v'\rangle_{L^2}|\le C_* 
\cdot \frac{2^{-\Lambda+2d\delta \lambda} \|g\|_{L^\infty}^2
\|d_\gamma^{\nu_{*}} u\|_{L^2} \|d_{\gamma'}^{\nu_{*}} u'\|_{L^2} 
}{\langle 2^{n/2}\|z(\gamma)-z(\gamma')\|\rangle^{2d+2}}.
\notag
\end{align} 

Now it is left to show that 
\begin{equation}\label{eqn:Qdif2}
|\langle v, v'\rangle_{L^2}| \le C_* \frac{2^{-\Lambda+2d\delta \lambda} \|g\|_{L^\infty}^2
\|d_\gamma^{\nu_{*}} u\|_{L^2} \|d_{\gamma'}^{\nu_{*}} u'\|_{L^2} 
}{\langle 2^{n/2-2\delta\lambda}\|z(\gamma)-z(\gamma')\|\rangle^{2d+2}}.
\end{equation} 
\begin{remark} The proof of (\ref{eqn:Qdif2}) below is the most essential part of our argument on the central part, where we will use the non-integrability of the contact form $\alpha_0$. Note however that the estimates therein are rather rough. 
\end{remark}

From the assumption $\widetilde{\psi}_{\gamma}(D) u=u$ and $\widetilde{\psi}_{\gamma'}(D) u'=u'$ in Lemma \ref{lm:key}, we may rewrite the functions $v$ and $v'$ as
\[
v= \Fourier\Psi_\sigma* (h\cdot \hat{g}((\widetilde{\psi}_{\gamma}(D)u)\circ \hG))\quad\mbox{and}\quad
v'= \Fourier\Psi_\sigma* (h\cdot \hat{g}((\widetilde{\psi}_{\gamma'}(D)u')\circ \hG))
\]
respectively. 
Hence, setting 
\[
f(y,z,\xi,\xi')=\langle \xi', \hG(y+z) \rangle-\langle\xi, \hG(y) \rangle
\]
and 
\[
\K(z)={\Fourier\Psi_\sigma}* \Fourier\Psi_{\sigma}(z) =\int  \Fourier\Psi_{\sigma}(z') \cdot {\Fourier\Psi_\sigma}(z-z')  dz',
\]
we write $\langle v, v'\rangle_{L^2}$ as
\[
\langle v, v'\rangle_{L^2}=(2\pi)^{-2(2d+1)}\int \K(z) \left(\int S(x,x';z) 
\cdot \overline{u(x)}\cdot  u'(x') dx dx' \right)dz,
\]
where $S(x,x';z)$ denotes  the integral
\begin{align*}
\int 
e^{-i\langle \xi, x\rangle +i\langle \xi',x'\rangle-i f(y,z,\xi,\xi')} \widetilde{\psi}_{\gamma}(\xi) \widetilde{\psi}_{\gamma'}(\xi')
  \hat{g}(y) \hat{g}(y+z)h(y) h'(y+z) dy d\xi d\xi'.
\end{align*}
Notice  that $\K(z)$ is the tensor product of the Dirac $\delta$-function on $E_+$ at the origin and a rapidly decaying function on $E_0\oplus E_-$.

We are going to apply the formula (\ref{eqn:intbypart}) of integration by parts to the integral with respect to the variable $y$ in $S(x,x';z)$ above.
To this end, we set up a unit vector $\unv\in E$ as follows.  
Recall that we have 
\[
d\alpha_0=2\cdot dx^-\wedge dx^+=2\sum_{i=1}^d dx^-_i\wedge dx^+_i.
\] 
We define $\unv$ as the unique  unit vector such that $D\hG_{y(\gamma)}(\unv)\in E_0\oplus E_+$, that 
\begin{align*}
&\langle \widetilde{\alpha}_0(y(\gamma)), \unv\rangle =\langle {\alpha}_0(z(\gamma)), D\hG_{y(\gamma)}(\unv)\rangle=0
\end{align*}
and that
\begin{align*}
d\alpha_0(D\hG_{y(\gamma)}({\unv}), 
\pi_{-}&(z(\gamma')-z(\gamma)))\\
&=2  \|\pi_{-}(z(\gamma')-z(\gamma))\|\|\pi_{+}(D\hG_{y(\gamma)}({\unv}))\|.
\end{align*}
The next sublemma tells that the term $e^{-i f(y,z,\xi,\xi')}$ in $S(x,x';z)$ as a function of $y$ oscillates very  fast in the direction of $w$. 
\begin{sublemma}\label{subl:7}
If  $y+z\in \supp h'$ for $y\in \supp h$ and  $z\in E_0\oplus E_-$, and if  $\xi \in \supp \widetilde{\psi}_\gamma$ and $\xi'\in \supp \widetilde{\psi}_{\gamma'}$, we have
\[
|D_w f(y,z,\xi,\xi')|
\ge 2^{n-10}\cdot\|\pi_-(z(\gamma')-z(\gamma))\|\cdot
  \|\pi_+(D\hG_{y(\gamma)}({\unv})))\|,
\]
where $D_w$ denotes the directional derivative along the unit vector $w$ with respect to the variable $y$, that is, 
\[
D_w f(y,z,\xi,\xi')=\langle \xi', D\hG_{y+z}({\unv})\rangle -\langle \xi, D\hG_{y}({\unv})\rangle.
\]

\end{sublemma}
We postpone the proof of this sublemma for a while. 
Under the same  assumption as  in the sublemma above, we have 
\begin{align}\label{eqn:dvf2}
|D_{\unv}^k f(y,z,\xi,\xi')|& \le C(G)\cdot \max\{\|\xi\|, \|\xi'\|\}\cdot \|z\|\\&\le C(G)\cdot 2^{n}\cdot \|z(\gamma')-z(\gamma)\|\notag
\end{align}
for $k=1,2$, where we used the estimate
\[
\|z\|\le C_* \|D\hG_{y(\gamma)}(y+z) -D\hG_{y(\gamma)}(y)\| \le C_*\|z(\gamma)-z(\gamma)\|
\]
that follows from the hyperbolic property of $D\hG_{y(\gamma)}$ (or that of $G$) and the definitions of $h$ and $h'$.  Also we have 
\begin{equation}\label{eqn:h}
\|D_{\unv} h\|_{L^\infty}\le C_* 2^{n/3}\quad\text{and}\quad \|D_{\unv} h'\|_{L^\infty} \le C_* 2^{n/3}
\end{equation}
from the condition $D\hG_{y(\gamma)}(w)\in E_0\oplus E_+$ in the choice of $w$. 

Now we apply the formula (\ref{eqn:intbypart}) of integration by parts along the single vector ${\unv}$ once to the integration with respect to $y$ in the integral $S(x,x';z)$. 
Then the result should be written in the form 
\begin{equation}\label{eqn:iexf}
\int 
e^{-i\langle \xi, x\rangle +i\langle \xi',x'\rangle-i f(y,z,\xi,\xi')}
 \widetilde{\psi}_{\gamma}(\xi) \widetilde{\psi}_{\gamma}(\xi')  R(y,z;x,x'; \xi, \xi') dy d\xi d\xi'.
\end{equation}
By using Sublemma \ref{subl:7}, (\ref{eqn:dvf2}) and (\ref{eqn:h}), we see that there exists a constant $C_{\alpha, \beta}(G,g)$ for each multi-indices $\alpha$ and $\beta$, which may depend on $G$, $g$ and $\lambda$, such that 
\[
\|\partial^\alpha_{\xi} \partial^\beta_{\xi'} R\|_{L^\infty} \le 
\frac{C_{\alpha, \beta}(G,g)\cdot  2^{-
(|\alpha| + |\beta|) n/2}}
{2^{n} \|z(\gamma)-z(\gamma')\|}.
\]
This implies that we have
\[
\left\|\partial^\alpha_{\xi} \partial^\beta_{\xi'} (\widetilde{\psi}_{\gamma}(\xi) \widetilde{\psi}_{\gamma}(\xi')  R(y,z;x,x'; \xi, \xi'))\right\|_{L^\infty} \le 
\frac{C_{\alpha, \beta}(G,g)\cdot  2^{-
(|\alpha| + |\beta|) n/2}}
{2^{n} \|z(\gamma)-z(\gamma')\|}.
\]
Therefore, performing integration with respect to the variables $\xi$ and $\xi'$ in (\ref{eqn:iexf}) and recalling the argument in the proof of Lemma \ref{lm:kest1},  we obtain the estimate   
\[
|S(x,x';z)|\le C(G,g)  \int 
\frac{|\K(z)|   b^{2d+2}_{n,0}(\hG(y)-x)   b^{2d+2}_{n,0}(\hG(y+z)-x')}
{ 2^{n} \|z(\gamma)-z(\gamma')\|}
dz dy.
\]
By this estimate and Young inequality, we obtain 
\[
|\langle v, v'\rangle_{L^2}|\le  C(G,g) \cdot \frac{
\|d_\gamma^{\nu_*} u\|_{L^2} \|d_{\gamma'}^{\nu_*} u'\|_{L^2} 
}{2^{n}\|z(\gamma)-z(\gamma')\|}.
\]
This implies (\ref{eqn:Qdif2}), since we have $(2d+2)\tau <1/2$ from the choice of $\tau$ and 
\[
(2^{n/2}\|z(\gamma)-z(\gamma')\|)^{2d+2}\le 
2^{(2d+2)\cdot \tau n}\le 
2^{((2d+2)\cdot \tau-1/2) n}\cdot 2^{n}\|z(\gamma)-z(\gamma')\|.
\]
(Recall that  $n\ge K$ and that we may take large $K$ depending on $G$ and $g$.)

Finally we complete the proof by proving Sublemma \ref{subl:7}.
\begin{proof}[Proof of Sublemma \ref{subl:7}] 
Recall that the supports of $h$ and $h'$ are contained in the disk with center at $y(\gamma)$ and radius $2^{-n/3+1}$ and that  $\hG$ is well approximated by its linearization at $y(\gamma)$ on that disk. From the assumption that $y$ and $y+z$ belong to $\supp h$ and $\supp h'$ respectively, we see that 
\[
\|\pi_-(D\hG_{y(\gamma)}(z)-z(\gamma')-z(\gamma))\|< \|\pi_-(z(\gamma')-z(\gamma))\|/4.
\]
From the choice of the vector $w$, we see that
\begin{align*}
|\langle  \alpha_0(\hG(y+z)), & D\hG_{y+z}({\unv})\rangle -\langle \alpha_0(\hG(y)), D\hG_{y}({\unv})\rangle| \\
&=|\langle  \talpha_0(y+z),  {\unv}\rangle -\langle \talpha_0(y), {\unv})\rangle|= |d\talpha_0(z,w)|
 \\
& \ge |d\alpha_0(\pi_-(z(\gamma')-z(\gamma)), D\hG_{y(\gamma)}({\unv}))|/2\\
&=  \|\pi_-(z(\gamma')-z(\gamma))\| \|\pi_+(D\hG_{y(\gamma)}({\unv}))\|.
\end{align*}
Since  $n(\gamma)=n(\gamma')=n$ and  
\[\Delta(n(\gamma), k(\gamma), n(\sigma), k(\sigma))=\Delta(n(\gamma'), k(\gamma'), n(\sigma), k(\sigma))=0
\]
from the definition of $S$, we have
\[
2^{n-2}\le |\xi_0|\le 2^{n+2},\quad 2^{n-2}\le |\xi'_0|\le 2^{n+2}\quad \mbox{and}\quad
|\xi_0-\xi'_0|\le 2^{n/2+5}
\]
for $\xi_0=\pi^*_0(\xi)$ and $\xi'_0=\pi^*_0(\xi')$.
Therefore the lemma follows if we show 
\begin{align*}
\langle \xi_0\cdot \alpha_0(\hG(y+z))-\xi',& D\hG_{y+z}({\unv})\rangle \le |\xi_0|   \|\pi_-(z(\gamma)-z(\gamma'))\| 
\|D\hG_{y(\gamma)}({\unv})\|/3
\end{align*}
and 
\[
\langle \xi_0\cdot  \alpha_0(\hG(y))-\xi, D\hG_{y}({\unv})\rangle \le |\xi_0|  \|\pi_-(z(\gamma)-z(\gamma'))\| 
\|D\hG_{y(\gamma)}({\unv})\|/3.
\]
These can be proved by a straightforward estimate. Below we prove the former inequality. The latter can be proved similarly.

Since we have $y+z\in \supp h'$, $\xi'\in \supp \widetilde{\psi}_{\gamma'}$ and $|m(\gamma')|\le \delta \lambda$, it holds
\begin{align*}
&\|\pi_{+,0}(D\hG_{y+z}({\unv}))\|\le  2 \|D\hG_{y(\gamma)}({\unv})\|, \quad \text{and }\\
&\|\pi_{+,0}^*(\xi_0 \cdot \alpha_0(\hG(y+z))-\xi')\|\\
&\qquad\le 
|\xi_0|\cdot  \|\pi_+^*(\alpha_0(\hG(y+z))-\alpha_0(z(\gamma')))\|+|\xi_0-\xi'_0| \cdot \|\pi_+^*(\alpha_0(z(\gamma')))\|\\
&\qquad\qquad  +
\|\pi_+^*(\xi'_0 \cdot \alpha_0(z(\gamma')))-\xi')\|\\
&\qquad \le  
|\xi_0|\|\pi_-(\hG(y+z)-\hG(z(\gamma')))\|+2^{n/2+6}+
2^{n/2+\delta \lambda+5}\\
&\qquad \le |\xi_0|  \|\pi_-(z(\gamma)-z(\gamma'))\|/10
\end{align*}
where, in the last inequality,  we used the facts that $\delta\lambda\ge \delta \lambda_*\ge 10$ and that 
\[
\|\pi_{-}(z(\gamma)-z(\gamma'))\|\ge \|z(\gamma)-z(\gamma')\|/10>2^{-n/2+2\delta \lambda-4}.
\]
By a rough estimate using the condition $D\hG_{y(\gamma)}(w)\in E_0\oplus E_+$ in the choice of $w$, we see that
\begin{align*}
&\|\pi_{-}(D\hG_{y+z}({\unv}))\|
\le C(G) \|(y+z)-y(\gamma)\|  \le  C(G) 2^{-n/3}\quad \mbox{and also} \\
&\|\pi^*_{-}(\xi_0 \cdot \alpha_0(\hG(y+z))-\xi')\|\le C(G)
 2^{(2/3)n}.
\end{align*}
Clearly these inequalities yield the required estimate.
\end{proof}

\appendix

\section{Proof of Lemma \ref{lm:wd}}\label{apd1}
Let $p_n(\xi)=\chi_n(|\xi|)$ and $\tilde{p}_n(\xi)=\widetilde{\chi}_n(|\xi|)$ for $n\ge 0$, where $\chi_n$ and $\widetilde{\chi}_n$ are those defined in Subsection \ref{ss:PU}. For $u\in C^\infty(\disk)$, we define $u_\gamma=p_\gamma(x,D)^* u$ for $\gamma\in \Gamma$ and $u_n=p_n(D) u$ for $n\ge 0$. We may and do suppose that the norm on the Sobolev space $W^s$ is defined by 
\[
\|u\|_{W^s}^2:=\sum_{n\ge 0} 2^{2sn} \|u_n\|^2_{L^2}.
\] 
Set $\widetilde{n}(\gamma)=\max\{n(\gamma), m(\gamma)+(n(\gamma)/2)\}$ for $\gamma\in \Gamma$. Then there exists a constant $c>0$  such that 
if $|\widetilde{n}(\gamma)-n|>c$, we have
\[
d(\mathrm{supp}(\psi_\gamma), \mathrm{supp}(\tilde{p}_n))> 2^{\max\{n,\widetilde{n}(\gamma)\}-c}.
\]

We first prove $W^s(\disk)\subset \cB^\beta_\nu$ for $s>\beta$ and $\nu\ge 2d+2$ by showing  $\|u\|_{\beta,\nu}\le C \|u\|_{W^s}$ for $u\in C^\infty(\disk)$. 
For each $n\ge 0$, we have
\begin{align*}
\sum_{\gamma:\widetilde{n}(\gamma)=n}&\|d_\gamma^\nu\cdot  u_\gamma\|_{L^2}^2=
\sum_{\gamma:\widetilde{n}(\gamma)=n}\left\|d_\gamma^\nu\cdot  \sum_{n'=0}^\infty p_\gamma(x,D)^* \tilde{p}_{n'}(D)u_n \right\|_{L^2}^2.
\end{align*}
We regard the operator $u\mapsto d_{\gamma}\cdot p_\gamma(x,D)^*   p_{n'}(D)u$ as an integral operator with the kernel  
\begin{align*}&\kappa_{n',\gamma}(x,x')=
\frac{1}{(2\pi)^{2(2d+1)}}\int   d_{\gamma}(x') e^{i\langle \xi, x'-y\rangle +i\langle \eta, y-x\rangle }
\rho_{\gamma}(y) \psi_\gamma(\xi) p_{n'}(\eta) dy d\xi d\eta.
\end{align*}
Fix some $\mu>\max\{2d+2, s\}$. Similarly to 
Lemma \ref{lm:kest0}, we have
\begin{align*}
|\kappa_{n',\gamma}(x,x')|&\le C \int_{\Ze(\gamma)}d_\gamma^\nu(x') b_{\gamma}^{\mu+\nu}(x'-y) b_{n,0}^{\mu}(y-x) dy\\
&\le C\int_{\Ze(\gamma)} b_{\gamma}^{\mu}(x'-y) b_{n,0}^{\mu}(y-x) dy.
\end{align*}
Further, in the case $|n'-\widetilde{n}(\gamma)|> c$, 
we can show
\[
|\kappa_{\gamma,n'}(x,x')|\le C 2^{-\mu \max\{n', \widetilde{n}(\gamma)\}/2}\int_{\Ze(\gamma)} b_{\gamma}^{\mu}(x'-y) b_{n,0}^\mu(y-x) dy,
\]
applying the formula (\ref{eqn:intbypart}) of integration by parts along a set of vectors $\{v_j\}_{j=0}^{2d}$ that form an orthogonal basis of $E$ for $\mu$ times to the integral with respect $y$ in $\kappa_{n',\gamma}(x,x')$. 
Therefore we obtain, using Young inequality, that 
\begin{align*}
\sum_{\gamma:\widetilde{n}(\gamma)=n}\|d_\gamma^\nu  u_\gamma\|_{L^2}^2\le &C n^2 \sum_{n':|n'-n|\le c} \|u_{n'}\|_{L^2}^2\\
&\quad +
 C n^2 \sum_{n':|n'-n|> c}2^{-\mu \max\{n', n\}/2 }\|u_{n'}\|_{L^2}^2.
\end{align*}
Take weighted sum of the both sides with respect to $n$ with weight $2^{\beta n}$. Then the weighted sum of the left hand side is not smaller than $\|u\|_{\beta,\nu}$ and that of the right hand side is bounded by $C\|u\|_{W^s}$. Thus we conclude $\|u\|_{\beta,\nu}\le C \|u\|_{W^s}$.

We next prove $\cB^\beta_\nu\subset W^{-s}(\disk)$  for $s>\beta$ and $\nu\ge 2d+2$ by showing $\|u\|_{W^{-s}}\le C \|u\|_{\beta,\nu}$. 
We have
\begin{align*}
\| u_n\|_{L^2}^2 &=\left\|p_n(D) 
\left(\sum_{|\widetilde{n}(\gamma)-n|<c} u_\gamma\right)\right\|_{L^2}^2\le C\left\|\sum_{|\widetilde{n}(\gamma)-n|<c} u_\gamma\right\|_{L^2}^2
\end{align*}
Since we have 
\[
|(u_\gamma, u_{\gamma'})|\le C_*\langle 2^{n/2}(z(\gamma)-z(\gamma')) \rangle^{-2\nu} \|(d_\gamma)^{\nu} u_\gamma\|_{L^2}
\|(d_{\gamma'})^{\nu} u_{\gamma'}\|_{L^2}
\]
for any pair $(\gamma, \gamma')\in \Gamma\times \Gamma$ and since  the left hand side above vanishes if the supports of $\psi_\gamma$ and $\psi_{\gamma'}$ does not meet, we obtain
\[
\left\|\sum_{|\widetilde{n}(\gamma)-n|<c} u_\gamma\right\|_{L^2}^2\le C
\sum_{|\widetilde{n}(\gamma)-n|<c} \left\| u_\gamma\right\|_{L^2}^2
\]
Take weighted sum of the both sides with respect to $n$ with weight $2^{-s n}$. Then the weighted sum of the left hand side is not smaller than $C^{-1} \|u\|_{\beta,\nu}$ and that of the right hand side is bounded by  $C\|u\|_{W^{-s}}$, provided $s>\beta$. Thus we conclude $\|u\|_{W^{-s}}\le C \|u\|_{\beta,\nu}$.
\bibliographystyle{amsplain}
\bibliography{mybib}

\providecommand{\bysame}{\leavevmode\hbox to3em{\hrulefill}\thinspace}
\providecommand{\MR}{\relax\ifhmode\unskip\space\fi MR }
\providecommand{\MRhref}[2]{%
  \href{http://www.ams.org/mathscinet-getitem?mr=#1}{#2}
}
\providecommand{\href}[2]{#2}
\begin{thebibliography}{10}

\bibitem{Aebischer}
B.~Aebischer, M.~Borer, M.~K{\"a}lin, Ch. Leuenberger, and H.~M. Reimann,
  \emph{Symplectic geometry}, Progress in Mathematics, vol. 124, Birkh\"auser
  Verlag, Basel, 1994. \MR{MR1296462 (96a:58082)}

\bibitem{Anantharaman00}
Nalini Anantharaman, \emph{Precise counting results for closed orbits of
  {A}nosov flows}, Ann. Sci. \'Ecole Norm. Sup. (4) \textbf{33} (2000), no.~1,
  33--56. \MR{MR1743718 (2002c:37048)}

\bibitem{Anosov69}
D.~V. Anosov, \emph{Geodesic flows on closed {R}iemann manifolds with negative
  curvature.}, Proceedings of the Steklov Institute of Mathematics, No. 90
  (1967). Translated from the Russian by S. Feder, American Mathematical
  Society, Providence, R.I., 1969. \MR{MR0242194 (39 \#3527)}

\bibitem{Baladi05}
Viviane Baladi, \emph{Anisotropic {S}obolev spaces and dynamical transfer
  operators: {$C\sp \infty$} foliations}, Algebraic and topological dynamics,
  Contemp. Math., vol. 385, Amer. Math. Soc., Providence, RI, 2005,
  pp.~123--135. \MR{MR2180233 (2007c:37022)}

\bibitem{BaladiTsujii07}
Viviane Baladi and Masato Tsujii, \emph{Anisotropic {H}\"older and {S}obolev
  spaces for hyperbolic diffeomorphisms}, Ann. Inst. Fourier (Grenoble)
  \textbf{57} (2007), no.~1, 127--154. \MR{MR2313087}

\bibitem{BaladiTsujii08}
\bysame, \emph{Dynamical determinants and spectrum for hyperbolic
  diffeomorphisms}, Geometric and probabilistic structures in dynamics,
  Contemp. Math., vol. 469, Amer. Math. Soc., Providence, RI, 2008, pp.~29--68.
  \MR{MR2478465}

\bibitem{BaladiVallee05}
Viviane Baladi and Brigitte Vall{\'e}e, \emph{Exponential decay of correlations
  for surface semi-flows without finite {M}arkov partitions}, Proc. Amer. Math.
  Soc. \textbf{133} (2005), no.~3, 865--874 (electronic). \MR{MR2113938
  (2006d:37047)}

\bibitem{Blank02}
Michael Blank, Gerhard Keller, and Carlangelo Liverani,
  \emph{Ruelle-{P}erron-{F}robenius spectrum for {A}nosov maps}, Nonlinearity
  \textbf{15} (2002), no.~6, 1905--1973. \MR{MR1938476 (2003m:37033)}

\bibitem{Bowen75}
Rufus Bowen, \emph{Equilibrium states and the ergodic theory of {A}nosov
  diffeomorphisms}, Springer-Verlag, Berlin, 1975, Lecture Notes in
  Mathematics, Vol. 470. \MR{MR0442989 (56 \#1364)}

\bibitem{Chernov98}
N.~I. Chernov, \emph{Markov approximations and decay of correlations for
  {A}nosov flows}, Ann. of Math. (2) \textbf{147} (1998), no.~2, 269--324.
  \MR{MR1626741 (99d:58101)}

\bibitem{ColletET84}
P.~Collet, H.~Epstein, and G.~Gallavotti, \emph{Perturbations of geodesic flows
  on surfaces of constant negative curvature and their mixing properties},
  Comm. Math. Phys. \textbf{95} (1984), no.~1, 61--112. \MR{MR757055
  (85m:58143)}

\bibitem{Dolgopyat98}
Dmitry Dolgopyat, \emph{On decay of correlations in {A}nosov flows}, Ann. of
  Math. (2) \textbf{147} (1998), no.~2, 357--390. \MR{MR1626749 (99g:58073)}

\bibitem{Dolgopyat98b}
\bysame, \emph{Prevalence of rapid mixing in hyperbolic flows}, Ergodic Theory
  Dynam. Systems \textbf{18} (1998), no.~5, 1097--1114. \MR{MR1653299
  (2000a:37014)}

\bibitem{Dolgopyat02}
\bysame, \emph{On mixing properties of compact group extensions of hyperbolic
  systems}, Israel J. Math. \textbf{130} (2002), 157--205. \MR{MR1919377
  (2003m:37037)}

\bibitem{FMNT05}
Michael Field, Ian Melbourne, Matthew Nicol, and Andrei T{\"o}r{\"o}k,
  \emph{Statistical properties of compact group extensions of hyperbolic flows
  and their time one maps}, Discrete Contin. Dyn. Syst. \textbf{12} (2005),
  no.~1, 79--96. \MR{MR2121250 (2006e:37045)}

\bibitem{Hopf39}
Eberhard Hopf, \emph{Statistik der geod\"atischen {L}inien in
  {M}annigfaltigkeiten negativer {K}r\"ummung}, Ber. Verh. S\"achs. Akad. Wiss.
  Leipzig \textbf{91} (1939), 261--304. \MR{MR0001464 (1,243a)}

\bibitem{Hormander3}
Lars H{\"o}rmander, \emph{The analysis of linear partial differential
  operators. {I}}, Classics in Mathematics, Springer-Verlag, Berlin, 2003.
  \MR{MR1996773}

\bibitem{Hormander1}
\bysame, \emph{The analysis of linear partial differential operators. {III}},
  Classics in Mathematics, Springer, Berlin, 2007, Pseudo-differential
  operators, Reprint of the 1994 edition. \MR{MR2304165 (2007k:35006)}

\bibitem{Iwata07}
Yukiko Iwata, \emph{A generalized local limit theorem for mixing semi-flows},
  Hokkaido Mathematical Journal \textbf{37} (2008), no.~1, 215--240.

\bibitem{Katok94}
Anatole Katok, \emph{Infinitesimal {L}yapunov functions, invariant cone
  families and stochastic properties of smooth dynamical systems}, Ergodic
  Theory Dynam. Systems \textbf{14} (1994), no.~4, 757--785, With the
  collaboration of Keith Burns. \MR{MR1304141 (95j:58097)}

\bibitem{Liverani04}
Carlangelo Liverani, \emph{On contact {A}nosov flows}, Ann. of Math. (2)
  \textbf{159} (2004), no.~3, 1275--1312. \MR{MR2113022 (2005k:37048)}

\bibitem{McKean72}
H.~P. McKean, \emph{Selberg's trace formula as applied to a compact {R}iemann
  surface}, Comm. Pure Appl. Math. \textbf{25} (1972), 225--246. \MR{MR0473166
  (57 \#12843a)}

\bibitem{MelbourneTorok02}
Ian Melbourne and Andrei T{\"o}r{\"o}k, \emph{Central limit theorems and
  invariance principles for time-one maps of hyperbolic flows}, Comm. Math.
  Phys. \textbf{229} (2002), no.~1, 57--71. \MR{MR1917674 (2003k:37012)}

\bibitem{Moore84}
Calvin~C. Moore, \emph{Exponential decay of correlation coefficients for
  geodesic flows}, Group representations, ergodic theory, operator algebras,
  and mathematical physics (Berkeley, Calif., 1984), Math. Sci. Res. Inst.
  Publ., vol.~6, Springer, New York, 1987, pp.~163--181. \MR{MR880376
  (89d:58102)}

\bibitem{Nussbaum70}
Roger~D. Nussbaum, \emph{The radius of the essential spectrum}, Duke Math. J.
  \textbf{37} (1970), 473--478. \MR{MR0264434 (41 \#9028)}

\bibitem{Pollicott92}
Mark Pollicott, \emph{Exponential mixing for the geodesic flow on hyperbolic
  three-manifolds}, J. Statist. Phys. \textbf{67} (1992), no.~3-4, 667--673.
  \MR{MR1171148 (93i:58119)}

\bibitem{Pollicott99}
\bysame, \emph{On the mixing of {A}xiom {A} attracting flows and a conjecture
  of {R}uelle}, Ergodic Theory Dynam. Systems \textbf{19} (1999), no.~2,
  535--548. \MR{MR1685406 (2001d:37038)}

\bibitem{PollicottSharp01}
Mark Pollicott and Richard Sharp, \emph{Asymptotic expansions for closed orbits
  in homology classes}, Geom. Dedicata \textbf{87} (2001), no.~1-3, 123--160.
  \MR{MR1866845 (2003b:37051)}

\bibitem{Ratner87}
Marina Ratner, \emph{The rate of mixing for geodesic and horocycle flows},
  Ergodic Theory Dynam. Systems \textbf{7} (1987), no.~2, 267--288.
  \MR{MR896798 (88j:58103)}

\bibitem{Stoyanov01}
Luchezar Stoyanov, \emph{Spectrum of the {R}uelle operator and exponential
  decay of correlations for open billiard flows}, Amer. J. Math. \textbf{123}
  (2001), no.~4, 715--759. \MR{MR1844576 (2002f:37061)}

\bibitem{Stoyanov07}
\bysame, \emph{Spectra of {R}uelle transfer operators for {A}xiom {A} flows on
  basic sets}, Preprint (2007).

\bibitem{Taylor}
Michael~E. Taylor, \emph{Pseudodifferential operators and nonlinear {PDE}},
  Progress in Mathematics, vol. 100, Birkh\"auser Boston Inc., Boston, MA,
  1991. \MR{MR1121019 (92j:35193)}

\bibitem{Tsujii2008}
Masato Tsujii, \emph{Decay of correlations in suspension semi-flows of
  angle-multiplying maps}, Ergodic Theory Dynam. Systems \textbf{28} (2008),
  291--317.

\end{thebibliography}


\end{document}